\documentclass[12pt]{amsart}

\usepackage{amsmath}
\usepackage{amssymb}
\usepackage{amscd}
\usepackage{color}
\usepackage{hyperref}
\usepackage{mathrsfs}
\usepackage{upgreek}
\usepackage{eucal}
\usepackage{mathtools}
\usepackage{enumitem}

\usepackage{float}
\restylefloat{table}

\usepackage{xy}
\xyoption{all}

\newenvironment{renumerate}{\begin{enumerate}[label={\textup{(\roman*)}}]}{\end{enumerate}}
\newenvironment{aenumerate}{\begin{enumerate}[label={\textup{(\alph*)}}]}{\end{enumerate}}

\newcommand{\nc}{\newcommand}

\nc{\CC}{{\mathbb{C}}}
\nc{\LL}{{\mathbb{L}}}
\nc{\RR}{{\mathbb{R}}}
\renewcommand{\P}{{\mathbb{P}}}
\nc{\OO}{{\mathbb{O}}}

\nc{\QQ}{{\mathbb{Q}}}
\nc{\ZZ}{{\mathbb{Z}}}

\nc{\cA}{{\mathcal{A}}}
\nc{\cB}{{\mathcal{B}}}
\nc{\cC}{{\mathcal{C}}}
\nc{\cD}{{\mathcal{D}}}
\nc{\cE}{{\mathcal{E}}}
\nc{\tcE}{{\tilde{\mathcal{E}}}}
\nc{\cF}{{\mathcal{F}}}
\nc{\cG}{{\mathcal{G}}}
\nc{\cH}{{\mathcal{H}}}
\nc{\cI}{{\mathcal{I}}}
\nc{\cJ}{{\mathcal{J}}}
\nc{\cK}{{\mathcal{K}}}
\nc{\cL}{{\mathcal{L}}}
\nc{\cM}{{\mathcal{M}}}
\nc{\cN}{{\mathcal{N}}}
\nc{\cO}{{\mathcal{O}}}
\nc{\cP}{{\mathcal{P}}}
\nc{\cQ}{{\mathcal{Q}}}
\nc{\cR}{{\mathcal{R}}}
\nc{\cS}{{\mathcal{S}}}
\nc{\cT}{{\mathcal{T}}}
\nc{\cU}{{\mathcal{U}}}
\nc{\cV}{{\mathcal{V}}}
\nc{\cW}{{\mathcal{W}}}
\nc{\cX}{{\mathcal{X}}}
\nc{\cY}{{\mathcal{Y}}}
\nc{\cZ}{{\mathcal{Z}}}

\nc{\rc}{{\mathrm{c}}}
\nc{\rf}{{\mathsf{f}}}
\nc{\rch}{{\mathrm{ch}}}
\nc{\rtd}{{\mathrm{td}}}

\nc{\rB}{{\mathrm{B}}}
\nc{\rF}{{\mathrm{F}}}
\nc{\rG}{{\mathrm{G}}}
\nc{\rH}{{\mathrm{H}}}
\nc{\rK}{{\mathrm{K}}}
\nc{\rM}{{\mathrm{M}}}
\nc{\rP}{{\mathrm{P}}}
\nc{\rR}{{\mathrm{R}}}
\nc{\rS}{{\mathrm{S}}}
\nc{\rT}{{\mathrm{T}}}
\nc{\rX}{{\mathrm{X}}}

\nc{\rQ}{{\mathrm{Q}}}

\nc{\bA}{{\mathbf{A}}}
\nc{\bB}{{\mathbf{B}}}
\nc{\bC}{{\mathbf{C}}}
\nc{\bD}{{\mathbf{D}}}
\nc{\bE}{{\mathbf{E}}}
\nc{\bF}{{\mathbf{F}}}
\nc{\bG}{{\mathbf{G}}}
\nc{\bH}{{\mathbf{H}}}
\nc{\bI}{{\mathbf{I}}}
\nc{\bJ}{{\mathbf{J}}}
\nc{\bK}{{\mathbf{K}}}
\nc{\bL}{{\mathbf{L}}}
\nc{\bM}{{\mathbf{M}}}
\nc{\bN}{{\mathbf{N}}}
\nc{\bO}{{\mathbf{O}}}
\nc{\bP}{{\mathbf{P}}}
\nc{\bQ}{{\mathbf{Q}}}
\nc{\bR}{{\mathbf{R}}}
\nc{\bS}{{\mathbf{S}}}
\nc{\bT}{{\mathbf{T}}}
\nc{\bU}{{\mathbf{U}}}
\nc{\bV}{{\mathbf{V}}}
\nc{\bW}{{\mathbf{W}}}
\nc{\bX}{{\mathbf{X}}}
\nc{\bY}{{\mathbf{Y}}}
\nc{\bZ}{{\mathbf{Z}}}

\nc{\ba}{{\mathbf{a}}}
\nc{\bb}{{\mathbf{b}}}
\nc{\bc}{{\mathbf{c}}}
\nc{\bd}{{\mathbf{d}}}
\nc{\be}{{\mathbf{e}}}
\nc{\bg}{{\mathbf{g}}}
\nc{\bh}{{\mathbf{h}}}
\nc{\bi}{{\mathbf{i}}}
\nc{\bj}{{\mathbf{j}}}
\nc{\bk}{{\mathbf{k}}}
\nc{\bl}{{\mathbf{l}}}
\nc{\bm}{{\mathbf{m}}}
\nc{\bn}{{\mathbf{n}}}
\nc{\bo}{{\mathbf{o}}}
\nc{\bp}{{\mathbf{p}}}
\nc{\bq}{{\mathbf{q}}}
\nc{\br}{{\mathbf{r}}}
\nc{\bs}{{\mathbf{s}}}
\nc{\bt}{{\mathbf{t}}}
\nc{\bu}{{\mathbf{u}}}
\nc{\bv}{{\mathbf{v}}}
\nc{\bw}{{\mathbf{w}}}
\nc{\bx}{{\mathbf{x}}}
\nc{\by}{{\mathbf{y}}}
\nc{\bz}{{\mathbf{z}}}

\nc{\fA}{{\mathfrak{A}}}
\nc{\fB}{{\mathfrak{B}}}
\nc{\fC}{{\mathfrak{C}}}
\nc{\fD}{{\mathfrak{D}}}
\nc{\fE}{{\mathfrak{E}}}
\nc{\fF}{{\mathfrak{F}}}
\nc{\fG}{{\mathfrak{G}}}
\nc{\fH}{{\mathfrak{H}}}
\nc{\fI}{{\mathfrak{I}}}
\nc{\fJ}{{\mathfrak{J}}}
\nc{\fK}{{\mathfrak{K}}}
\nc{\fL}{{\mathfrak{L}}}
\nc{\fM}{{\mathfrak{M}}}
\nc{\fN}{{\mathfrak{N}}}
\nc{\fO}{{\mathfrak{O}}}
\nc{\fP}{{\mathfrak{P}}}
\nc{\fQ}{{\mathfrak{Q}}}
\nc{\fR}{{\mathfrak{R}}}
\nc{\fS}{{\mathfrak{S}}}
\nc{\fT}{{\mathfrak{T}}}
\nc{\fU}{{\mathfrak{U}}}
\nc{\fV}{{\mathfrak{V}}}
\nc{\fW}{{\mathfrak{W}}}
\nc{\fX}{{\mathfrak{X}}}
\nc{\fY}{{\mathfrak{Y}}}
\nc{\fZ}{{\mathfrak{Z}}}

\nc{\fa}{{\mathfrak{a}}}
\nc{\fb}{{\mathfrak{b}}}
\nc{\fc}{{\mathfrak{c}}}
\nc{\fd}{{\mathfrak{d}}}
\nc{\fe}{{\mathfrak{e}}}
\nc{\ff}{{\mathfrak{f}}}
\nc{\fg}{{\mathfrak{g}}}
\nc{\fh}{{\mathfrak{h}}}
\nc{\fj}{{\mathfrak{j}}}
\nc{\fk}{{\mathfrak{k}}}
\nc{\fl}{{\mathfrak{l}}}
\nc{\fm}{{\mathfrak{m}}}
\nc{\fn}{{\mathfrak{n}}}
\nc{\fo}{{\mathfrak{o}}}
\nc{\fp}{{\mathfrak{p}}}
\nc{\fq}{{\mathfrak{q}}}
\nc{\fr}{{\mathfrak{r}}}
\nc{\fs}{{\mathfrak{s}}}
\nc{\ft}{{\mathfrak{t}}}
\nc{\fu}{{\mathfrak{u}}}
\nc{\fv}{{\mathfrak{v}}}
\nc{\fw}{{\mathfrak{w}}}
\nc{\fx}{{\mathfrak{x}}}
\nc{\fy}{{\mathfrak{y}}}
\nc{\fz}{{\mathfrak{z}}}

\nc{\sA}{{\mathsf{A}}}
\nc{\sB}{{\mathsf{B}}}
\nc{\sC}{{\mathsf{C}}}
\nc{\sD}{{\mathsf{D}}}
\nc{\sE}{{\mathsf{E}}}
\nc{\sF}{{\mathsf{F}}}
\nc{\sG}{{\mathsf{G}}}
\nc{\sH}{{\mathsf{H}}}
\nc{\sI}{{\mathsf{I}}}
\nc{\sJ}{{\mathsf{J}}}
\nc{\sK}{{\mathsf{K}}}
\nc{\sL}{{\mathsf{L}}}
\nc{\sM}{{\mathsf{M}}}
\nc{\sN}{{\mathsf{N}}}
\nc{\sO}{{\mathsf{O}}}
\nc{\sP}{{\mathsf{P}}}
\nc{\sQ}{{\mathsf{Q}}}
\nc{\sR}{{\mathsf{R}}}
\nc{\sS}{{\mathsf{S}}}
\nc{\sT}{{\mathsf{T}}}
\nc{\sU}{{\mathsf{U}}}
\nc{\sV}{{\mathsf{V}}}
\nc{\sW}{{\mathsf{W}}}
\nc{\sX}{{\mathsf{X}}}
\nc{\sY}{{\mathsf{Y}}}
\nc{\sZ}{{\mathsf{Z}}}

\nc{\sa}{{\mathsf{a}}}
\nc{\sd}{{\mathsf{d}}}
\nc{\se}{{\mathsf{e}}}
\nc{\sg}{{\mathsf{g}}}
\nc{\sh}{{\mathsf{h}}}
\nc{\si}{{\mathsf{i}}}
\nc{\sj}{{\mathsf{j}}}
\nc{\sk}{{\mathsf{k}}}
\nc{\sm}{{\mathsf{m}}}
\nc{\sn}{{\mathsf{n}}}
\nc{\so}{{\mathsf{o}}}
\nc{\sq}{{\mathsf{q}}}
\nc{\sr}{{\mathsf{r}}}
\nc{\st}{{\mathsf{t}}}
\nc{\su}{{\mathsf{u}}}
\nc{\sv}{{\mathsf{v}}}
\nc{\sw}{{\mathsf{w}}}
\nc{\sx}{{\mathsf{x}}}
\nc{\sy}{{\mathsf{y}}}
\nc{\sz}{{\mathsf{z}}}

\nc{\oA}{{\overline{A}}}
\nc{\oB}{{\overline{B}}}
\nc{\oC}{{\overline{C}}}
\nc{\oD}{{\overline{D}}}
\nc{\oE}{{\overline{E}}}
\nc{\oF}{{\overline{F}}}
\nc{\oG}{{\overline{G}}}
\nc{\oH}{{\overline{H}}}
\nc{\oI}{{\overline{I}}}
\nc{\oJ}{{\overline{J}}}
\nc{\oK}{{\overline{K}}}
\nc{\oL}{{\overline{L}}}
\nc{\oM}{{\overline{M}}}
\nc{\oN}{{\overline{N}}}
\nc{\oO}{{\overline{O}}}
\nc{\oP}{{\overline{P}}}
\nc{\oQ}{{\overline{Q}}}
\nc{\oR}{{\overline{R}}}
\nc{\oS}{{\overline{S}}}
\nc{\oT}{{\overline{T}}}
\nc{\oU}{{\overline{U}}}
\nc{\oV}{{\overline{V}}}
\nc{\oW}{{\overline{W}}}
\nc{\oX}{{\overline{X}}}
\nc{\oY}{{\overline{Y}}}
\nc{\oZ}{{\overline{Z}}}

\nc{\oa}{{\overline{a}}}
\nc{\ob}{{\overline{b}}}
\nc{\oc}{{\overline{c}}}
\nc{\od}{{\overline{d}}}
\nc{\of}{{\overline{f}}}
\nc{\og}{{\overline{g}}}
\nc{\oh}{{\overline{h}}}
\nc{\oi}{{\overline{i}}}
\nc{\oj}{{\overline{j}}}
\nc{\ok}{{\overline{k}}}
\nc{\ol}{{\overline{l}}}
\nc{\om}{{\overline{m}}}
\nc{\on}{{\overline{n}}}
\nc{\oo}{{\overline{o}}}
\nc{\op}{{\overline{p}}}
\nc{\oq}{{\overline{q}}}
\nc{\os}{{\overline{s}}}
\nc{\ot}{{\overline{t}}}
\nc{\ou}{{\overline{u}}}
\nc{\ov}{{\overline{v}}}
\nc{\ow}{{\overline{w}}}
\nc{\ox}{{\overline{x}}}
\nc{\oy}{{\overline{y}}}
\nc{\oz}{{\overline{z}}}

\nc{\tA}{{\tilde{A}}}
\nc{\tB}{{\tilde{B}}}
\nc{\tC}{{\tilde{C}}}
\nc{\tD}{{\tilde{D}}}
\nc{\tE}{{\tilde{E}}}
\nc{\tF}{{\tilde{F}}}
\nc{\tG}{{\tilde{G}}}
\nc{\tH}{{\tilde{H}}}
\nc{\tI}{{\tilde{I}}}
\nc{\tJ}{{\tilde{J}}}
\nc{\tK}{{\tilde{K}}}
\nc{\tL}{{\tilde{L}}}
\nc{\tM}{{\tilde{M}}}
\nc{\tN}{{\tilde{N}}}
\nc{\tO}{{\tilde{O}}}
\nc{\tP}{{\tilde{P}}}
\nc{\tQ}{{\tilde{Q}}}
\nc{\tR}{{\tilde{R}}}
\nc{\tS}{{\tilde{S}}}
\nc{\tT}{{\tilde{T}}}
\nc{\tU}{{\tilde{U}}}
\nc{\tV}{{\tilde{V}}}
\nc{\tW}{{\tilde{W}}}
\nc{\tX}{{\tilde{X}}}
\nc{\tY}{{\tilde{Y}}}
\nc{\tZ}{{\tilde{Z}}}

\nc{\ta}{{\tilde{a}}}
\nc{\tb}{{\tilde{b}}}
\nc{\tc}{{\tilde{c}}}
\nc{\td}{{\tilde{d}}}
\nc{\te}{{\tilde{e}}}
\nc{\tf}{{\tilde{f}}}
\nc{\tg}{{\tilde{g}}}
\nc{\ti}{{\tilde{i}}}
\nc{\tj}{{\tilde{j}}}
\nc{\tk}{{\tilde{k}}}
\nc{\tl}{{\tilde{l}}}
\nc{\tm}{{\tilde{m}}}
\nc{\tn}{{\tilde{n}}}
\nc{\tp}{{\tilde{p}}}
\nc{\tq}{{\tilde{q}}}
\nc{\tr}{{\tilde{r}}}
\nc{\ts}{{\tilde{s}}}
\nc{\tu}{{\tilde{u}}}
\nc{\tv}{{\tilde{v}}}
\nc{\tw}{{\tilde{w}}}
\nc{\tx}{{\tilde{x}}}
\nc{\ty}{{\tilde{y}}}
\nc{\tz}{{\tilde{z}}}

\nc{\hA}{{\hat{A}}}
\nc{\hB}{{\hat{B}}}
\nc{\hC}{{\hat{C}}}
\nc{\hD}{{\hat{D}}}
\nc{\hE}{{\hat{E}}}
\nc{\hF}{{\hat{F}}}
\nc{\hG}{{\hat{G}}}
\nc{\hH}{{\hat{H}}}
\nc{\hI}{{\hat{I}}}
\nc{\hJ}{{\hat{J}}}
\nc{\hK}{{\hat{K}}}
\nc{\hL}{{\hat{L}}}
\nc{\hM}{{\hat{M}}}
\nc{\hN}{{\hat{N}}}
\nc{\hO}{{\hat{O}}}
\nc{\hP}{{\hat{P}}}
\nc{\hQ}{{\hat{Q}}}
\nc{\hR}{{\hat{R}}}
\nc{\hS}{{\widehat{S}}}
\nc{\hT}{{\hat{T}}}
\nc{\hU}{{\widehat{U}}}
\nc{\hV}{{\hat{V}}}
\nc{\hW}{{\hat{W}}}
\nc{\hX}{{\hat{X}}}
\nc{\hY}{{\hat{Y}}}
\nc{\hZ}{{\hat{Z}}}

\nc{\ha}{{\hat{a}}}
\nc{\hb}{{\hat{b}}}
\nc{\hc}{{\hat{c}}}
\nc{\hd}{{\hat{d}}}
\nc{\he}{{\hat{e}}}
\nc{\hf}{{\widehat{f}}}
\nc{\hg}{{\hat{g}}}
\nc{\hh}{{\hat{h}}}
\nc{\hi}{{\hat{i}}}
\nc{\hj}{{\hat{j}}}
\nc{\hk}{{\hat{k}}}
\nc{\hl}{{\hat{l}}}
\nc{\hm}{{\hat{m}}}
\nc{\hn}{{\hat{n}}}
\nc{\ho}{{\hat{o}}}
\nc{\hp}{{\hat{p}}}
\nc{\hq}{{\hat{q}}}
\nc{\hr}{{\hat{r}}}
\nc{\hs}{{\hat{s}}}
\nc{\hu}{{\hat{u}}}
\nc{\hv}{{\hat{v}}}
\nc{\hw}{{\hat{w}}}
\nc{\hx}{{\hat{x}}}
\nc{\hy}{{\hat{y}}}
\nc{\hz}{{\hat{z}}}

\nc{\eps}{\varepsilon}
\nc{\lan}{\big\langle}
\nc{\ran}{\big\rangle}
\nc{\kk}{{\Bbbk}}

\newcommand{\moplus}{\mathop{\textstyle\bigoplus}\limits}
\nc{\et}{{\mathrm{\acute{e}t}}}
\nc{\num}{{\mathrm{num}}}

\nc{\xrightiso}{ \xrightarrow{\ \raisebox{-0.5ex}[0ex][0ex]{$\sim$}\ }}

\def\bw#1#2{\textstyle{\bigwedge\hskip-0.9mm^{#1}}\hskip0.2mm{#2}}

\DeclareMathOperator{\cores}{\mathrm{cores}}
\DeclareMathOperator{\Res}{\mathrm{Res}}

\DeclareMathOperator{\Map}{\mathrm{Map}}
\DeclareMathOperator{\Hom}{\mathrm{Hom}}
\DeclareMathOperator{\Ext}{\mathrm{Ext}}

\DeclareMathOperator{\cHom}{\mathcal{H}\mathit{om}}

\DeclareMathOperator{\Spec}{\mathrm{Spec}}

\DeclareMathOperator{\Coh}{\mathrm{Coh}}
\DeclareMathOperator{\Qcoh}{\mathrm{Qcoh}}
\DeclareMathOperator{\Bl}{\mathrm{Bl}}

\DeclareMathOperator{\Pic}{\mathrm{Pic}}
\DeclareMathOperator{\Pictw}{\mathrm{Pic}^{\mathrm{tw}}}
\DeclareMathOperator{\bPic}{\mathbf{Pic}}

\newcommand{\uPic}{{\Pic}}

\DeclareMathOperator{\Sch}{\mathrm{Sch}}
\DeclareMathOperator{\Br}{\mathrm{Br}}
\DeclareMathOperator{\CH}{\mathrm{CH}}

\DeclareMathOperator{\Gal}{\mathbf{G}}

\DeclareMathOperator{\Sym}{\mathrm{Sym}}
\DeclareMathOperator{\Ker}{\mathrm{Ker}}

\DeclareMathOperator{\Cone}{\mathrm{Cone}}

\DeclareMathOperator{\pr}{\mathrm{pr}}

\DeclareMathOperator{\Gr}{\mathrm{Gr}}
\DeclareMathOperator{\OGr}{\mathrm{OGr}}

\DeclareMathOperator{\LGr}{\mathrm{LGr}}

\DeclareMathOperator{\Fl}{\mathrm{Fl}}

\DeclareMathOperator{\Gm}{\mathbb{G}_{\mathrm{m}}}

\DeclareMathOperator{\id}{\mathrm{id}}
\DeclareMathOperator{\rank}{\mathrm{rk}}

\nc{\bkk}{{\overline{\kk}}}
\newcommand{\g}{{\mathrm{g}}}

\makeatletter
\@addtoreset{equation}{section}
\makeatother

\theoremstyle{plain}

\newtheorem{theorem}{Theorem}[section]

\newtheorem{lemma}[theorem]{Lemma}
\newtheorem{proposition}[theorem]{Proposition}
\newtheorem{corollary}[theorem]{Corollary}

\theoremstyle{definition}

\newtheorem{definition}[theorem]{Definition}

\newtheorem{example}[theorem]{Example}
\newtheorem{notation}[theorem]{Notation}

\theoremstyle{remark}

\newtheorem{remark}[theorem]{Remark}

\title{Derived categories of families of Fano threefolds}
\author{Alexander Kuznetsov}
\address{{\sloppy
\parbox{0.99\textwidth}{
Algebraic Geometry Section, Steklov Mathematical Institute of Russian Academy of Sciences,\\
8 Gubkin str., Moscow 119991 Russia
\\[5pt]
Laboratory of Algebraic Geometry, NRU Higher School of Economics, Russian Federation
}\bigskip}}
\email{akuznet@mi-ras.ru}
\date{}
\thanks{I was partially supported by the HSE University Basic Research Program.}

\begin{document}

\begin{abstract}
We construct $S$-linear semiorthogonal decompositions of derived categories of smooth Fano threefold fibrations~$X/S$
with relative Picard rank~$1$ and rational geometric fibers 
and discuss how the structure of components of these decompositions is related to rationality properties of~$X/S$.
\end{abstract}

\maketitle

{\small 
\tableofcontents}

\section{Introduction}

Fano varieties form one of the most interesting classes of algebraic varieties.
Over an algebraically closed field of characteristic zero and in dimensions up to~3 
smooth Fano varieties have been completely classified.
In dimension~3 the classification, obtained by works of Fano, Iskovskikh, and Mori--Mukai, 
counts up to~105 deformation families.
Geometry of Fano threefolds has been thoroughly investigated;
in particular, quite a lot is known about their derived categories. 
The most important case of threefolds of Picard rank~1 was discussed in~\cite{k2009Fano} 
and in the general case one can use the Minimal Model Program to reduce the description 
to simpler Fano threefolds, or conic bundles, or del Pezzo surface fibrations, 
which are also in many cases accessible to investigation.
The goal of this paper is to study derived categories of smooth Fano threefolds~$X$ over \emph{non-closed} fields of characteristic zero,
as well as $G$-equivariant derived categories of $G$-Fano varieties, 
and more generally, derived categories of smooth families~$X/S$ of Fano threefolds 
over arbitrary connected characteristic zero base schemes~$S$
(the two cases above correspond to~$S = \Spec(\kk)$, the spectrum of a non-closed field, 
and~$S = \rB G$, the classifying stack of a finite group~$G$, respectively).
Note that the case of smooth families of Fano varieties of dimension~$1$ 
(i.e., $\P^1$-bundles) is easy (see Theorem~\ref{thm:bernardara}),
and the case of dimension~$2$ has been discussed in~\cite{AB}.

Of course, the main invariant of a family~$X/S$ of Fano threefolds is the deformation type of its geometric fibers,
i.e., of the fibers of~$X \to S$ over geometric points of the base
(in the case where~$S = \Spec(\kk)$, this is just the deformation type of the Fano threefold~$X_\bkk$,
and, if~$S = \rB G$, of the underlying threefold~$X$).
So, 105 deformation types in the Fano--Iskovskikh--Mori--Mukai classification 
lead to~105 types of Fano threefold fibrations.
As the number of types is rather large, and since the methods we have to access the derived category are rather ad hoc,
we restrict our attention to smooth Fano threefold fibrations~$X/S$ which enjoy the following two properties:
\begin{aenumerate}
\item 
\label{restriction-picard}
the relative Picard rank of~$X/S$ is~$1$, and
\item 
\label{restriction-rationality}
the geometric fibers of~$X/S$ are rational.
\end{aenumerate}
The reasons to consider only such~$X$ are quite obvious: 
property~\ref{restriction-picard} ensures that the study of~$X/S$ 
does not reduce by the Minimal Model Program to simpler cases,
while property~\ref{restriction-rationality} is relevant to potential applications to rationality problems.

Assumptions~\ref{restriction-picard} and~\ref{restriction-rationality} 
reduce the number of deformation types significantly, leaving only:
\begin{itemize}
\item 
8 types of Fano threefolds of geometric Picard rank~$1$: 
$\P^3$, quadric~$\sQ^3$, del Pezzo threefolds~$\sY_d$ with~$d \in \{4,5\}$, 
and prime Fano threefolds~$\sX_g$ with~$g \in \{7,9,10,12\}$;
\item 
6 types of Fano threefolds with higher geometric Picard rank;
over an algebraically closed field these varieties have the following explicit descriptions:
\begin{itemize}
\item 
$\sX_{1,1,1} = \P^1 \times \P^1 \times \P^1$;
\item 
$\sX_{2,2} \subset \P^2 \times \P^2$, a divisor of bidegree~$(1,1)$;
\item 
$\sX_{2,2,2} \subset \P^2 \times \P^2 \times \P^2$, a complete intersection of divisors of multidegree~$(1,1,0)$, $(1,0,1)$, and~$(0,1,1)$;
\item  
$\sX_{4,4} \subset \P^4 \times \P^4$, an intersection of the graph of the Cremona transformation~$\P^5 \dashrightarrow \P^5$
(given by quadrics passing through the Veronese surface) with~$\P^4 \times \P^4 \subset \P^5 \times \P^5$;
\item 
$\sX_{3,3} \subset \P^3 \times \P^3$, a complete intersection of three divisors of bidegree~$(1,1)$;
\item 
$\sX_{1,1,1,1} \subset \P^1 \times \P^1 \times \P^1 \times \P^1$, a divisor of multidegree~$(1,1,1,1)$.
\end{itemize}
\end{itemize}
Indeed, the case of geometric Picard rank~1 is classical,
and in the case of higher geometric Picard rank 
the classification of threefolds with property~\ref{restriction-picard} is contained in~\cite{Pro13},
while the restriction imposed by assumption~\ref{restriction-rationality} on the list from~\cite{Pro13} 
can be found in~\cite{AB92} (cf.~\cite{KP21}).

As it was already mentioned above, the goal of this paper 
is to study derived categories of smooth Fano fibrations~$X/S$ of the~14 types listed above;
more precisely, we will construct interesting \emph{$S$-linear} semiorthogonal decompositions of their derived categories
(see~\S\ref{ss:db-linear} for a reminder about the $S$-linear property).


As we also hinted, we expect the constructed semiorthogonal decompositions to have implications for rationality problems,
although we have no results in this direction and never mention rationality in the body of the paper.
So, having in mind rationality criteria from~\cite{KP19,KP21} 
in the case~$S = \Spec(\kk)$ 
for Fano threefolds of the above types
(which amount to the existence of points or appropriate rational curves defined over~$\kk$),
we will discuss how the components of our decompositions simplify
when~$S$ is arbitrary and the natural generalizations of these criteria 
(existence of sections of~$X/S$ or of appropriate relative Hilbert schemes over~$S$) are satisfied.

Our results are summarized in the following theorems.
We denote by~$\bD(X)$ the bounded derived category of coherent sheaves on~$X$,
and by~$\bD(Y,\upbeta)$ the bounded derived category of~$\upbeta$-twisted coherent sheaves on~$Y$,
where~$\upbeta \in \Br(Y)$ is a Brauer class.
We remnd the definition and main properties of twisted sheaves in~\S\ref{ss:twisted-sheaves}.

In the case where the geometric Picard rank of fibers of~$X/S$ is~1, 
the description we obtain is similar to the description over algebraically closed fields from~\cite{k2009Fano},
the main difference is the appearance of various Brauer classes that could not be observed over~$\bkk$.
First, there are four types of Fano threefolds~$X/S$, 
where all the components are twisted derived categories of~$S$.
We write~$X(S)$ for the set of all sections~$S \to X$ of the morphism~$X \to S$.

\begin{theorem}
\label{thm:main-base}
Let~$p \colon X \to S$ be a smooth projective morphism with geometric fibers isomorphic to~$\P^3$, or~$\sQ^3$, or~$\sY_5$, or~$\sX_{12}$.
Then~$\bD(X)$ has an $S$-linear semiorthogonal decomposition 
\begin{equation}
\label{eq:dbx-base}
\bD(X) = \langle \bD(S), \bD(S,\upbeta_1), \bD(S,\upbeta_2), \bD(S,\upbeta_3) \rangle
\end{equation}
with four components equivalent to twisted derived categories of the base~$S$, 
where the Brauer classes~\mbox{$\upbeta_i \in \Br(S)$} are the following:
\begin{aenumerate}
\item 
\label{brauer:y5}
if the fibers have type~$\sY_5$ then~$\upbeta_1 = \upbeta_2 = \upbeta_3 = 1$;
\item 
\label{brauer:q3-x12}
if the fibers have type~$\sQ^3$ or~$\sX_{12}$ then~$\upbeta_1 = \upbeta_2 = 1$ and~$\upbeta_3^2 = 1$;
\item 
\label{brauer:p3}
if the fibers have type~$\P^3$ then~$\upbeta_i = \upbeta^i$ where~$\upbeta \in \Br(S)$ is such that~$\upbeta^4 = 1$.
\end{aenumerate}
Moreover, if the fibers have type~$\P^3$ and~$X(S) \ne \varnothing$ then~$\upbeta = 1$,
and if the fibers have type~$\sQ^3$ or~$\sX_{12}$ and~$X(S) \ne \varnothing$ then~$\upbeta_3$ can be represented by a conic bundle.
\end{theorem}

This theorem is a combination of Theorem~\ref{thm:bernardara} (and Example~\ref{ex:p3}) 
and Theorems~\ref{thm:db-quadric}, \ref{thm:y5}, \ref{thm:x12}.

\begin{remark}
\label{rem:rationality-base}
As we mentioned above, it is interesting to compare these results to rationality criteria 
over non-closed fields~$\kk$.
Recall that Fano threefolds of type~$\sY_5$ are always rational over~$\kk$, 
while those of type~$\P^3$, $\sQ^3$, and~$\sX_{12}$ are rational over~$\kk$ 
if and only if~$X(\kk) \ne \varnothing$, see~\cite[Theorem~1.1]{KP19}.
We obtain a simple implication:
if~$X$ is rational over~$\kk$ then all Brauer classes appearing in the right hand side of~\eqref{eq:dbx-base} have order at most~2,
and those of order~2 can be represented by conic bundles.
We will discuss the meaning of this observation at the end of the Introduction.
\end{remark}

In the second case the category~$\bD(X)$ decomposes into two twisted derived categories of the base 
and a twisted derived category of a smooth projective curve over~$S$.
In the statement of the theorem below~$\rF_d(X/S)$ 
denotes the relative Hilbert scheme of rational curves of degree~$d$ 
(with respect to the primitive ample generator of the Picard group) in the fibers of~$X/S$
and we write~$\rF_d(X/S)(S)$ for the set of all sections~$S \to \rF_d(X/S)$ of the morphism~$\rF_d(X/S) \to S$.

\begin{theorem}
\label{thm:main-base+curve}
Let~$p \colon X \to S$ be a smooth projective morphism with geometric fibers isomorphic to~$\sY_4$, or~$\sX_{10}$, or~$\sX_9$, or~$\sX_7$.
Then~$\bD(X)$ has an $S$-linear semiorthogonal decomposition 
\begin{equation}
\label{eq:dbx-base+curve}
\bD(X) = \langle \bD(S), \bD(S,\upbeta_1), \bD(\Gamma,\upbeta_\Gamma) \rangle
\end{equation}
with two components equivalent to twisted derived categories of the base~$S$
and one component equivalent to a twisted derived category of a smooth projective curve~$\Gamma \to S$,
where 
\begin{aenumerate}
\item 
\label{brauer:y4}
if the fibers have type~\hbox to 1.4em{$\sY_4$\hfill} then~$\upbeta_1^2 = 1$, $\g(\Gamma) = 2$, and~$\upbeta_\Gamma^4 = 1$;
\item 
\label{brauer:x10}
if the fibers have type~\hbox to 1.4em{$\sX_{10}$\hfill} then~$\upbeta_1 = 1$, $\g(\Gamma) = 2$, and~$\upbeta_\Gamma^3 = 1$;
\item 
\label{brauer:x9}
if the fibers have type~\hbox to 1.4em{$\sX_9$\hfill} then~$\upbeta_1 = 1$, $\g(\Gamma) = 3$, and~$\upbeta_\Gamma^2 = 1$;
\item 
\label{brauer:x7}
if the fibers have type~\hbox to 1.4em{$\sX_7$\hfill} then~$\upbeta_1 = 1$, $\g(\Gamma) = 7$, and~$\upbeta_\Gamma = 1$.
\end{aenumerate}
Moreover, when the fibers have type~$\sY_4$, or~$\sX_{10}$, or~$\sX_9$ and one has
\begin{equation*}
X(S) \ne \varnothing
\qquad\text{and}\qquad 
\rF_d(X/S)(S) \ne \varnothing
\end{equation*}
where $d = 1$ for type~$\sY_4$, $d = 2$ for type~$\sX_{10}$, and~$d = 3$ for type~$\sX_9$, 
then~$\upbeta_1 = 1$ and~$\upbeta_\Gamma = 1$.
\end{theorem}

This theorem is a combination of Theorems~\ref{thm:y4}, \ref{thm:x10}, \ref{thm:x9}, and~\ref{thm:x7}.

\begin{remark}
\label{rem:rationality-base+curve}
As before, these results should be looked at from the perspective of the rationality criteria.
Indeed, threefolds of type~$\sX_7$ 
are rational over a non-closed field~$\kk$ if and only if~$X(\kk) \ne \varnothing$,
while~$\sY_4$, $\sX_{10}$, and~$\sX_9$ are rational over~$\kk$ 
if and only if~$X(\kk) \ne \varnothing$ and~$\rF_d(X)(\kk) \ne \varnothing$
(where~$d$ is the same as in Theorem~\ref{thm:main-base+curve}),
see~\cite{BW} and~\cite[Theorem~1.1]{KP19}.
Thus, as before, if~$X$ is rational over~$\kk$ then~$\upbeta_1 = 1$ and~$\upbeta_\Gamma = 1$.
\end{remark}

As we mentioned above, the results of Theorems~\ref{thm:main-base} and~\ref{thm:main-base+curve}
are just extensions to the relative case of the analogous results for Fano threefolds over algebraically closed fields.
In the last part of the paper, discussing the case of Fano fibrations with fibers of higher geometric Picard rank, 
we can no longer use the easy semiorthogonal decompositions 
of the corresponding Fano threefolds over algebraically closed fields
because they are not invariant under possible monodromy actions, and so they do not extend to $S$-linear decompositions.
Accordingly, to construct an $S$-linear semiorthogonal decomposition we need to find 
sufficiently symmetric semiorthogonal decompositions of derived categories of these threefolds.
We were able to do this in four out of six cases.
The new feature here is the appearance of two components equivalent to (twisted) derived categories
of finite \'etale coverings of the base of degree equal to geometric Picard rank of the fibers.

\begin{theorem}
\label{thm:main-base+coverings}
Let~$p \colon X \to S$ be a smooth projective morphism with geometric fibers 
isomorphic to~$\sX_{1,1,1}$, or~$\sX_{2,2}$, or~$\sX_{2,2,2}$, or~$\sX_{4,4}$.
Then~$\bD(X)$ has an $S$-linear semiorthogonal decomposition 
\begin{equation}
\label{eq:dbx-base+coverings}
\bD(X) = \langle \bD(S), \bD(S,\upbeta_1), \bD(S',\upbeta'_0), \bD(S',\upbeta'_1) \rangle
\end{equation}
where~$S' \to S$ is a finite \'etale covering of degree equal to the geometric Picard rank of~$X/S$,
and 
\begin{aenumerate}
\item 
\label{brauer:x111}
if the fibers have type~\hbox to 2.1em{$\sX_{1,1,1}$\hfill} then~$\upbeta_1^2 = 1$, ${\upbeta'_0}^2 = {\upbeta'_1}^2 = 1$;
\item 
\label{brauer:x22}
if the fibers have type~\hbox to 2.1em{$\sX_{2,2}$\hfill} then~$\upbeta_1 = 1$, ${\upbeta'_0}^3 = {\upbeta'_1}^3 = 1$;
\item 
\label{brauer:x222}
if the fibers have type~\hbox to 2.1em{$\sX_{2,2,2}$\hfill} then~$\upbeta_1^2 = 1$, ${\upbeta'_0} = {\upbeta'_1} = 1$;
\item 
\label{brauer:x44}
if the fibers have type~\hbox to 2.1em{$\sX_{4,4}$\hfill} then~$\upbeta_1 = 1$, $\upbeta'_0 =  1$, ${\upbeta'_1}^2 = 1$.
\end{aenumerate}
Moreover, if the fibers have type~$\sX_{1,1,1}$ or~$\sX_{2,2}$ and~$X(S) \ne \varnothing$ 
then~$\upbeta_1 = \upbeta'_0 = \upbeta'_1 = 1$ 
and if the fibers have type~$\sX_{2,2,2}$ or~$\sX_{4,4}$ and~$X(S) \ne \varnothing$ then~$\upbeta_1$ and~$\upbeta'_1$
can be represented by conic bundles.
\end{theorem}

This theorem is a combination of Theorems~\ref{thm:x111}, \ref{thm:x22}, \ref{thm:x222}, and~\ref{thm:x44}.

\begin{remark}
\label{rem:rationality-base+coverings}
For Fano threefolds of these types the criterion of rationality over a non-closed field~$\kk$
established in~\cite[Theorem~1.2(ii)]{KP21}
amounts to the existence of a $\kk$-point;
as before if it holds all Brauer classes appearing in~\eqref{eq:dbx-base+coverings} 
are trivial or can be represented by conic bundles.
\end{remark}

In the last two cases --- Fano fibrations with fibers of types~$\sX_{3,3}$ and~$\sX_{1,1,1,1}$ ---
we have not managed to find $S$-linear semiorthogonal decompositions in which all components are geometric.
The best we could achieve is the following result, 
where we use the notion of base change for semiorthogonal decompositions developed in~\cite{K11}.

\begin{theorem}
\label{thm:main-weird}
Let~$p \colon X \to S$ be a smooth projective morphism with geometric fibers 
isomorphic to~$\sX_{3,3}$ or~$\sX_{1,1,1,1}$.
Then~$\bD(X)$ has an $S$-linear semiorthogonal decomposition 
\begin{equation}
\label{eq:dbx-base+weird}
\bD(X) = \langle \bD(S), \bD(S',\upbeta'), \cA \rangle
\end{equation}
where~$S' \to S$ is a finite \'etale covering of degree equal to the geometric Picard rank of~$X/S$,
and 
\begin{aenumerate}
\item 
\label{brauer:x33}
if the fibers have type~$\sX_{3,3}$ then~${\upbeta'}^4 = 1$ 
and the base change~$\cA_{S'}$ of~$\cA$ along~$S' \to S$ has a semiorthogonal decomposition
\begin{equation*}
\cA_{S'} = \langle \bD(S', {\upbeta'}^2), \bD(\Gamma') \rangle,
\end{equation*}
where~$\Gamma' \to S'$ is a smooth projective curve of genus~$3$;
\item 
\label{brauer:x1111}
if the fibers have type~$\sX_{1,1,1,1}$ then~${\upbeta'}^2 = 1$
and the base change~$\cA_{S'}$ of~$\cA$ along~$S' \to S$ has a semiorthogonal decomposition
\begin{equation*}
\cA_{S'} = \langle \bD(S'', \upbeta''), \bD(\Gamma') \rangle,
\end{equation*}
where~$\Gamma' \to S'$ is a smooth projective curve of genus~$1$,
$S'' \to S'$ is a finite \'etale covering of degree~$3$, and~$\upbeta'' \in \Br(S'')$ is a Brauer class such that~${\upbeta''}^2 = 1$. 
\end{aenumerate}
\end{theorem}

This theorem is a combination of Theorems~\ref{thm:x33} and~\ref{thm:x1111}.

The components~$\cA \subset \bD(X)$ appearing in~\eqref{eq:dbx-base+weird} are very interesting.
As the theorem tells us, after an \'etale base change they decompose into two geometric components, 
but over~$S$ this decomposition is not defined.
The appearance of categories of this new type seems to be related 
to the new feature in the rationality behaviour over~$\kk$ observed in~\cite[Theorem~1.2(iii) and Conjecture~1.3]{KP21}:
Fano threefolds~$X$ of type~$\sX_{3,3}$ (and conjecturally of type~$\sX_{1,1,1,1}$ as well) 
are \emph{never rational} over~$\kk$
(under the usual assumption that the Picard rank over~$\kk$ is~$1$).

As it was mentioned in Remarks~\ref{rem:rationality-base}, \ref{rem:rationality-base+curve}, and~\ref{rem:rationality-base+coverings}
our results are compatible with the rationality criteria. 
The relation can be formulated in the language of hypothetical \emph{Griffiths components}.
Generalizing~\cite[Definition~3.9]{K16} we say
(see also Definition~\ref{def:griffiths}) 
that an $S$-linear indecomposable over~$S$ semiorthogonal component~\mbox{$\cA \subset \bD(X)$}
is a {\sf Griffiths component} if it does not have
an $S$-linear embedding into the derived category of a smooth projective variety over~$S$ 
of dimension at most~$\dim(X/S) -2$; 
such components are expected (see~\cite[\S3.3]{K16}) to provide obstructions to rationality, 
in the same way as the Griffiths components of intermediate Jacobians of Fano threefolds do.
The main issue with this definition is that semiorthogonal decompositions 
are known to violate the Jordan--H\"older property (see~\cite[\S3.4]{K16} and~\cite{K13}),
so that it is not clear if the set of Griffiths components of~$\bD(X)$ is independent on the choice of a semiorthogonal decomposition.

It is easy to see that categories of the form~$\bD(S')$ and~$\bD(\Gamma)$, 
where~$S' \to S$ is a finite \'etale morphism and~$\Gamma \to S$ is a smooth projective curve over~$S$, 
are non-Griffiths components for~$X$ of dimension~$3$ over~$S$.
Moreover, for~$S'$ as above if~$\upbeta' \in \Br(S')$ is a 2-torsion Brauer class
which can be represented by a conic bundle over~$S'$
then~$\bD(S',\upbeta')$ is also a non-Griffiths component 
(it can be embedded into the derived category of the conic bundle).
In Proposition~\ref{prop:non-griffiths} we show that these are the only possible non-Griffiths components.
Therefore, our results imply the following:

\begin{corollary}
\label{cor:griffiths}
Let~$p \colon X \to S$ be a smooth projective morphism as in Theorems~\textup{\ref{thm:main-base}}, 
\textup{\ref{thm:main-base+curve}}, or~\textup{\ref{thm:main-base+coverings}}.
If~$S = \Spec(\kk)$ and~$X$ is rational over~$S$ then~$\bD(X)$ has an $S$-linear semiorthogonal decomposition with no Griffiths components.
\end{corollary}

We consider this result as yet another confirmation of the Griffiths components philosophy. 

\begin{remark}
Note that the converse of Corollary~\ref{cor:griffiths} is not true: 
for instance a smooth Fano fibration~$X \to S$ with fibers of type~$\P^3$ 
associated to a nontrivial 2-torsion Brauer class which can be represented by a conic bundle
has no Griffiths components but is not rational over~$S$; 
however in this example~$X$ is \emph{stably birational} to the conic bundle.
\end{remark}

\medskip 

The paper is organized as follows.

In~\S\ref{sec:picard} we discuss the relative Picard group and twisted sheaves.
In~\S\ref{ss:picard} we define the monodromy action 
of the \'etale fundamental group of the base of a Fano fibration on the Picard group of its geometric fiber
and identify the invariant classes for this action with global sections of the Picard sheaf.
In~\S\ref{ss:twisted-sheaves} we remind the notion of twisted sheaves with respect to a Brauer class,
and in~\S\ref{ss:relative-divisors} we discuss the special case of relative twisted line bundles 
and their relation to morphisms to Severi--Brauer varieties.

In~\S\ref{sec:db-forms} we discuss derived categories and moduli spaces in the relative setting.
In~\S\ref{ss:db-linear} we review the notion of $S$-linear semiorthogonal decompositions and $S$-linear functors and their properties,
and remind a result of Bernardara about derived categories of Severi--Brauer varieties.
In~\S\ref{ss:moduli-definition} we discuss the definition and basic properties of moduli spaces
and prove that in some cases universal sheaves exist as twisted sheaves.
In~\S\ref{ss:uniqueness} we establish some general uniqueness results about stable vector bundles on Fano threefolds
which we use later to prove the uniqueness of Mukai bundles and to provide a modular interpretation 
to the curves appearing in semiorthogonal decompositions.

After these preparations we pass to the main story of the paper 
and discuss the case of Fano fibrations with fibers of geometric Picard rank~1:
in~\S\ref{sec:big-index} we discuss Fano fibrations of large index, 
i.e., smooth quadric fibrations and smooth del Pezzo fibrations of degree~5 and~4 (i.e., fiber types~$\sQ^3$, $\sY_5$ and~$\sY_4$), 
and in~\S\ref{sec:prime-fano} we discuss prime Fano fibrations (i.e., fiber types~$\sX_{12}$, $\sX_{10}$, $\sX_9$, and~$\sX_7$).

For the case of higher geometric Picard rank we first recall in~\S\ref{sec:weil} some material about Weil restriction of scalars:
in~\S\ref{ss:forms-powers} we classify Fano fibrations with geometric fibers 
being powers of other Fano varieties (Proposition~\ref{prop:forms-general}),
and in~\S\ref{ss:weil-decomposition} we prove a useful result 
about derived categories of Weil restrictions (Theorem~\ref{thm:pmd-qs-ps}), which is interesting by itself.
After that in~\S\ref{sec:big-picard} we discuss the case of Fano fibrations with fibers of higher geometric Picard rank~1.

Finally, in Appendix~\ref{sec:griffiths-components} we classify relative non-Griffiths components of threefold fibrations. 

\medskip

\noindent{\bf Conventions.}
All schemes in this paper are schemes of finite type over a field~$\kk$ of characteristic zero.
When we consider a morphism~$X \to S$, the base scheme~$S$ is usually assumed to be connected;
we denote by~$s_0 \in S$ a fixed geometric point and by~$X_{s_0}$ the corresponding geometric fiber of~$X$.

Given a Grassmannian~$\Gr(k,V)$ we denote by~$\cU$ and~$\cU^\perp$ the tautological subbundles of rank~$k$ and~$n - k$
in~$V \otimes \cO$ and~$V^\vee \otimes \cO$, respectively;
we also use the same notation for the relative Grassmannian~$\Gr_S(k,V)$, where~$V$ is a vector bundle on~$S$.

Finally, as it was already mentioned before, $\bD(X)$ stands for the bounded derived category of coherent sheaves on~$X$
and~$\bD(Y,\upbeta)$ is the bounded derived category of $\upbeta$-twisted coherent sheaves on~$Y$, 
where~$\upbeta \in \Br(Y)$ is a Brauer class.

\medskip
\noindent{\bf Acknowledgements.}
I would like to thank S.~Gorchinskiy, D.~Huybrechts, D.~Orlov, Yu.~Prokho\-rov, and~C.~Shramov for useful discussions.

\section{Relative Picard group and twisted sheaves}
\label{sec:picard}

In this section~$S$ is a connected scheme of finite type over the base field~$\kk$ of characteristic zero;
in particular~$S$ is noetherian.
If~$s_0 \in S$ is a geometric point we denote by~$\uppi_1(S,s_0)$ the \emph{\'etale fundamental group} of~$S$,
so that there is an equivalence of categories between finite \'etale morphisms $S' \to S$ and finite~$\uppi_1(S,s_0)$-sets.

\subsection{Monodromy action on the Picard group}
\label{ss:picard}

Let $p \colon X \to S$ be a smooth projective morphism with connected fibers.
Consider the \'etale sheaf of abelian groups
\begin{equation*}
\uPic_{X/S} \coloneqq \bR_\et^1p_*(\Gm),
\end{equation*}
where the direct image is taken in the \'etale topology.
In other words, this is the \'etale sheafification of the presheaf 
that associated to an \'etale morphism~$U \to S$ the group~$\Pic(X \times_S U)$.

We will often consider elements of the group~$\uPic_{X/S}(S) \coloneqq \rH^0(S, \uPic_{X/S})$
and call them {\sf relative divisor classes}.
Note that for any geometric point~$s \in S$ there is a natural restriction map 
\begin{equation}
\label{eq:picard-restriction}
\uPic_{X/S}(S) \to \Pic(X_s).
\end{equation}
One of the goals of this section is to interpret its image.
We concentrate on the case of smooth Fano fibrations, i.e., smooth projective morphisms such that~$-K_{X/S}$ is ample over~$S$.

\begin{proposition}
\label{prop:picard-covering}
Let~$p \colon X \to S$ be a smooth Fano fibration.
There is a finite \'etale morphism~$S' \to S$ and a geometric point~$s'_0 \in S'$ over~$s_0$ such that the restriction morphism
\begin{equation*}
\uPic_{X \times_S S'/S'}(S') \to \Pic((X \times_S S')_{s'_0}) \cong \Pic(X_{s_0})
\end{equation*}
is an isomorphism.
Moreover, we can assume~$S'$ is connected.
\end{proposition}

\begin{proof}
By~\cite[Theorem~9.4.8]{Kleiman} the Picard functor is represented by the scheme~$\bPic_{X/S}$ 
separated and locally of finite type over~$S$,
i.e., there is an isomorphism~$\Pic_{X/S}(T) \cong \Map_S(T,\bPic_{X/S})$ for all \'etale $S$-schemes~$T$.
Since by Kodaira vanishing and the Fano condition we have
\begin{equation*}
H^1(X_s,\cO_{X_s}) = H^2(X_s,\cO_{X_s}) = 0
\end{equation*}
for each geometric point~$s \in S$,
applying~\cite[Theorem~9.5.11, Remark~9.5.12, and Proposition~9.5.19]{Kleiman}
we conclude that the morphism~$\bPic_{X/S} \to S$ is \'etale.
Therefore, the restriction morphism
\begin{equation*}
\Pic_{X/S}(S) = \Map_S(S,\bPic_{X/S}) \to \Map_S(s_0,\bPic_{X/S}) = \Pic(X_{s_0})
\end{equation*}
is injective, and the same argument shows that it stays injective after any base change.
It remains to find a finite \'etale morphism~$S' \to S$ and a geometric point~$s'_0 \in S'$ over~$s_0$ 
such that after base change to~$S'$, the above morphism is surjective at~$s'_0$.

Let~$L \in \Pic(X_{s_0})$ be a line bundle, let~$\varphi$ be the Hilbert polynomial of~$L$ (with respect to the anticanonical polarization),
and let~$\uPic^{\varphi}_{X/S} \subset \uPic_{X/S}$ be the subfunctor 
that parameterizes line bundles on fibers of~$X/S$ with Hilbert polynomial~$\varphi$.
By~\cite[Theorem~9.6.20]{Kleiman} it is represented by an open and closed subscheme~$\bPic^{\varphi}_{X/S} \subset \bPic_{X/S}$
which is finite over~$S$.
Therefore, after base change to the finite \'etale covering~$S_L \coloneqq \bPic^{\varphi}_{X/S} \to S$ 
with marked point~$s_{L,0}$ corresponding to the line bundle~$L$ on~$X_{s_0}$
there is a section~$\uplambda \in \Pic_{X \times_S {S_L}/S_L}(S_L)$ with value at~$s_{L,0}$ equal to~$L$.
Thus, $L$ belongs to the image of the restriction morphism~$\Pic_{X \times_S {S_L}/S_L}(S_L) \to \Pic((X \times_S {S_L})_{s_{L,0}})$.

Now we choose a finite generating set~$\{L_i\}$, $1 \le i \le N$, for~$\Pic(X_{s_0})$ and applying the above construction 
we obtain a finite collection of finite \'etale morphisms~$S_{L_i} \to S$ with marked points~$s_{L_i,0}$ 
and sections~$\uplambda_i \in \Pic_{X_{S_{L_i}}/S_{L_i}}(S_{L_i})$.
It remains to take
\begin{equation*}
S' \coloneqq S_{L_1} \times_S S_{L_2} \times_S \dots \times_S S_{L_N}
\qquad\text{and}\qquad 
s'_0 \coloneqq (s_{L_1,0},s_{L_2,0},\dots,s_{L_N,0}).
\end{equation*}
Then the pullbacks of the sections~$\uplambda_i$ to~$\Pic(X \times_S {S'}/S')$ take value~$L_i$ at~$s'_0$,
hence the image of the restriction morphism~$\uPic_{X \times_S {S'}/S'}(S') \to \Pic((X \times_S {S'})_{s'_0})$ 
contains a generating set of the group~$\Pic((X \times_S {S'})_{s'_0})$, and therefore it is surjective.

Finally, if the scheme~$S'$ constructed above is not connected, 
just replace it by the connected component containing the point~$s'_0$.
\end{proof}

Let~$S' \to S$ be the \'etale morphism constructed in Proposition~\ref{prop:picard-covering}.
Let~$\hS \to S$ be a Galois \'etale covering which factors through~$S'$. 
Then for any geometric point~$\hs_0$ lying over~$s'_0$ 
the restriction morphism~$\Pic_{X \times_S \hS/\hS}(\hS) \to \Pic((X \times_S \hS)_{\hs_0}) \cong \Pic(X_{s_0})$ is bijective.
Consider the exact sequence
\begin{equation*}
1 \to \uppi_1(\hS,\hs_0) \to \uppi_1(S,s_0) \to \Gal_{\hS/S} \to 1,
\end{equation*}
where~$\Gal_{\hS/S}$ is the Galois group of~$\hS/S$.
The action of~$\Gal_{\hS/S}$ on~$\hS$ induces its action on~$\Pic_{X \times_S \hS/\hS}(\hS)$, 
and via the restriction isomorphism, a $\Gal_{\hS/S}$-action on~$\Pic(X_{s_0})$,
i.e., a continuous~$\uppi_1(S,s_0)$-action on~$\Pic(X_{s_0})$. 
We call it the {\sf monodromy action}.

\begin{remark}
It is easy to check that the above definition of the monodromy action does not depend on the choice of the morphism~$S' \to S$.
As we will not need this fact, we omit a verification.
\end{remark}

We have the following immediate consequence.

\begin{corollary}
\label{cor:picard-invariant}
If~$p \colon X \to S$ is a smooth Fano fibration and~$s_0 \in S$ is a geometric point 
the restriction morphism~\eqref{eq:picard-restriction} induces an isomorphism
\begin{equation}
\label{eq:picard-isomorphism}
\Pic_{X/S}(S) \cong \Pic(X_{s_0})^{\uppi_1(S,s_0)} \cong \Pic(X_{s_0})^{\Gal_{\hS/S}} \subset \Pic(X_{s_0}).
\end{equation} 
with the subgroup of monodromy invariant line bundles in the Picard group of a geometric fiber.
\end{corollary}

\begin{proof}
Injectivity of the restriction morphism has been shown in the proof of Proposition~\ref{prop:picard-covering}
and the description of the image follows from the above discussion taking into account an identification
\begin{equation*}
\Pic_{X/S}(S) \cong \Pic_{X \times_S \hS/\hS}(\hS)^{\Gal_{\hS/S}},
\end{equation*}
which in its turn follows from the fact that for any \'etale morphism~$U \to S$ if~$\hU \coloneqq U \times_S \hS$
the group~$\Gal_{\hS/S}$ acts freely on~$X \times_S \hU = (X \times_S U) \times_U \hU$ and the quotient is~$X \times_S U$.
\end{proof}

\begin{corollary}
\label{cor:orbit-covering}
Let~$p \colon X \to S$ be a smooth Fano fibration.
For a geometric point~\mbox{$s_0 \in S$} and a divisor class~$h_{s_0} \in \Pic(X_{s_0})$ let~$d$ denote the length 
of the monodromy orbit~\mbox{$\uppi_1(S,s_0) \cdot h_{s_0} \subset \Pic(X_{s_0})$}.
There is a finite \'etale morphism~$f \colon S' \to S$ of degree~$d$ with connected~$S'$
and a point~\mbox{$s'_0 \in f^{-1}(s_0)$} 
such that the class~$h_{s_0}\vert_{X_{s'_0}} \in \Pic(X_{s'_0})$ is monodromy invariant.
\end{corollary}

\begin{proof}
The covering~$S' \to S$ is associated with the~$\uppi_1(S,s_0)$-set being the orbit~$\uppi_1(S,s_0) \cdot h_{s_0}$;
the covering degree is equal to the length of the orbit, and
the scheme~$S'$ is connected because the~$\uppi_1(S,s_0)$-action on this set is transitive.
The point~$s'_0 \in S'$ corresponds to the point~$h_{s_0}$ in the orbit;
then the class~$h_{s_0}$ is invariant under the action of the subgroup~$\uppi_1(S',s'_0) \subset \uppi_1(S,s_0)$
which is equal to the stabiliser of~$h_{s_0}$ in~$\uppi_1(S,s_0)$, 
hence~$h_{s_0}\vert_{X_{s'_0}}$ is monodromy invariant.
\end{proof}

The monodromy action preserves intrinsic geometric structures of~$\Pic(X_{s_0})$.

\begin{lemma}
\label{lemma:nef-cone}
Let~$p \colon X \to S$ be a smooth Fano fibration.
The monodromy action of~$\uppi_1(S,s_0)$ on~$\Pic(X_{s_0})$ 
preserves the canonical class and the nef cone in~$\Pic(X_{s_0})$.
\end{lemma}

\begin{proof}
The relative canonical class~$K_{X/S}$ provides a global section of~$\Pic_{X/S}$ over~$S$,
and its restriction to the geometric fiber~$X_{s_0}$ is the canonical class~$K_{X_{s_0}}$.
By Corollary~\ref{cor:picard-invariant} we conclude that~$K_{X_{s_0}}$ is monodromy invariant.

To prove the invariance of the nef cone we need to show that if~$L$ is a line bundle on~$X/S$ 
and the restriction of~$L$ to~$X_{s_0}$ is nef then the restriction of~$L$ to~$X_{s_1}$ is nef for any other geometric point~$s_1$ of~$S$.
A standard argument reduces the general statement to the case where~$S$ is a complex curve and~$s_0$, $s_1$ are its closed points;
in this case the required result is proved in~\cite[Theorem~1]{Wis}.
\end{proof}

Recall the following standard invariants of a Fano variety~$X$ over an algebraically closed field:
\begin{itemize}
\item 
the {\sf Picard rank}~$\uprho(X) \coloneqq \rank(\Pic(X))$;
\item 
the {\sf Fano index}~$\upiota(X) \coloneqq \max\{ m \in \ZZ_{>0} \mid K_X \in m \Pic(X) \}$;
\item 
the {\sf fundamental divisor class}~$H_X \coloneqq -\frac1{\upiota(X)}K_X \in \Pic(X)$.
\end{itemize}
Furthermore, for a coherent sheaf~$\cF$ we denote by~$\upchi(\cF)$ its Euler characteristic.

\begin{corollary}
\label{cor:constant}
If~$p \colon X \to S$ is a smooth Fano fibration, the integers~$\uprho(X_s)$, $\upiota(X_s)$, $\upchi(\cO_{X_s}(H_{X_s}))$ 
are constant as functions of geometric point~$s \in S$.
Moreover, there is a unique relative divisor class~$H_X \in \Pic_{X/S}(S)$
that restricts to the fundamental divisor class of each geometric fiber.
\end{corollary}

The relative divisor class~$H_X \in \Pic_{X/S}(S)$ is called the {\sf fundamental class} of the Fano fibration.

\begin{proof}
Let~$s_0,s_1 \in S$ be two geometric points.
Applying Proposition~\ref{prop:picard-covering} we can find a finite \'etale morphism~$S' \to S$ with connected~$S'$
such that for points~$s'_0$ and~$s'_1$ over~$s_0$ and~$s_1$, respectively, we have a chain of group isomorphisms
\begin{equation*}
\Pic(X_{s_1}) \cong \Pic(X_{s'_1}) \cong \Pic_{X \times_S S'/S'}(S') \cong \Pic(X_{s'_0}) \cong \Pic(X_{s_0}),
\end{equation*}
where the middle isomorphisms are given by the restriction maps.
We conclude from this that the Picard ranks of~$X_{s_1}$ and~$X_{s_0}$ are the same.
Moreover, under the above isomorphisms the canonical classes correspond to each other 
(because both correspond to the relative canonical class of~$X/S$),
therefore~$\upiota(X_{s_1}) = \upiota(X_{s_0})$ 
and the fundamental divisors correspond to each other.
Finally, the Euler characteristics of these divisors are determined by the Hilbert polynomial of the anticanonical classes,
which agree because the morphism~$X \to S$ is flat.
\end{proof}

\subsection{Brauer group and twisted sheaves}
\label{ss:twisted-sheaves}

Recall that the \emph{Brauer group}~$\Br(S)$ of a scheme~$S$ is defined 
as the group of Morita-equivalence classes of Azumaya algebras on~$S$
with the operation of tensor product.
This group is closely related to the torsion subgroup of~$\rH^2_\et(S,\Gm)$,
which is known as the \emph{cohomological Brauer group}~$\Br'(S)$;
in fact, there is a natural injective morphism
\begin{equation*}
\Br(S) \hookrightarrow \Br'(S),
\end{equation*}
and for quasiprojective schemes the two groups coincide, see~\cite{dJ}.
We will not need this result; however we will widely use the language of twisted sheaves adopted in~\cite{dJ}
(see also~\cite{Cal,Lieb} for details). 
We remind the basic definitions in this subsection.

Let~$S$ be a scheme and let~$\upbeta \in \rH^2_\et(S, \Gm)$
be an \'etale cohomology class.
Assume for simplicity the class~$\upbeta$ can be represented by a \v{C}ech cocycle~$\beta \in \Gamma(U \times_S U \times_S U, \Gm)$ 
in an \'etale cover~\mbox{$U \to S$}.
Then a {\sf $(U,\beta)$-twisted (quasi)coherent sheaf on~$S$} is defined as
a (quasi)coherent sheaf~$\cF_U$ on~$U$ together with an isomorphism
\begin{equation*}
\varphi \colon \pr_1^*\cF_U \xrightiso \pr_2^*\cF_U,
\end{equation*}
on~$U \times_S U$, where $\pr_1, \pr_2 \colon U \times_S U \to U$ are the projections, 
satisfying the condition
\begin{equation}
\label{eq:cocycle-condition}
\pr_{1,2}^*\varphi \circ \pr_{2,3}^*\varphi \circ \pr_{1,3}^*\varphi^{-1} = \beta \cdot \id,
\end{equation}
where $\pr_{i,j} \colon U \times_S U \times_S U \to U \times_S U$ are the projections to the product of $i$-th and $j$-th factors.

If~$(\cF_U,\varphi)$ is a $(U,\beta)$-twisted quasicoherent sheaf 
and $\beta' \in \Gamma(U \times_S U \times_S U, \Gm)$ is another representative of the same cohomology class, 
i.e., $\beta^{-1} \cdot \beta' = \partial \gamma$ for some~$\gamma \in \Gamma(U \times_S U, \Gm)$, 
modifying~$\varphi$ by~$\gamma$, we obtain a $(U,\beta')$-twisted quasicoherent sheaf~$(\cF_U,\gamma \cdot\varphi)$.
Similarly, if~$U' \to U$ is a refining of the cover~$U \to S$ and~$\beta'$ is the pullback to~$U'$ of the \v{C}ech cocycle~$\beta$,
then the pullback of~$(\cF_U,\varphi)$ to~$U'$ is a~$(U',\beta')$-twisted quasicoherent sheaf.

We define a {\sf $\upbeta$-twisted (quasi)coherent sheaf on~$S$} as an equivalence class of~$(U,\beta)$-twisted (quasi)coherent sheaves,
where~$U \to S$ is an \'etale cover and~$\beta$ is a \v{C}ech cocyle representing~$\upbeta$,
under the two above operations (modifying the \v{C}ech cocycle by a coboundary and passing to a refinement of the cover).
To define a morphism of $\upbeta$-twisted quasicoherent sheaves we may assume they are represented 
by~$(U,\beta)$-twisted quasicoherent sheaves~$(\cF_{1U},\varphi_1)$ and~$(\cF_{2U},\varphi_2)$ for the same~$U$ and~$\beta$;
then a morphism is given by a morphism~$f \colon \cF_{1U} \to \cF_{2U}$ of quasicoherent sheaves on~$U$ 
such that the equality~$\varphi_2 \circ \pr_1^*(f) = \pr_2^*(f) \circ \varphi_1$ holds on~$U \times_S U$.
We will denote by~$\Qcoh(S,\upbeta)$ and~$\Coh(S,\upbeta)$ the abelian categories of $\upbeta$-twisted (quasi)coherent sheaves on~$S$ 
and by~$\bD(S,\upbeta)$ the bounded derived category of complexes of $\upbeta$-twisted quasicoherent sheaves with coherent cohomology.

In a contrast to the usual category of coherent sheaves, the category~$\Coh(X,\upbeta)$ does not have a monoidal structure,
but there is a replacement for it described in the following lemma.

\begin{lemma}[{\cite[Proposition~1.2.10]{Cal}}]
\label{lemma:twisted-sd-ld}
If~$\cF'$ is a $\upbeta'$-twisted sheaf and~$\cF''$ is a $\upbeta''$-twisted sheaf 
then~$\cF' \otimes \cF''$ is~$\upbeta' \cdot \upbeta''$-twisted and~$\cHom(\cF', \cF'')$ is~$(\upbeta')^{-1} \cdot \upbeta''$-twisted.
In particular, if~$\cF$ is a $\upbeta$-twisted sheaf, 
then~$\cHom(\cF,\cO)$ is~$\upbeta^{-1}$-twisted and~$\Sym^d(\cF)$ and~$\wedge^d(\cF)$ are $\upbeta^d$-twisted sheaves.
\end{lemma}

The following corollary is standard.

\begin{corollary}
\label{cor:twisted-rank}
If~$\cF$ is a $\upbeta$-twisted vector bundle of rank~$r$ then $\upbeta^r = 1$.
\end{corollary}

\begin{proof}
First, assume~$r = 1$.
Let~$(\cF_U,\varphi)$ be an~$(U,\beta)$-twisted sheaf representing~$\cF$.
Refining the cover~$U$ if necessary, we may assume that~$\cF_U \cong \cO_U$.
Then~$\varphi$ is an invertible function on~$U \times_S U$, 
and the condition~\eqref{eq:cocycle-condition} means that~$\beta = \partial \varphi$,
hence the cohomology class of~$\beta$ is trivial.

Now for any~$r \ge 1$, the line bundle~$\wedge^r(\cF)$ is~$\upbeta^r$-twisted by Lemma~\ref{lemma:twisted-sd-ld},
hence~$\upbeta^r = 1$ by the first part of the lemma.
\end{proof}

Twisted sheaves are functorial for pullbacks and pushforwards if the twists are compatible.

\begin{lemma}[{\cite[\S\S2.2--2.3]{Cal}}]
\label{lemma:twisted-pb-pf}
Let~$f \colon T \to S$ be a morphism of schemes.
If~\mbox{$\upbeta \in \rH^2_\et(S,\Gm)$} there is an adjoint pair of functors
\begin{equation*}
f^* \colon \Qcoh(S, \upbeta) \to \Qcoh(T, f^*\upbeta)
\qquad\text{and}\qquad 
f_* \colon \Qcoh(T, f^*\upbeta) \to \Qcoh(S, \upbeta)
\end{equation*}
Under appropriate finiteness conditions \textup(properness for~$f_*$, flatness for~$f^*$\textup)
these functors extend to an adjoint pair of derived functors between derived categories~$\bD(S, \upbeta)$ and~$\bD(T,f^*\upbeta)$
such that the pullback functor is monoidal and the pushforward functor satisfies the projection formula.
\end{lemma}

The notion introduced below may seem not very natural, 
but it appears often when dealing with Severi--Brauer varieties (see~\S\ref{ss:relative-divisors}) 
and universal sheaves on moduli spaces (see~\S\ref{ss:moduli-definition}).

\begin{definition}
If~$\upbeta \in \rH^2_\et(S,\Gm)$ we define a {\sf relative $p^*(\upbeta)$-twisted vector bundle on~$X/S$} 
as an equivalence class of~$(X \times_S U, p^*(\beta))$-twisted vectot bundles~$(\cF_{X \times_S U},\varphi)$
(where~$U \to S$ is an \'etale cover and~$\beta$ is a \v{C}ech cocyle representing~$\upbeta$)
with respect to the equivalence relation coming from refining~$U$ (the cover of~$S$) and replacing~$\beta$ by a coboundary.
Note that we \emph{do not allow} to refine~$X \times_S U$ by covers which are not pullbacks of covers of~$U$.
\end{definition}

First, consider the case of line bundles.
We denote by~$\Pic^\upbeta(X/S)$ the set of isomorphism classes of relative $p^*(\upbeta)$-twisted line bundles on~$X/S$ and set
\begin{equation*}
\Pictw(X/S) \coloneqq \bigoplus_{\upbeta \in \rH^2_\et(S,\Gm)} \Pic^\upbeta(X/S)
\end{equation*}
to be the set of isomorphism classes of relative twisted line bundles on~$X/S$,
where the twist is allowed to vary in the group~$\rH^2_\et(S,\Gm)$.

If~$\cL_1$ and~$\cL_2$ are relative~$p^*(\upbeta_1)$ and~$p^*(\upbeta_2)$-twisted line bundles on~$X/S$ 
then as in Lemma~\ref{lemma:twisted-sd-ld} one can define~$\cL_1 \otimes \cL_2$ 
as a relative $p^*(\upbeta_1 \cdot \upbeta_2)$-twisted line bundle on~$X/S$.
This operation endows the set~$\Pictw(X/S)$ with a commutative group structure.

If~$\cL$ is an untwisted line bundle on~$X$, 
it can be considered as a relative twisted line bundle on~$X/S$ with the trivial twist.
This defines an injective morphism~$\Pic(X) \to \Pictw(X/S)$ which we call {\sf the canonical embedding}.
Composing it with the pullback morphism~$p^* \colon \Pic(S) \to \Pic(X)$, 
we obtain an injective morphism~$\Pic(S) \to \Pictw(X/S)$.

\begin{lemma}
\label{lemma:picard-twisted}
If the morphism~$p \colon X \to S$ is smooth and proper with connected fibers then
there is a natural isomorphism~$\Pictw(X/S)/\Pic(S) \cong \Pic_{X/S}(S)$ of abelian groups.
\end{lemma}

\begin{proof}
Assume~$\upbeta \in \rH^2_\et(S,\Gm)$ and~$\cL$ is a relative $p^*(\upbeta)$-twisted line bundle on~$X/S$.
Then there is an \'etale cover~$U \to S$ such that~$\upbeta\vert_U$ is trivial.
Therefore, the pullback~$\cL_U$ of~$\cL$ to~$X \times_S U$ is an untwisted line bundle, i.e., an element of~$\Pic_{X/S}(S)$.
This defines a group homomorphism 
\begin{equation}
\label{eq:twisted-picard-morphism}
\Pictw(X/S) \to \Pic_{X/S}(S).
\end{equation}
If~$\cL$ is in the kernel, there is a cover~$U \to S$ such that~$\cL_U \cong \cO_U$.
The assumptions about the morphism~$p$ imply that 
the gluing isomorphism~$\varphi \colon \pr_1^*\cL_U \xrightiso \pr_2^*\cL_U$ on~$(X \times_S U) \times_X (X \times_S U)$
is a pullback of an isomorphism~$\bar\varphi$ on~$U \times_S U$, which defnes a $(U,\beta)$-twisted line bundle~$\bar\cL$ on~$S$.
It follows from Corollary~\ref{cor:twisted-rank} that the cohomology class of~$\beta$ is trival,
and that the line bundle~$\bar\cL$ is untwisted.
Furthermore, since~$\varphi = p^*(\bar\varphi)$, it follows that~$\cL \cong p^*(\bar\cL)$.
This argument proves that the kernel of~\eqref{eq:twisted-picard-morphism} 
is the subgroup~$p^*(\Pic(S)) \subset \Pic(X) \subset \Pictw(X/S)$.

Now assume a relative divisor class~$h \in \Pic_{X/S}(S)$ is given. 
By definition there is an \'etale cover~$U \to S$ and a line bundle~$\cL_{U}$ on~$X \times_S U$ 
such that~$\pr_1^*\cL_{U} \cong \pr_2^*\cL_{U}$
on~$(X \times_S U) \times_X (X \times_S U)$ 
(more precisely, the definition tells that we have an isomorphism up to a line bundle on~$U \times_S U$,
but refining the cover~$U$ we may assume this line bundle to be trivial).
Let us choose such an isomorphism~$\varphi$ and consider the composition in the left-hand side of~\eqref{eq:cocycle-condition}.
It is an automorphism of a line bundle, hence given by an invertible function.
The assumptions about the morphism~$p$ imply that this function can be written as~$p^*(\beta)$, 
where~$\beta$ is a \v{C}ech 2-cocycle on~$U$.
This means that~$(\cL_{U},\varphi)$ is an $(X \times_S U,p^*(\beta))$-twisted line bundle on~$X$,
and by definition of the morphism~\eqref{eq:twisted-picard-morphism} 
the image of the corresponding relative twisted line bundle in~$\Pic_{X/S}(S)$ is~$h$.
This proves the surjectivity of~\eqref{eq:twisted-picard-morphism}.
\end{proof}

\begin{notation}
\label{not:oh}
Let~$p \colon X \to S$ be a smooth proper morphism with connected fibers.
Given a relative divisor class~$h \in \Pic_{X/S}(S)$
we denote by~$\cO_X(h)$ a relative twisted line bundle on~$X/S$ 
(defined up to twist by a line bundle on~$S$) corresponding to it under the isomorphism of Lemma~\ref{lemma:picard-twisted}.
\end{notation}

Recall the standard exact sequence (see, e.g., ~\cite[Proposition~2.5]{Liedtke}):
\begin{equation}
\label{eq:picard-brauer-sequence}
0 \to \Pic(S) \to \Pic(X) \to \uPic_{X/S}(S) \xrightarrow{\ \bB\ } 
\rH^2_\et(S,\Gm).
\end{equation}
Comparing the proof of Lemma~\ref{lemma:picard-twisted} with the definition of the morphism~$\bB$,
we can rewrite it as 
\begin{equation}
\label{eq:twisted-picard-brauer-sequence}
0 \to \Pic(X)/\Pic(S) \to \Pictw(X/S)/\Pic(S) \xrightarrow{\ \bB\ } 
\rH^2_\et(S,\Gm),
\end{equation}
where the second arrow takes a relative $p^*(\upbeta)$-twisted line bundle on~$X/S$ to the corresponding cohomology class~$\upbeta$.

Now we consider the more general situation of relative twisted vector bundles.
We will often use the restrictions they impose on the corresponding Brauer classes.
The first is quite straightforward.

\begin{lemma}
\label{lemma:beta-kr}
If~$\cE$ is a relative~$p^*(\upbeta)$-twisted vector bundle 
then~$\wedge^d(\cE)$ is a relative~$p^*(\upbeta^d)$-twisted vector bundle.
In particular, if~$\bB(\det(\cE)) \in \rH^2_\et(S,\Gm)$ is $m$-torsion, 
and the morphism~$p \colon X \to S$ is smooth and proper with connected fibers then~$\upbeta^{m\rank(\cE)} = 1$.
\end{lemma}

\begin{proof}
The first part is analogous to Lemma~\ref{lemma:twisted-sd-ld}
and the second follows from Lemma~\ref{lemma:picard-twisted} 
combined with the above description of the map~$\bB$.
\end{proof}

The second restriction is a bit more involved.

\begin{lemma}
\label{lemma:beta-2}
Let~$p \colon X \to S$ be a smooth and proper morphism with connected fibers and
let~$\cE$ be a relative~$p^*(\upbeta)$-twisted vector bundle.
Assume for each geometric point~$s \in S$ we have~$\dim(\rH^0(X_s,\cE_{X_s}^\vee)) = 2\rank(\cE)$,
the evaluation morphism extends to an exact sequence
\begin{equation*}
0 \to \cE_{X_s} \to \rH^0(X_s,\cE_{X_s}^\vee) \otimes \cO_{X_s} \to \cE_{X_s}^\vee \to 0,
\end{equation*}
and~$\Hom(\cE_{X_s},\cE_{X_s})$ is generated by the identity morphism.
Then~$\upbeta^2 = 1$.
\end{lemma}

\begin{proof}
Set $V \coloneqq \bR^0p_*(\cE^\vee)$.
By assumption, semicontinuity theorem, and Lemma~\ref{lemma:twisted-pb-pf} 
this is a $\upbeta^{-1}$-twisted vector bundle 
and there is an exact sequence of~$p^*(\upbeta^{-1})$-twisted vector bundles
\begin{equation*}
0 \to \cE' \to p^*V \to \cE^\vee \to 0.
\end{equation*}
The restrictions of this sequence to geometric fibers of~$p$ recover the exact sequences for~$\cE_{X_s}$,
hence~$\cE'_{X_s} \cong \cE_{X_s}$ for each geometric point~$s \in S$.
Therefore the space
\begin{equation*}
\rH^0(X_s, (\cE^\vee \otimes \cE')\vert_{X_s}) \cong
\Hom(\cE_{X_s}, \cE'_{X_s}) \cong
\Hom(\cE_{X_s}, \cE_{X_s}) 
\end{equation*}
is 1-dimensional, and hence~$\bR^0p_*(\cE^\vee \otimes \cE')$ is a~$\upbeta^{-2}$-twisted line bundle on~$S$.
Now we apply Corollary~\ref{cor:twisted-rank} and conclude that~$\upbeta^{-2} = 1$, hence the claim.
\end{proof}

\subsection{Relative divisor classes and morphisms to Severi--Brauer varieties}
\label{ss:relative-divisors}

Recall that a {\sf Severi--Brauer variety} over~$S$ is a morphism~$p \colon Y \to S$
which is \'etale locally trivial fibration with fiber~$\P^{N-1}$.
As this is a Fano fibration, we can talk about the fundamental class~$H_Y$ of~$Y/S$.
Note that by Corollary~\ref{cor:picard-invariant} and Lemma~\ref{lemma:nef-cone} we have~$\Pic_{Y/S}(S) \cong \ZZ H_Y$.

Using the language of twisted vector bundles it is easy to give a description of all Severi--Brauer varieties over a given scheme~$S$.
Indeed, let~$\cE$ be a $\upbeta$-twisted vector bundle on~$S$ 
(by Corollary~\ref{cor:twisted-rank} this implies that~$\upbeta$ is torsion)
represented by a $(U,\beta)$-twisted vector bundle~$(\cE_U,\varphi)$, where~$U \to S$ is an \'etale cover.
Then the isomorphism~$\varphi \colon \pr_1^*\cE_U \xrightiso \pr_2^*\cE_U$ induces an isomorphism
\begin{equation*}
\bar\varphi \colon \pr_1^*(\P_U(\cE_U)) \xrightiso \pr_2^*(\P_U(\cE_U))
\end{equation*}
and~~\eqref{eq:cocycle-condition} implies that~$\pr_{1,2}^*\bar\varphi \circ \pr_{2,3}^*\bar\varphi \circ \pr_{1,3}^*\bar\varphi^{-1} = \id$.
Therefore, the isomorphism~$\bar\varphi$ can be used to glue~$\P_U(\cE_U)$ into a scheme which we denote by~$\P_S(\cE)$ 
and which is endowed with a projection
\begin{equation*}
p \colon \P_S(\cE) \to S.
\end{equation*}
Moreover, if~$\cO_{\P_U(\cE_U)}(1)$ is the Grothendieck line bundle on~$\P_U(\cE_U)$ (such that its pushforward to~$U$ is~$\cE_U^\vee$)
the isomorphism~$\varphi$ provides it with a structure of a relative $(\P_U(\cE_U),p^*(\beta^{-1}))$-twisted line bundle.
We denote the corresponding relative~$p^*(\upbeta^{-1})$-twisted line bundle by~$\cO_{\P_S(\cE)}(1)$.
Note that
\begin{equation*}
p_*\cO_{\P_S(\cE)}(1) \cong \cE^\vee;
\end{equation*}
this follows by gluing the standard isomorphism over the \'etale cover~$U \to S$.

\begin{lemma}
\label{lemma:sb-twisted-sheaves}
If~$\cE$ is a twisted vector bundle on~$S$ then~$p \colon \P_S(\cE) \to S$ is a Severi--Brauer variety over~$S$.
Conversely, if~$p \colon Y \to S$ is a Severi--Brauer variety then~$Y \cong \P_S(\cE)$ 
for a twisted vector bundle on~$S$ unique up to twist by a line bundle on~$S$.
\end{lemma}

\begin{proof}
By construction of~$\P_S(\cE)$ the morphism~$p \colon \P_S(\cE) \to S$ 
is \'etale locally (over~$S$) isomorphic to the projectivization of a trivial vector bundle,
hence it is a Severi--Brauer variety.

Conversely, if~$p \colon Y \to S$ is a Severi--Brauer variety,
$H_Y \in \Pic_{Y/S}(S)$ is its relative fundamental class, 
and~$\cO_Y(H_Y) \in \Pictw(Y/S)$ is the corresponding relative twisted line bundle (see Notation~\ref{not:oh}),
then~$\cE \coloneqq p_*(\cO_Y(H_Y))^\vee$ is a twisted vector bundle on~$S$
(both~$\cO_Y(H_Y)$ and~$\cE$ are defined up to twist by a line bundle on~$S$), 
and it is clear that~$Y \cong \P_S(\cE)$.
\end{proof}

For a morphism of schemes~$p \colon X \to S$
we will say that a class~$h \in \uPic_{X/S}(S)$ is {\sf relatively ample}, 
{\sf relatively globally generated}, {\sf relatively has vanishing higher cohomology}, and so on,
if the corresponding properties hold for the restrictions of this class to all geometric fibers of~$X/S$.
If~$p$ is proper and flat we denote by~$\upchi_{X/S}(h)$ the {\sf relative Euler characteristic} of~$h$,
defined as the Euler characteristic of the corresponding line bundle on any geometric fiber of~$X/S$.

Recall the morphism~$\bB$ from~\eqref{eq:picard-brauer-sequence} and Notation~\ref{not:oh}.
The following observation is quite useful.

\begin{lemma}
\label{lemma:brauer-obstruction}
Let~$p \colon X \to S$ be a smooth proper morphism with connected fibers.
If~$h \in \uPic_{X/S}(S)$ is a relatively globally generated class with vanishing higher cohomology then 
\begin{equation*}
\bB(h) \in \Br(S) \subset \Br'(S) \subset \rH^2_\et(S, \Gm)
\end{equation*}
and there is a $\bB(h)^{-1}$-twisted vector bundle~$\cE$ of rank~$N \coloneqq \upchi_{X/S}(h)$ 
on~$S$ and an $S$-morphism
\begin{equation*}
\phi_h \colon X \to \P_S(\cE)
\end{equation*}
such that~$\phi_h^*(\cO_{\P_S(\cE)}(1)) \cong \cO_X(h)$
and the morphism~$(\phi_h)_s \colon X_s \to (\P_S(\cE))_s \cong \P^{N-1}$
coincides with the morphism given by the complete linear system~$|h\vert_{X_s}|$
for every geometric point~$s \in S$.
\end{lemma}

\begin{proof}
This is essentially the content of~\cite{Liedtke}; however, for the readers' convenience we sketch a proof.
Choose an \'etale cover $U \to S$ such that~$h$ is represented by a line bundle~$\cL \in \Pic(X_U)$,
and consider the sheaf~$\cE_U \coloneqq p_{U*}(\cL)^\vee$, where~$p_U \colon X_U \to U$ is the base change of~$p$.
Then (since~$\cL$ is globally generated and has no higher cohomology on the fibers of~$p_U$)
the sheaf~$\cE_U$ is locally free of rank~$N$ and there is a unique morphism~$\phi_\cL \colon X_U \to \P_U(\cE_U)$
such that the canonical epimorphism~$p_U^*\cE_U^\vee \to \cL$ is the pullback under~$\phi_\cL$ of the tautological epimorphism.
After gluing we obtain the morphism~$\phi_h \colon X \to \P_S(\cE)$ that has all required properties.
\end{proof}

Now let~$p \colon X \to S$ be a smooth Fano fibration.
Recall from~\S\ref{ss:picard} the definition of the relative Fano index~$\upiota(X/S)$ and the relative fundamental class~$H_X$.

\begin{corollary}
\label{cor:fundamental-class}
Let~$p \colon X \to S$ be a smooth Fano fibration with the relative fundamental class~$H_X$.
Let~$m \coloneqq \gcd(\upiota(X/S), \upchi_{X/S}(H_X))$.
Then~$\bB(H_X)^m = 1$.
\end{corollary}

\begin{proof}
The class $\upiota(X/S) H_X = -K_{X/S} \in \Pic_{X/S}(S)$ comes from a class in~$\Pic(X)$,
hence its image in~$\Br(S) \subset \rH^2_\et(S,\Gm)$ under the map~$\bB$ in~\eqref{eq:picard-brauer-sequence} is trivial,
hence~$\bB(H_X)^{\upiota(X/S)} = 1$.

On the other hand, the vector bundle~$p_*(\cO_X(H_X))$ is~$\bB(H_X)$-twisted 
(note that~$H_X$ is relatively ample by definition, hence by Kodaira vanishing it relatively has vanishing higher cohomology)
and has rank~$\upchi_{X/S}(H_X)$,
hence~$\bB(H_X)^{\upchi_{X/S}(H_X)} = 1$ by Corollary~\ref{cor:twisted-rank}.
\end{proof}

\section{Derived categories and moduli spaces}
\label{sec:db-forms}

In this section we remind some results about derived categories and moduli spaces of smooth fibrations.
In~\S\ref{ss:db-linear} we recall the notions of $S$-linear decompositions and functors
and state a criterion for $S$-linear functors to be fully faithful and generate a semiorthogonal decomposition.
We also remind a result of Bernardara about derived categories of Severi--Brauer varieties.
In~\S\ref{ss:moduli-definition} we remind the definition and basic properties of (relative) moduli spaces of sheaves.
Finally, in~\S\ref{ss:uniqueness} we prove some uniqueness results for stable sheaves.

Starting from this section all functors are derived.

\subsection{Linear semiorthogonal decompositions and forms of~$\P^3$}
\label{ss:db-linear}

Let~$p \colon X \to S$ be a smooth projective morphism.
Recall from~\cite{K06} that~$\bD(X) = \langle \cA_1, \dots, \cA_k \rangle$ is an {\sf $S$-linear} semiorthogonal decomposition 
if the components~$\cA_i$ are preserved by tensor products with pullbacks of objects of~$\bD(S)$, i.e.,
\begin{equation*}
\cA_i \otimes p^*\bD(S) \subset \cA_i
\end{equation*}
for all~$i$.
One can think of $S$-linear semiorthogonal decompositions 
as families of semiorthogonal decompositions of fibers of~$p$.
In particular, as it was shown in~\cite{K11}, one can apply base change to a point embedding~$\{s\} \hookrightarrow S$
and obtain a semiorthogonal decomposition
\begin{equation*}
\bD(X_s) = \langle {\cA_1}_s, \dots, {\cA_k}_s \rangle
\end{equation*}
of the fiber, called the base change of the original decomposition.

Here is a sample example of $S$-linear semiorthogonal decomposition.

\begin{theorem}[{\cite{Bernardara}}]
\label{thm:bernardara}
If $p \colon X \to S$ is a Severi--Brauer variety, 
$n = \dim(X/S)$, 
$H_X$ is the fundamental class of~$X/S$,
and $\upbeta = \bB(H_X) \in \Br(S)$ is the corresponding Brauer class, 
then for each~$i \in \ZZ$ there is an $S$-linear semiorthogonal decomposition
\begin{equation*}
\bD(X) = \Big\langle \cO_X(iH_X) \otimes \bD(S, \upbeta^{-i}), 
\cO_X((i+1)H_X) \otimes \bD(S,\upbeta^{-i-1}), \dots, 
\cO_X((i+n)H_X) \otimes \bD(S,\upbeta^{-i-n}) \Big\rangle.
\end{equation*}
\end{theorem}

Despite of the ambiguity in the choice of a twisted line bundle~$\cO_{X}(H_X)$,
the components of the above decomposition do not depend on this choice.
Note also that if~$X/S$ has a section, then~$\upbeta = 1$,
hence all the components of the above decomposition are equivalent to~$\bD(S)$.

\begin{example}
\label{ex:p3}
Let~$X \to S$ be a $3$-dimensional Severi--Brauer variety 
and let~\mbox{$\upbeta = \bB(H_X) \in \Br(S)$} be the corresponding $4$-torsion Brauer class.
Then~$\bD(X)$ has an $S$-linear semiorthogonal decomposition
\begin{equation*}
\bD(X) = \Big\langle \cO_X \otimes \bD(S), \cO_X(H_X) \otimes \bD(S,\upbeta^{-1}), 
\cO_X(2H_X) \otimes \bD(S,\upbeta^{-2}), \cO_X(3H_X) \otimes \bD(S,\upbeta^{-3}) \Big\rangle.
\end{equation*}
\end{example}

In the rest of this section we state a result which is used
for obtaining an $S$-linear semiorthogonal decomposition of~$\bD(X)$ from semiorthogonal decompositions of the fibers of~$X/S$.

We concentrate on the situation where the components are twisted derived categories, see~\S\ref{ss:twisted-sheaves}.
Given a smooth projective morphism~$q \colon Y \to S$ and a geometric point~$s \in S$ we denote by
\begin{equation*}
\upeta_s \colon X_s \times Y_s \cong (X \times_S Y)_s \hookrightarrow X \times_S Y
\end{equation*}
the natural embedding.
Given an appropriately twisted object~$\cE$ on a fiber product of two varieties 
we denote by~$\Phi_\cE$ the corresponding $S$-linear Fourier--Mukai functor 
between their twisted derived categories.

\begin{proposition}
\label{prop:relative-sod}
Let~$q_i \colon Y_i \to S$, $1 \le i \le k$, be smooth projective morphisms, 
let~$\upbeta_i \in \Br(Y_i)$ be Brauer classes, 
and let~$\cE_i \in \bD(X \times_S Y_i, \pr_{Y_i}^*(\upbeta_i))$ be $\pr_{Y_i}^*(\upbeta_i)$-twisted objects.

\begin{renumerate}
\item
\label{item:linear-ff}
If for every geometric point~$s \in S$ 
the functor~\mbox{$\Phi_{\upeta_s^*\cE_i} \colon \bD({Y_i}_s, \upbeta_i^{-1}\vert_{{Y_i}_s}) \to \bD(X_s)$}
is fully faithful then the functor~$\Phi_{\cE_i} \colon \bD(Y_i, \upbeta_i^{-1}) \to \bD(X)$
is also fully faithful.
Moreover, its image is an $S$-linear admissible subcategory in~$\bD(X)$.
\item
\label{item:linear-so}
If for every geometric point~$s \in S$ the subcategories~$\Phi_{\upeta_s^*\cE_i}(\bD({Y_i}_s, \upbeta_i^{-1}\vert_{{Y_i}_s}))$ 
are semiorthogonal in~$\bD(X_s)$ for~$i = i_1, i_2$,
then also the subcategories~$\Phi_{\cE_i}(\bD(Y_i,\upbeta_i^{-1}))$ 
are semiorthogonal in~$\bD(X)$ for~$i = i_1, i_2$.
\item
\label{item:linear-sod}
If for every geometric point~$s \in S$ there is a semiorthogonal decomposition
\begin{equation*}
\bD(X_s) = \Big\langle \Phi_{\upeta_s^*\cE_1}(\bD({Y_1}_s, \upbeta_1^{-1}\vert_{{Y_1}_s})), \dots, 
\Phi_{\upeta_s^*\cE_k}(\bD({Y_k}_s, \upbeta_k^{-1}\vert_{{Y_k}_s})) \Big\rangle
\end{equation*}
then also there is an $S$-linear semiorthogonal decomposition
\begin{equation*}
\bD(X) = \Big\langle \Phi_{\cE_1}(\bD(Y_1,\upbeta_1^{-1})), \dots, \Phi_{\cE_k}(\bD(Y_k,\upbeta_k^{-1})) \Big\rangle.
\end{equation*}
\end{renumerate}
\end{proposition}

\begin{proof}
See~\cite[Proposition~2.44]{K06} and~\cite[proof of Theorem~5.2]{K19}.
\end{proof}

We will often use the special case of Proposition~\ref{prop:relative-sod} where~$Y_i = S$.
If~$p \colon X \to S$ is a smooth projective morphism and~$\upbeta_1,\dots,\upbeta_k \in \Br(S)$,
we say that a collection of objects
\begin{equation*}
\cE_1 \in \bD(X,p^*(\upbeta_1)), \dots, \cE_k \in \bD(X,p^*(\upbeta_k))
\end{equation*}
is a {\sf relative exceptional collection} if the collection~$\cE_1\vert_{X_s}, \dots, \cE_k\vert_{X_s} \in \bD(X_s)$ 
is exceptional for each geometric point~$s \in S$.

\begin{corollary}
\label{cor:relative-exceptional}
If~$\cE_i \in \bD(X,p^*(\upbeta_i))$, $1 \le i \le k$, is a relative exceptional collection,
the functors 
\begin{equation*}
\Phi_{\cE_i} \colon \bD(S, \upbeta_i^{-1}) \to \bD(X),
\qquad 
\cF \mapsto \cE \otimes p^*(\cF)
\end{equation*}
are fully faithful and the subcategories~$\cE_i \otimes \bD(S,\upbeta_i^{-1}) \coloneqq \Phi_{\cE_i}(\bD(S,\upbeta_i^{-1}))$, $1 \le i \le k$,
form a semiorthogonal collection of admissible $S$-linear subcategories in~$\bD(X)$.
\end{corollary}

\subsection{Moduli spaces and universal bundles}
\label{ss:moduli-definition}

Let~$p \colon X \to S$ be a smooth projective morphism,
let~$H \in \Pic(X)$ be a relatively ample divisor class,
and let~$\rP(t) \in \QQ[t]$ be a polynomial.
We denote by
\begin{equation*}
f \colon \rM_{X/S,H}(\rP) \to S
\end{equation*}
the relative moduli space of Gieseker semistable sheaves on fibers of~$p$
with Hilbert polynomial~$\rP$ (with respect to the polarization given by the restriction of~$H$). 
This is the coarse moduli space for (the \'etale sheafification of) the functor~$\fM_{X/S,H}(\rP)$ 
from the category of schemes over~$S$ to the category of groupoids
that associates to a morphism~$T \to S$ the groupoid of all sheaves~$E$ on~$X \times_S T$ which are flat over~$T$
and such that for each geometric point~$t \in T$ the sheaf~~$E_t$ on~$X_t$ is $H\vert_{X_t}$-semistable and has Hilbert polynomial~$\rP$.

We refer to~\cite[\S4]{HL} for the details of the definition and basic properties of the moduli space
(and~\cite[Theorem~4.3.7]{HL} for the existence)
and to~\cite{Maruyama,Simpson} for technical details (especially in the relative case).
Here we state some of the most important properties.
The first is immediate from the definition.

\begin{theorem}[{\cite[Theorem~4.3.4]{HL}}]
\label{thm:moduli-existence}
The natural morphism~$f \colon \rM_{X/S,H}(\rP) \to S$ is projective and is compatible with base change, i.e.,
\begin{equation*}
\rM_{X/S,H}(\rP) \times_S S' \cong \rM_{X \times_S S'/S',H}(\rP)
\end{equation*}
for any morphism~$S' \to S$.
In particular, the geometric fibers of~$\rM_{X/S,H}(\rP)$ 
are the moduli spaces of semistable sheaves on the corresponding geometric fibers of~$X \to S$.
\end{theorem}

\begin{remark}
\label{rem:moduli-chern-hilbert}
Note that the Hilbert polynomial of a sheaf is determined by its Chern classes via the Hirzebruch--Riemann--Roch theorem.
However, different values of Chern classes may give rise to the same Hilbert polynomials.
Anyway, we will sometimes abuse notation by writing
\begin{equation*}
\rM_{X/S,H}(r;\rc_1,\rc_2,\dots,\rc_n) \coloneqq
\rM_{X/S,H}(\rP_{r;\,\rc_1,\rc_2,\dots,\rc_n}),
\end{equation*}
where~$\rP_{r;\,\rc_1,\rc_2,\dots,\rc_n}$ is the Hilbert polynomial of sheaves with the given rank and Chern classes,
even if the listed Chern classes are not determined by the Hilbert polynomial.
Moreover, when~$X/S$ is a relative Fano threefold with the geometric Picard number of fibers equal to~$1$,
we will often use notation~$L_X \coloneqq \frac1{H_X^3}H_X^2$ for the class of a line on~$X$ and~$P_X$ for the class of a point. 
\end{remark}

We will need the following general result.

\begin{theorem}[{\cite[Proposition~6.7]{Maruyama}, \cite[Corollary~4.5.2]{HL}}]
\label{thm:moduli-smooth}
Let~$p \colon X \to S$ be a smooth projective morphism.
Let~$\cE$ be a stable vector bundle on a geometric fiber~$X_s$ of~$p$ with Hilbert polynomial~$\rP$.
Assume~$\Ext^2(\cE,\cE) = 0$.
Then
\begin{renumerate}
\item
the morphism~$f \colon \rM_{X/S,H}(\rP) \to S$ is smooth at~$[\cE]$, and
\item
the relative tangent space of~$f$ at~$[\cE]$ is isomorphic to~$\Ext^1(\cE,\cE)$.
\end{renumerate}
\end{theorem}

The following two results will be used in the paper to identify some important moduli spaces.

\begin{corollary}
\label{cor:moduli-etale}
Let~$p \colon X \to S$ be a smooth projective morphism.
Let~$\rM^\circ \subset \rM_{X/S,H}(\rP)$ be an open subscheme in the relative moduli space.
Assume that
\begin{aenumerate}
\item 
\label{item:moduli-bijective}
the natural morphism~$\rM^\circ \to S$ is bijective on geometric points, and 
\item 
\label{item:moduli-exceptional}
every sheaf on a fiber of~$p$ corresponding to a geometric point of~$\rM^\circ$ is exceptional. 
\end{aenumerate}
Then the morphism~$\rM^\circ \to S$ is an isomorphism.
\end{corollary}
\begin{proof}
The morphism~$\rM^\circ \to S$ is smooth of relative dimension zero by Theorem~\ref{thm:moduli-smooth}, hence \'etale.
Since it is also bijective, is an isomorphism by~\cite[\href{https://stacks.math.columbia.edu/tag/02LC}{Tag 02LC}]{SP}.
\end{proof}

\begin{corollary}
\label{cor:moduli-curve}
Let~$p \colon X \to S$ be a smooth projective morphism.
Let~$\Gamma \to S$ be a smooth projective curve and let~$\cE \in \Coh(X \times_S \Gamma)$ 
be a family of $H$-stable sheaves on fibers of~$p$ with Hilbert polynomial~$\rP$ parameterized by~$\Gamma$.
If the Fourier--Mukai functor~$\Phi_\cE \colon \bD(\Gamma) \to \bD(X)$ is fully faithful
then the corresponding morphism~$\phi_\cE \colon \Gamma \to \rM_{X/S,H}(\rP)$ is an isomorphism onto an open subscheme.
\end{corollary}

\begin{proof}
Set~$\rM \coloneqq \rM_{X/S,H}(\rP)$.
The morphism~$\phi_\cE$ is \'etale because for any point~$y \in \Gamma$ the functor~$\Phi_\cE$, being fully faithful, 
induces an isomorphism of tangent spaces
\begin{equation*}
\rT_{y, \Gamma/S} = \Ext^1(\cO_y, \cO_y) \xrightarrow{ \Phi_\cE\ } \Ext^1(\cE_y, \cE_y) = \rT_{[\cE_y], \rM/S}.
\end{equation*}
On the other hand, the morphism~$\Gamma \to \rM$ is injective (again by full faithfulness of~$\Phi_\cU$).
Therefore, it is an open immersion, see~\cite[\href{https://stacks.math.columbia.edu/tag/02LC}{Tag 02LC}]{SP},
i.e., an isomorphism onto an open subscheme.
\end{proof}

We end this subsection with a discussion of the existence of a universal sheaf on the fiber product~\mbox{$X \times_S \rM_{X/S,H}(\rP)$};
in fact, under appropriate assumptions we show it exists as a twisted sheaf.
We use the following result.

\begin{theorem}[{\cite[Proposition~4.6.2]{HL}, \cite[Theorem~1.21(4)]{Simpson}}]
\label{theorem:moduli-universal}
Assume all sheaves classfied by the moduli space~$\rM = \rM_{X/S,H}(\rP)$ are $H$-stable. 
Then \'etale locally on~$\rM$ there exists a universal sheaf~$\cE$.
\end{theorem}

The precise meaning of the theorem is the following.
There is an \'etale cover~$\rM' \to \rM$ and a sheaf~$\cE' \in \fM_{X/S,H}(\rP)(\rM')$ on~$X \times_S \rM'$  
such that for any scheme~$T$ and any sheaf~$\cF \in \fM_{X/S,H}(\rP)(T)$ on~$X \times_S T$ 
there is a unique morphism~$T \to \rM$ such that the pullbacks of the sheaves~$\cF$ and~$\cE'$ 
to~$(X \times_S T) \times_\rM \rM' \cong (X \times_S \rM') \times_\rM T$ are isomorphic up to twist by a line bundle on~$T \times_\rM \rM'$.

\begin{proposition}[{cf.~\cite[Proposition~3.3.2]{Cal}}]
\label{prop:twisted-universal}
Assume all sheaves classified by the moduli space~\mbox{$\rM = \rM_{X/S,H}(\rP)$} are $H$-stable. 
Let~$\pr_\rM \colon X \times_S \rM \to \rM$ be the projection.
There exists a Brauer class~$\upbeta \in \Br(\rM)$ and a $\pr_\rM^*(\upbeta)$-twisted sheaf~$\cE$ on~$X \times_S \rM$
such that for each point~$m \in \rM$ of the moduli space the sheaf~$\cE_m$ on~$X_{f(m)}$ 
is the $H$-stable sheaf corresponding to the point~$m$.
\end{proposition}

\begin{proof}
Let~$\rM' \to \rM$ be an \'etale cover and let~$\cE' \in \fM_{X/S,H}(\rP)(\rM')$ be an \'etale local universal sheaf on~$X \times_S \rM'$.
The universal property of~$\cE'$ implies that the sheaves~$\pr_1^*\cE'$ 
and~$\pr_2^*\cE'$ on the fiber product~$(X \times_S \rM) \times_\rM \rM' \times_\rM \rM' \cong X \times_S \rM' \times_S \rM'$
agree up to a line bundle twist.
Refining the \'etale cover~$\rM' \to \rM$ we may assume that the line bundle is trivial
and we have an isomorphism~$\varphi \colon \pr_1^*\cE' \xrightiso \pr_2^*\cE'$.
Using stability of sheaves classified by the moduli space it is easy to deduce 
the cocycle condition~\eqref{eq:cocycle-condition} for appropriate~$\beta$.
Since~$X \to S$ is smooth and proper with connected fibers, it follows that~$\beta$ 
is a pullback from~$\rM' \times_\rM \rM' \times_\rM \rM'$.
If~$\upbeta \in \rH^2_\et(\rM,\Gm)$ is the corresponding Brauer class, it follows that~$(\cE',\varphi)$ 
defines a $\pr_\rM^*(\upbeta)$ sheaf on~$X \times_S\rM$; by construction it has the universal property.
Finally, to conclude that~$\upbeta \in \Br(\rM) \subset \rH^2_\et(\rM,\Gm)$ is a Brauer class,
it is enough to note that for~$n \gg 0$ the pushforward~${\pr_\rM}_*(\cE(nH)) \in \Coh(\rM,\upbeta)$ is a $\upbeta$-twisted vector bundle.
\end{proof}

\subsection{Some uniqueness results}
\label{ss:uniqueness}

In this section we prove some uniqueness results for vector bundles on Fano threefolds over algebraically closed fields.
In the next proposition stability, slope~$\upmu$, and the Hilbert polynomials 
are taken with respect to the anticanonical polarization~$-K_X$
(or, equivalently, with respect to the fundamental class~$H_X$).

\begin{proposition}
\label{prop:uniqueness-general}
Let~$X$ be a smooth Fano threefold.
Let 
\begin{equation}
\label{eq:sequence-cu-cv-cw}
0 \to \cU \to \cV \to \cW \to 0
\end{equation}
be an exact sequence of vector bundles, where
\begin{aenumerate}
\item 
\label{item:uniqueness-cu}
$\cU$ is stable with $\upchi(\cU,\cU) > 0$, 
\item 
\label{item:uniqueness-cw}
$\cW$ is semistable with~$\upmu(\cW(K_X)) < \upmu(\cU)$. 
\end{aenumerate}
If~$\cE$ is a semistable bundle with
\begin{equation}
\label{eq:cu-ce-assumptions}
\upchi(\cU,\cE) = \upchi(\cU,\cU),
\qquad 
\upmu(\cE) = \upmu(\cU),
\qquad 
\rank(\cE) = \rank(\cU),
\end{equation} 
and~$\Ext^2(\cV,\cE) = 0$ then~$\cE \cong \cU$.
\end{proposition}

\begin{proof}
We have~$\upmu(\cE)  = \upmu(\cU) > \upmu(\cW(K_X))$.
Since both~$\cE$ and~$\cW$ are semistable, we conclude that~$\Hom(\cE, \cW(K_X)) = 0$, and hence by Serre duality~$\Ext^3(\cW,\cE) = 0$.
Now applying the functor~$\Ext^\bullet(-,\cE)$ to~\eqref{eq:sequence-cu-cv-cw} and using the assumption~$\Ext^2(\cV,\cE) = 0$, 
we obtain
\begin{equation*}
\Ext^2(\cU,\cE) = 0.
\end{equation*}
On the other hand, $\upchi(\cU,\cE) = \upchi(\cU,\cU) > 0$.
Therefore, $\Hom(\cU,\cE) \ne 0$, and since the bundles~$\cU$ and~$\cE$ 
are semistable of the same slope and rank, and~$\cU$ is stable, 
any nontrivial morphism between them must be an isomorphism.
\end{proof}

In practice the conditions~\eqref{eq:cu-ce-assumptions} may be deduced from the numerical equality 
of Chern classes of~$\cE$ and~$\cU$.
However, sometimes it is also possible to deduce these conditions 
from the equality of the Hilbert polynomials of~$\cE$ and~$\cU$.

\begin{lemma}
\label{lemma:hilbert-euler}
Let~$X$ be a smooth projective variety and let~$H \in \Pic(X)$ be an ample divisor class.
Assume~$\cF_1$, $\cF_2$, $\cF$ are coherent sheaves on~$X$ 
such that
\begin{equation*}
\rP_H(\cF_1,t) = \rP_H(\cF_2,t)
\qquad\text{and}\qquad 
\rch(\cF) \in \QQ[H] \subset \CH^\bullet_{\mathrm{num}}(X, \QQ),
\end{equation*}
where~$\rP_H(\cF_p, t)$ are the Hilbert polynomials of~$\cF_p$, 
$\CH^\bullet_{\mathrm{num}}(X, \QQ)$ is the Chow ring with rational coefficients modulo numerical equivalence,
and~$\QQ[H]$ is its subring generated by~$H \in \CH^1(X)$. 
Then~$\upchi(\cF, \cF_1) = \upchi(\cF, \cF_2)$.
\end{lemma}

\begin{proof}
The assumption~$\rch(\cF) \in \QQ[H]$ implies that~$\rch(\cF) = \sum a_i\rch(\cO_X(iH))$ for some~$a_i \in \QQ$,
hence by Hirzebruch--Riemann--Roch~$\upchi(\cF, \cF_p) = \sum a_i \rP_H(\cF_p,-i)$ for~$p = 1,2$.
Thus, the equality of the Hilbert polynomials implies the equality~$\upchi(\cF, \cF_1) = \upchi(\cF, \cF_2)$.
\end{proof}

Another useful result is the following.

\begin{proposition}
\label{prop:curve-general}
Let~$X$ be a smooth Fano threefold.
Let~$\Gamma$ be a smooth proper curve and let~$\cU$ be a vector bundle on~$X \times \Gamma$
such that the Fourier--Mukai functor
\begin{equation*}
\Phi_\cU \colon \bD(\Gamma) \to \bD(X)
\end{equation*}
is fully faithful.
Assume for each point~$y \in \Gamma$ the corresponding vector bundle~$\cU_y$ on~$X$ is stable,
and fits into an exact sequence
\begin{equation}
\label{eq:sequence-cu-cv-cw-2}
0 \to \cU_y \to \cV'_y \to \cV''_y \to \cW_y \to 0.
\end{equation}
If~$\cE$ is a semistable bundle such that~\eqref{eq:cu-ce-assumptions} holds for~$\cU = \cU_y$
and~$\Ext^2(\cV'_y,\cE) = \Ext^3(\cV''_y,\cE) = 0$ for all points~$y \in \Gamma$
then either~$\cE \cong \cU_y$ for some point~$y \in \Gamma$,
or~$\cE \in \Phi_\cU(\bD(\Gamma))^\perp$.
\end{proposition}

\begin{proof}
First, if~$\Hom(\cU_y,\cE) \ne 0$ for some~$y \in \Gamma$ then~$\cU_y \cong \cE$ because
the bundles are semistable of the same slope and rank, and~$\cU_y$ is stable.
So, assume that
\begin{equation*}
\Hom(\cU_y,\cE) = 0
\end{equation*}
for all~$y \in \Gamma$.
Applying the functor~$\Ext^\bullet(-,\cE)$ to the exact sequence~\eqref{eq:sequence-cu-cv-cw-2} 
and using the assumptions~$\Ext^2(\cV'_y,\cE) = \Ext^3(\cV''_y,\cE) = 0$, we conclude that
\begin{equation*}
\Ext^2(\cU_y,\cE) = 0
\end{equation*}
for all~$y \in \Gamma$.
Now from~\eqref{eq:cu-ce-assumptions}, full faithfulness of~$\Phi_\cU$, and smoothness of~$\Gamma$ 
we deduce
\begin{equation*}
\upchi(\cU_y,\cE) = \upchi(\cU_y,\cU_y) = \upchi(\Phi_\cU(\cO_y), \Phi_\cU(\cO_y)) = \upchi(\cO_y, \cO_y) = 0.
\end{equation*}
Combining this with the vanishing~$\Hom(\cU_y,\cE) = \Ext^2(\cU_y,\cE) = 0$ proved above, we conclude that
\begin{equation*}
\Ext^1(\cU_y,\cE) = \Ext^3(\cU_y,\cE) = 0
\end{equation*}
for all~$y \in \Gamma$, i.e., by adjunction~$\Ext^\bullet(\cO_y, \Phi_\cU^!(\cE)) = 0$ for all~$y \in \Gamma$,
where~$\Phi_\cU^! \colon \bD(X) \to \bD(\Gamma)$ is the right adjoint functor of~$\Phi_\cU$.
We conclude that~$\Phi_\cU^!(\cE) = 0$, hence~$\cE \in \Phi_\cU(\bD(\Gamma))^\perp$.
\end{proof}

\section{Quadrics and del Pezzo threefolds}
\label{sec:big-index}

In this section we consider Fano fibrations with geometrically rational fibers 
of geometric Picard number~1 and index greater than~1, excluding the well-known case of the projective space.
Sometimes we will loosely call such Fano fibrations ``forms'' of the corresponding varieties.

\subsection{Forms of~$\sQ^3$}

First, we consider forms of smooth quadrics.
This, of course, is also a well-known case, but we provide a proof relying on the results from the previous sections,
because a similar approach works for other types of Fano threefold fibrations.

We start with a reminder of the situation over an algebraically closed field.
In this case if~$X$ is a smooth 3-dimensional quadric 
there is a vector space~$V$ of dimension~$5$ and an embedding~\mbox{$X \hookrightarrow \P(V)$}
as a hypersurface with equation given by a quadratic form~$\Sym^2 V \to \kk$.
Rephrasing this description, we may say that~$X$ is a hyperplane section 
of the second Veronese embedding~$\P(V) \hookrightarrow \P(\Sym^2 V)$ by a hyperplane (defined by a quadratic form).
The following proposition provides an analogue of this description over any base.

\begin{proposition}
\label{prop:quadric-forms}
Let $p \colon X \to S$ be a smooth fibration in $3$-dimensional quadrics.
Then there is a vector bundle~$V$ of rank~$5$ on~$S$, a line bundle~$\cL$ on~$S$, 
and an epimorphism $\varphi \colon \Sym^2V \to \cL^\vee$ such that
\begin{equation*}
X = \P_S(V) \times_{\P_S(\Sym^2V)} \P_S(\Ker(\varphi)),
\end{equation*}
where the morphism $\P_S(V) \to \P_S(\Sym^2V)$ in the fiber product is the double Veronese embedding.
\end{proposition}

\begin{remark}
\label{rem:quadric-hx-untwisted}
Note that the vector bundle~$V$ and the line bundle~$\cL$ in this theorem are both untwisted.
Therefore, the fundamental class~$H_X$ of~$X/S$ can be represented by an (untwisted) line bundle~$\cO_X(H_X) \in \Pic(X)$.
\end{remark}

\begin{proof}
Let~$H_X \in \Pic_{X/S}(S)$ be the fundamental class.of~$X/S$ and set~$\upbeta \coloneqq \bB(H_X) \in \Br(S)$.
Then
\begin{equation*}
V \coloneqq (p_*\cO_{X/S}(H_X))^\vee
\end{equation*}
is a $\upbeta^{-1}$-twisted vector bundle on~$S$ of rank~5.
Lemma~\ref{lemma:brauer-obstruction} imples that the $p^*(\upbeta)$-twisted line bundle~$\cO_{X/S}(H_X)$ 
defines a closed embedding $X \subset \P_S(V)$ as a hypersurface of relative degree~$2$.

Let further $\cL \coloneqq p_{V*}\cI_X(2H_X)$, where $\cI_X$ is the ideal of~$X$ in~$\P_S(V)$ 
and~$p_V \colon \P_S(V) \to S$ is the natural projection.
Then~$\cL$ is a $\upbeta^2$-twisted line bundle on~$S$ and~$X$ has the prescribed form.

Finally, we conclude from Corollary~\ref{cor:twisted-rank} that~$\upbeta^2 = 1$, because the rank of~$\cL$ is~$1$,
and~$\upbeta^5 = 1$, because the rank of~$V$ is~5; 
combining these observations we see that~$\upbeta = 1$.
\end{proof}

\begin{remark}
The same argument works for any quadric fibration of odd relative dimension.
In the case of even relative dimension, the class~$\upbeta$ may be a non-trivial 2-torsion class.
\end{remark}

A semiorthogonal decomposition of the derived category of a smooth 3-dimensional quadric~$X$ over an algebraically closed field~$\kk$ 
has been described in~\cite[\S4]{Kap}; it takes the form
\begin{equation}
\label{eq:sod-quadric}
\bD(X) = \langle \cS \otimes \bD(\kk), \cO_X \otimes \bD(\kk), \cO_X(1) \otimes \bD(\kk), \cO_X(2) \otimes \bD(\kk) \rangle,
\end{equation}
where~$\cS$ is a spinor bundle. 
Note that the spinor bundle fits into an exact sequence 
\begin{equation}
\label{eq:exact-spinor}
0 \to \cS \to \cO_X^{\oplus 4} \to \cS(1) \to 0
\end{equation}
(see~\cite[Theorem~2.8]{Ott})
and it is stable (see~\cite[Theorem~2.1]{Ott}) and exceptional by~\eqref{eq:sod-quadric}.
Now we describe a relative analogue of~\eqref{eq:sod-quadric};
this result could be also extracted from~\cite{k2008quadrics}, 
but we provide an alternative argument to introduce the ideas used in other cases.

\begin{theorem}
\label{thm:db-quadric}
If $p \colon X \to S$ is a form of a smooth $3$-dimensional quadric over~$S$, there is a semiorthogonal decomposition
\begin{equation*}
\bD(X) = \langle \cS \otimes \bD(S,\upbeta_\cS^{-1}), \cO_X \otimes \bD(S), \cO_X(H_X) \otimes \bD(S), \cO_X(2H_X) \otimes \bD(S) \rangle,
\end{equation*}
where~$\cO_{X/S}(H_X)$ is a line bundle associated with the fundamental class of~$X/S$, 
see Remark~\textup{\ref{rem:quadric-hx-untwisted}}, 
$\upbeta_\cS \in \Br(S)$ is a $2$-torsion Brauer class, 
and~$\cS$ is a $p^*(\upbeta_\cS)$-twisted vector bundle of rank~$2$ on~$X$.

Moreover, if~$X(S) \ne \varnothing$ then~$\upbeta_\cS$ can be represented by a conic bundle.
\end{theorem}

\begin{proof}
Consider the moduli space 
\begin{equation*}
\rM \coloneqq \rM_{X/S,H_X}\left(\tfrac23t(t+1)(t+2)\right) = \rM_{X/S,H_X}(2; -H_X, L_X, 0),
\end{equation*}
where we use the convention of Remark~\ref{rem:moduli-chern-hilbert} in the right-hand side.
Let also~$\rM^\circ \subset \rM$ be the open subvariety parameterizing sheaves~$\cE$ on fibers~$X_s$ of~$X/S$ with the vanishing
\begin{equation*}
\rH^\bullet(X_s, \cE) = 0.
\end{equation*}
Applying Proposition~\ref{prop:uniqueness-general} to~$\cU = \cS_{X_s}$, the spinor bundle on the quadric~$X_s$, 
sequence~\eqref{eq:exact-spinor}, and a sheaf~$\cE$ as above 
(conditions~\eqref{eq:cu-ce-assumptions} are satisfied because~$\cE$ is numerically equivalent to~$\cU$),
we conclude that~$\cE \cong \cS_{X_s}$.
This proves that the natural morphism~$f \colon \rM^\circ \to S$ is bijective on geometric points.
On the other hand, since the spinor bundle~$\cS_{X_s}$ on~$X_s$ is exceptional, the morphism~$f$ 
is an isomorphism by Corollary~\ref{cor:moduli-etale}.

Furthermore, note that every sheaf parameterized by the moduli space~$\rM$ is $H$-stable.
Therefore, applying Proposition~\ref{prop:twisted-universal} and restricting to the open subscheme~$\rM^\circ \subset \rM$,
we obtain a Brauer class (which we denote~$\upbeta_\cS$) on~$\rM^\circ \cong S$
and a $p^*(\upbeta_\cS)$-twisted universal family (which we denote by~$\cS$) on~$X \times_S \rM^\circ = X$.

Note that the restriction of~$\cS$ to~$X_s$ is the spinor bundle of~$X_s$.
Therefore, by~\eqref{eq:sod-quadric} the bundles~$(\cS, \cO_X, \cO_X(H_X), \cO_X(2H_X))$ 
form a relative exceptional collection, 
hence the corresponding Fourier--Mukai functors are fully faithful 
and their images are semiorthogonal by Corollary~\ref{cor:relative-exceptional},
and applying Proposition~\ref{prop:relative-sod}\ref{item:linear-sod} 
we obtain the required semiorthogonal decomposition of~$\bD(X)$.

It remains to show that the Brauer class~$\upbeta_\cS \in \Br(S)$ is 2-torsion.
For this note that (up to twist by a line bundle on~$S$)
we have an isomorphism~$\wedge^2\cS \cong \cO(-H_X)$;
since the line bundle~$\cO(H_X)$ is untwisted by Proposition~\ref{prop:quadric-forms}, 
we conclude from Lemma~\ref{lemma:beta-kr} that~$\upbeta_\cS$ is indeed 2-torsion.

Finally, if~$X(S) \ne \varnothing$ and if~$i \colon S \to X$ is a section of~$X \to S$ 
then~$i^*\cS$ is a $\upbeta_\cS$-twisted vector bundle of rank~$2$ on~$S$, 
so that~$\upbeta_\cS$ is represented by the conic bundle~$\P_S(i^*\cS)$.
\end{proof}

\subsection{Forms of~$\sY_5$}

Now consider quintic del Pezzo threefolds.
Recall that over an algebraically closed field
every such threefold~$X$ can be represented as a complete intersection
\begin{equation*}
X = \Gr(2,V) \cap \P^6,
\end{equation*}
where~$V$ is a vector space of dimension~$5$ and the intersection is considered inside the Pl\"ucker space~$\P(\wedge^2 V) = \P^9$.
The following proposition provides an analogue over any base.

\begin{proposition}
\label{prop:forms-v5}
If $p \colon X \to S$ is smooth fibration in quintic del Pezzo threefolds, 
there are vector bundles~$V$ and~$A$ of respective ranks~$5$ and~$3$ on~$S$
and an epimorphism $\varphi \colon \wedge^2V \to A^\vee$ such that
\begin{equation}
X = \Gr_S(2,V) \times_{\P_S(\wedge^2V)} \P_S(\Ker(\varphi)),
\end{equation} 
where the morphism $\Gr_S(2,V) \to \P_S(\wedge^2V)$ in the fiber product is the Pl\"ucker embedding.
\end{proposition}

\begin{proof}
Let~$H_X \in \Pic_{X/S}(S)$ be the relative fundamental class.
Since over an algebraically closed field the Fano index of a quintic del Pezzo threefold is~2 
and the Euler characteristic of its fundamental divisor class is~7,
by Corollary~\ref{cor:fundamental-class} the Brauer class~$\bB(H_X)$ is annihilated by~$\gcd(2,7) = 1$, hence vanishes.
Therefore, the line bundle~$\cO_{X/S}(H_X)$ is untwisted.

Let $W \coloneqq (p_*\cO_{X/S}(H_X))^\vee$.
Then~$W$ is a vector bundle on~$S$ of rank~7, 
and the natural morphism~$X \to \P_S(W)$ is a closed embedding. 
For every geometric point $s$ of~$S$ the corresponding geometric fiber $X_s \subset \P(W_s) \cong \P^6$ 
is an intersection of a 5-dimensional space of Pl\"ucker quadrics, 
so if~$p_W \colon \P_S(W) \to S$ is the natural morphism and $\cI_X$ is the ideal sheaf of~$X$ in~$\P_S(W)$ then
\begin{equation*}
V \coloneqq p_{W*}\cI_X(2H_X)
\end{equation*}
is a vector bundle of rank~5 and the natural morphism $p_W^*V \to \cI_X(2H_X)$ is surjective.
Consider the restriction of this morphism to~$X$ and denote by~$\cU$ the kernel bundle, so that we have an exact sequence
\begin{equation*}
0 \to \cU \to p^*V \to \cI_X/\cI_X^2(2H_X) \to 0.
\end{equation*}
This bundle is, up to twist, the excess conormal bundle of~$X \subset \P_S(W)$ as defined in~\cite[Appendix~A]{DK18}.
The rank of~$\cU$ is~2, hence it defines a morphism~$X \to \Gr_S(2,V)$.
To relate it to the morphism~$X \to \P_S(W)$ defined above, we note that 
\begin{equation*}
\det(\cI_X/\cI_X^2) = \det(\cN^\vee_{X/\P_S(W)}) \cong \omega_{X/\P_S(W)}^\vee
\end{equation*}
is isomorphic to~$\cO_X(-5H_X)$ up to twist by a line bundle on~$S$, 
hence~$\wedge^2\cU^\vee$ is isomorphic to~$\cO(H_X)$ up to twist by a line bundle on~$S$.
Therefore, the canonical morphism~$\wedge^2(p^*V^\vee) \to \wedge^2\cU^\vee$ induces after pushforward to~$S$ 
a morphism~$\wedge^2V^\vee \to W^\vee \otimes \cL$ for an appropriate line bundle~$\cL$ on~$S$.
This morphism is surjective on each geometric fiber, hence it is an epimorphism.
Let~$A$ be its kernel bundle (of rank~3) and let~$\varphi \colon \wedge^2V \to A^\vee$ be the dual of the kernel morphism.
Then we obtain the equality~$X = \Gr_S(2,V) \times_{\P_S(\wedge^2V)} \P_S(\Ker(\varphi))$, as required.
\end{proof}

A semiorthogonal decomposition of the derived category of a quintic del Pezzo threefold~$X$ over an algebraically closed field~$\kk$ 
has been described in~\cite{Orlov:v5}; it takes the form
\begin{equation}
\label{eq:sod-y5}
\bD(X) = \langle \cO_X \otimes \bD(\kk), \cU^\vee \otimes \bD(\kk), \cO_X(1) \otimes \bD(\kk), \cU^\vee(1) \otimes \bD(\kk) \rangle,
\end{equation}
where~$\cU$ is the restriction of the tautological bundle from~$\Gr(2,V)$.
Now we describe a relative analogue of~\eqref{eq:sod-y5}.

\begin{theorem}
\label{thm:y5}
If $X/S$ is a form of a quintic del Pezzo threefold then there is a semiorthogonal decomposition
\begin{equation*}
\bD(X) = \langle \cO_X \otimes \bD(S), \cU^\vee \otimes \bD(S), \cO_X(H_X) \otimes \bD(S), \cU^\vee(H_X) \otimes \bD(S) \rangle,
\end{equation*}
where~$\cU$ is the vector bundle of rank~$2$ on~$X$ constructed in Proposition~\textup{\ref{prop:forms-v5}}. 
\end{theorem}

\begin{proof}
By~\eqref{eq:sod-y5} the collection~$(\cO_X,\cU^\vee,\cO_X(H_X),\cU^\vee(H_X))$ is a relative exceptional collection,
hence the theorem follows from Corollary~\ref{cor:relative-exceptional} and Proposition~\ref{prop:relative-sod}\ref{item:linear-sod}.
\end{proof}

\subsection{Forms of~$\sY_4$}

Now consider quartic del Pezzo threefolds.
Recall that over an algebraically closed field
every such threefold~$X$ can be represented as an intersection of two quadrics
\begin{equation*}
X = Q_1 \cap Q_2 \subset \P(V),
\end{equation*}
where~$V$ is a vector space of dimension~$6$,
or equivalently as a linear section of codimension~2 
of the second Veronese embedding~$\P(V) \hookrightarrow \P(\Sym^2 V)$.
The following proposition provides an analogue over any base.

\begin{proposition}
\label{prop:forms-v4}
If~$p \colon X \to S$ is a smooth fibration in quartic del Pezzo threefolds, 
there is a $2$-torsion Brauer class~$\upbeta \in \Br(S)$, 
a $\upbeta$-twisted vector bundle~$V$ of rank~$6$, an untwisted vector bundle~$A$ of rank~$2$,
and an epimorphism~$\varphi \colon \Sym^2V \to A^\vee$ such that
\begin{equation}
X = \P_S(V) \times_{\P_S(\Sym^2V)} \P_S(\Ker(\varphi)).
\end{equation} 
If~$X(S) \ne \varnothing$ then~$\upbeta = 1$.
\end{proposition}

\begin{proof}
Let~$H_X \in \Pic_{X/S}(S)$ be the relative fundamental class.
Since over an algebraically closed field the Fano index of a quartic del Pezzo threefold is~2,
we conclude from Corollary~\ref{cor:fundamental-class} that the Brauer obstruction~$\upbeta \coloneqq \bB(H_X)$ is 2-torsion.
Let 
\begin{equation*}
V \coloneqq (p_*\cO_{X/S}(H_X))^\vee.
\end{equation*}
Then~$V$ is a $\upbeta^{-1}$-twisted vector bundle on~$S$ of rank~6 
and by Lemma~\ref{lemma:brauer-obstruction} there is a closed embedding~$X \to \P_S(V)$ 
such that for every geometric point~$s$ of~$S$ the corresponding geometric fiber~$X_s \subset \P(V_s) \cong \P^5$ 
is an intersection of a pencil of quadrics, 
so if~$p_V \colon \P_S(V) \to S$ is the natural morphism and~$\cI_X$ is the ideal sheaf of~$X$ in~$\P_S(V)$ then
\begin{equation*}
A \coloneqq p_{V*}\cI_X(2H_X)
\end{equation*}
is a vector bundle of rank~2, and it is untwisted because~$\bB(2H_X) = \upbeta^2 = 1$.
Moreover, the pushforward of the natural morphism~$\cI_X(2H_X) \hookrightarrow \cO_X(2H_X)$ 
gives an embedding~$A \hookrightarrow \Sym^2V^\vee$.
Then we have~$X = \P_S(V) \times_{\P_S(\Sym^2V)} \P_S(\Ker(\varphi))$, 
where~$\varphi \colon \Sym^2V \to A^\vee$ is the dual of the embedding~$A \hookrightarrow \Sym^2V^\vee$.

Finally, if~$X(S) \ne \varnothing$ then~$\P_S(V) \to S$ has a section, hence~$\upbeta = 1$.
\end{proof}

A semiorthogonal decomposition of the derived category of a quartic del Pezzo threefold~$X$ 
over an algebraically closed field has been described in~\cite{BO95}.
We summarize their results in a slightly modified form that is convenient for our applications below.

\begin{proposition}
\label{prop:y4-moduli}
Let~$X$ be a quartic del Pezzo threefold over an algebraically closed field~$\kk$.
Consider the moduli space 
\begin{equation*}
\rM \coloneqq \rM_{X,H_X}\left(\tfrac23t(t+1)(2t+1)\right) = \rM_{X,H_X}(2; -H_X, 2L_X, 0)
\end{equation*}
and the open subscheme~$\rM^\circ \subset \rM$ parameterizing sheaves~$\cE$ on~$X$ such that
\begin{equation*}
\rH^\bullet(X, \cE) = \rH^\bullet(X, \cE(-1)) = 0.
\end{equation*}
Then~$\Gamma_X \coloneqq \rM^\circ$ is a smooth projective curve of genus~$2$, 
there exists a universal family~$\cU$ of sheaves on~$X \times \rM^\circ = X \times \Gamma_X$,
the Fourier--Mukai functor~\mbox{$\Phi_\cU \colon \bD(\Gamma_X) \to \bD(X)$} is fully faithful,  
and there is a semiorthogonal decomposition
\begin{equation}
\label{eq:sod-y4}
\bD(X) = \langle \Phi_\cU(\bD(\Gamma_X)), \cO_X \otimes \bD(\kk), \cO_X(1) \otimes \bD(\kk) \rangle.
\end{equation}
\end{proposition}

\begin{proof}
Let~$\Gamma_X \to \P^1$ be the double covering of the pencil of quadrics passing through~$X$
branched at the~6 points corresponding to singular quadrics in the pencil;
this is a smooth curve of genus~2.
A family~$\cU$ of vector bundles on~$X \times \Gamma_X$ has been constructed in~\cite[\S2]{BO95} (it is denoted there by~$S$), 
full faithfulness of the corresponding Fourier--Mukai functor was proved in~\cite[Theorem~2.7]{BO95}
and the semiorthogonal decomposition~\eqref{eq:sod-y4} was established in~\cite[Theorem~2.9]{BO95}.
So, we only need to provide the curve~$\Gamma_X$ and the bundle~$\cU$ with a modular interpretation.

For this we note that by~\cite[\S2]{BO95} the bundles~$\cU_y$, $y \in \Gamma_X$, in the above family 
are restrictions of spinor bundles from quadrics in the pencil defining~$X$,
and so their Chern classes have been computed in~\cite[Remark~2.9]{Ott},
and they match Chern classes in the definition of~$\rM$.
Furthermore, these bundles have the required cohomology vanishings 
(because of the semiorthogonal decomposition~\eqref{eq:sod-y4})
and are all stable (since stability of~$\cE$ is equivalent to the vanishing of~$\rH^0(X,\cE(-H_X))$).
Therefore, there is a unique morphism~$\Gamma_X \to \rM^\circ$ such that the family~$\cU$
is the pullback of a universal family
and it is an open immersion by Corollary~\ref{cor:moduli-curve}.

Now let~$\cE$ be a sheaf on~$X$ corresponding to a geometric point of~$\rM^\circ$.
Applying Proposition~\ref{prop:curve-general} to the exact sequences
\begin{equation*}
0 \to \cU_y \to \cO_X^{\oplus 4} \to \cO_X(1)^{\oplus 4} \to \cU_y(2) \to 0
\end{equation*}
(obtained from~\cite[Theorem~2.8]{Ott}) and using the cohomology vanishings in the definition of~$\cE$,
we conclude that if~$\cE \not\cong \cU_y$ for~$y \in \Gamma_X$ then~$\cE$ belongs to the orthogonal 
of the right-hand-side of~\eqref{eq:sod-y4}, which is impossible.
Thus $\cE \cong \cU_y$, hence the open immersion~$\Gamma_X \to \rM^\circ$ is surjective, hence it is an isomorphism.
\end{proof}

Now we describe a relative analogue of~\eqref{eq:sod-y4}.
Recall that~$\rF_1(X/S)$ denotes the Hilbert scheme of lines (with respect to~$H_X$) in the fibers of~$X \to S$.

\begin{theorem}
\label{thm:y4}
If $X/S$ is a form of a quartic del Pezzo threefold then there is a semiorthogonal decomposition
\begin{equation*}
\bD(X) = \langle \bD(\Gamma, \upbeta_\Gamma), \cO_X \otimes \bD(S), \cO_X(H_X) \otimes \bD(S,\upbeta) \rangle,
\end{equation*}
where~$\upbeta \in \Br(S)$ is the $2$-torsion Brauer class constructed in Proposition~\textup{\ref{prop:forms-v4}}, 
$\Gamma / S$ is a smooth projective family of curves of genus~$2$,
and~$\upbeta_\Gamma \in \Br(\Gamma)$ is a $4$-torsion Brauer class.

Moreover, if~$X(S) \ne \varnothing$ then~$\upbeta = 1$ and~$\upbeta_\Gamma^2 = 1$,
and if~$\rF_1(X/S)(S) \ne \varnothing$ then~$\upbeta_\Gamma = 1$.
\end{theorem}

\begin{proof}
Consider the same moduli space as in Proposition~\ref{prop:y4-moduli} but in the relative setting
\begin{equation*}
\rM \coloneqq \rM_{X/S,H_X}\left(\tfrac23t(t+1)(2t+1)\right) = \rM_{X/S,H_X}(2; -H_X, 2L_X, 0)
\end{equation*}
(where we use the convention of Remark~\ref{rem:moduli-chern-hilbert} in the right-hand side)
and its open subscheme~\mbox{$\rM^\circ \subset \rM$} 
parameterizing sheaves~$\cE$ on fibers~$X_s$ of~$X/S$ with the vanishing
\begin{equation*}
\rH^\bullet(X_s, \cE) = \rH^\bullet(X_s, \cE(-1)) = 0.
\end{equation*}
By Theorem~\ref{thm:moduli-existence} and Proposition~\ref{prop:y4-moduli} 
the geometric fibers of the morphism~$\rM^\circ \to S$ are the smooth projective curves~$\Gamma_{X_s}$ of genus~$2$ 
associated with the fibers~$X_s$ of~$X/S$.
Moreover, all sheaves parameterized by~$\rM^\circ$ have the form
\begin{equation*}
\cE_y = \Phi_{\cU_s}(\cO_y),
\end{equation*}
where~$y$ is a point of the smooth curve~$\Gamma_{X_s}$ 
and~$\cU_s$ is the universal bundle on~$X_s \times \Gamma_{X_s}$ from the proposition.
In particular, since the functor~$\Phi_{\cU_s}$ is fully faithful and~$\Gamma_{X_s}$ is smooth, we have 
\begin{equation*}
\Ext^2(\cE_y,\cE_y) \cong \Ext^2(\cO_y,\cO_y) = 0,
\qquad 
\Ext^1(\cE_y,\cE_y) \cong \Ext^1(\cO_y,\cO_y) = \rT_{y, \Gamma_{X_s}}.
\end{equation*}
Applying Theorem~\ref{thm:moduli-smooth} we see that the morphism~$\rM^\circ \to S$ 
is smooth of relative dimension~1, i.e., it is a smooth family of curves.
Since~$\rM^\circ$ is open in~$\rM$, and~$\rM$ is projective over~$S$, it follows that~$\rM^\circ$ is quasiprojective over~$S$.
Finally, since~$\rM^\circ \to S$ is a quasiprojective morphism with proper fibers, the morphism is projective.
From now on we will use the notation
\begin{equation*}
\Gamma \coloneqq \rM^\circ.
\end{equation*}
By construction this is a smooth projective family of curves of genus~$2$ over~$S$.

Further, as the proof of Proposition~\ref{prop:y4-moduli} shows, 
every sheaf parameterized by~$\rM^\circ$ is stable and locally free, 
hence by Proposition~\ref{prop:twisted-universal} there is a Brauer class~$\upbeta_\Gamma \in \Br(\Gamma)$ 
and a~$\pr_\Gamma^*(\upbeta_\Gamma)$-twisted universal bundle~$\cU$ on~$X \times_S \Gamma = X \times_S \rM^\circ$
(where~$\pr_\Gamma \colon X \times_S \Gamma \to \Gamma$ is the natural projection).
By construction, the restriction of~$\cU$ to every fiber of the morphism~$X \times_S \Gamma \to S$
is isomorphic (up to twist by a line bundle on the curve~$\Gamma_{X_s}$) 
to the bundle on~$X_s \times \Gamma_{X_s}$ from Proposition~\ref{prop:y4-moduli}.

Applying Proposition~\ref{prop:relative-sod}\ref{item:linear-ff} we conclude that the $S$-linear functor 
\begin{equation*}
\Phi_\cU \colon \bD(\Gamma, \upbeta_\Gamma) \to \bD(X)
\end{equation*}
is fully faithful and its image is admissible.
Furthermore, by~\eqref{eq:sod-y4} the pair~$(\cO_X,\cO_X(H_X))$ is relative exceptional,
and combining Corollary~\ref{cor:relative-exceptional} with Proposition~\ref{prop:relative-sod}\ref{item:linear-so}--\ref{item:linear-sod}
we obtain the required semiorthogonal decomposition of~$\bD(X)$.

Now by definition of~$\rM$ we have an isomorphism~$\wedge^2\cU \cong \cO(-H_X)$ on~$X \times_S \Gamma$ 
(up to twist by a line bundle on~$\Gamma$) and applying Lemma~\ref{lemma:beta-kr}
to deduce the 4-torsion property of~$\upbeta_\Gamma$ from the 2-torsion property of~$\upbeta$.
Similarly, if~$X/S$ has a section the equality~$\upbeta = 1$ (established in Proposition~\ref{prop:forms-v4})
implies~$\upbeta_\Gamma^2 = 1$.

Finally, assume that the morphism~$\rF_1(X/S) \to S$ has a section.
Then there is a relative line~$L \hookrightarrow X$,
i.e., an $S$-flat subscheme with the appropriate Hilbert polynomial.
Then, denoting by~$\pr_X \colon X \times_S \Gamma \to X$ and~$\pr_\Gamma \colon X \times_S \Gamma \to \Gamma$ the projections, 
it is easy to see that
\begin{equation*}
\cL \coloneqq \pr_{\Gamma *}(\cU \otimes \pr_X^*\cO_L) \in \bD(\Gamma,\upbeta_\Gamma)
\end{equation*}
is a line bundle.
Therefore, the Brauer class~$\upbeta_\Gamma$ vanishes by Corollary~\ref{cor:twisted-rank}.
\end{proof}

One could also construct the curve~$\Gamma$ directly, 
as an appropriate double covering of the projective bundle~$\P_S(A)$,
the projectivization of the vector bundle of rank~2 from Proposition~\ref{prop:forms-v4}.

\section{Prime Fano threefolds}
\label{sec:prime-fano}

In this section we describe smooth proper morphisms~$p \colon X \to S$ with fibers prime Fano threefolds.
In other words, we assume that~$\Pic(X_s) = \ZZ K_{X_s}$ 
(by Corollary~\ref{cor:constant} if this holds for one geometric point in~$S$, the same is true for any geometric point),
so that the fundamental class of~$X/S$ is equal to the relative anticanonical class, $H_X = -K_{X/S}$,
and generates the relative Picard group~$\Pic_{X/S}(S)$.
In particular, the line bundle~$\cO_X(H_X)$ is untwisted and canonically defined.

\subsection{Forms of~$\sX_{12}$}

Recall that over an algebraically closed field every prime Fano threefold~$X$ of genus~$12$ (type~$\sX_{12}$)
can be represented as the zero locus of a global section of the vector bundle~$(\wedge^2(\cU^\vee))^{\oplus 3}$ on~$\Gr(3,7)$
(we recall that~$\cU$ denotes the tautological bundle on the Grassmannian),
or equivalently, as the linear section
\begin{equation*}
X = \Gr(3,V) \cap \P^{13},
\end{equation*}
where~$V$ is a vector space of dimension~$7$, the intersection is considered inside the Pl\"ucker space~$\P(\wedge^3 V) = \P^{34}$,
and the subspace~$\P^{13} \subset \P^{34}$ is the projectivization of the kernel of the morphism~$\wedge^3(V) \to V^{\oplus 3}$
given by the global section of~$(\wedge^2(\cU^\vee))^{\oplus 3}$.

Furthermore, recall the full exceptional collection in the derived category of~$X$ (see~\cite{k1996v22})
\begin{equation}
\label{eq:sod-x22}
\bD(X) = \langle \cO_X, \cU^\vee, \cE, \wedge^2\cU^\vee \rangle
\end{equation}
where~$\cU$ is the restriction of the tautological bundle from Grassmannian (of rank~3) 
and~$\cE$ is an exceptional vector bundle of rank~2.
The following proposition provides an analogue of the description of~$X$ 
and the construction of (twisted) vector bundles on~$X$ for families over any base.

\begin{proposition}
\label{prop:x12}
If~$p \colon X \to S$ is a smooth Fano fibration with fibers of type~$\sX_{12}$
there is a vector bundle~$V$ of rank~$7$, a vector bundle~$A$ of rank~$3$, 
and an epimorphism~$\varphi \colon \wedge^2V \to A^\vee$ such that
\begin{equation}
\label{eq:x12}
X = \Gr_S(3,V) \times_{\P_S(\wedge^3V)} \P_S(\Ker(\tilde\varphi)),
\end{equation} 
where $\tilde\varphi \colon \wedge^3V \to V \otimes A^\vee$ is the morphism induced by~$\varphi$.
\end{proposition}

\begin{proof}
The main step in the proof is a construction of the vector bundle~$\cU$ on~$X$ (that would give a morphism to the Grassmannian);
as we will see the rest follows from a fiberwise description of~$X$.
As an intermediate step we construct a twisted bundle~$\cE$ of rank~$2$.

Consider the relative moduli space 
\begin{equation*}
\rM_2 \coloneqq \rM_{X/S}(2;H_X,7L_X,0)
\end{equation*}
where we use the convention of Remark~\ref{rem:moduli-chern-hilbert} in the right-hand side.
By~\cite[Theorem~B.1.1, Proposition~B.1.5]{KPS} the natural proper morphism~$f \colon \rM_2 \to S$ is bijective on geometric points
and for every geometric point~$[E] \in \rM_2$, the bundle~$E$ is exceptional by~\cite[Lemma~B.1.9]{KPS}.
Therefore, $f$ is an isomorphism by Corollary~\ref{cor:moduli-etale}.

Note that every sheaf parameterized by the moduli space~$\rM_2$ is $H_X$-stable.
Therefore, applying Proposition~\ref{prop:twisted-universal} 
we obtain a Brauer class~$\upbeta \in \Br(S)$ on~$\rM_2 \cong S$
and a $p^*(\upbeta)$-twisted universal family (which we denote by~$\cE$) on~$X \times_S \rM_2 = X$.
Since $\wedge^2\cE \cong \cO(H_X)$ (up to twist by a line bundle on~$S$);
and since the line bundle~$\cO(H_X)$ is untwisted, it follows from Lemma~\ref{lemma:beta-kr} that~$\upbeta^2 = 1$.

Similarly, consider the relative moduli space
\begin{equation*}
\rM_3 \coloneqq \rM_{X/S}(3; -H_X, 10L_X, -2P_X).
\end{equation*}
Let also $\rM_3^\circ \subset \rM_3$ be the open subscheme parameterizing bundles~$\cU$ on~$X_s$ with the vanishing
\begin{equation*}
\Ext^\bullet(\cE_s,\cU) = 0.
\end{equation*}
where~$\cE$ is the universal bundle constructed above.
By~\cite[Corollary~5.6]{KP20a} the natural morphism~$\rM_3^\circ \to S$ is bijective on geometric points 
and for every geometric point~$[U] \in \rM_3^\circ$, the bundle~$U$ is exceptional.
As before we conclude that~$\rM_3^\circ \to S$ is an isomorphism,
there is a Brauer class~$\upbeta_\cU \in \Br(\rM_3^\circ) = \Br(S)$ 
and a~$p^*(\upbeta_\cU)$-twisted universal bundle~$\cU$ on~$X \times_S \rM_3^\circ \cong X$.

Since $\wedge^3\cU \cong \cO(-H_X)$ (up to twist by a line bundle on~$S$), 
it follows from Lemma~\ref{lemma:beta-kr} that~\mbox{$\upbeta_\cU^3 = 1$}.
On the other hand, it follows from~\cite[Proposition~5.5]{KP20a} that
\begin{equation*}
V \coloneqq (p_*\cU^\vee)^\vee,
\end{equation*}
is a~$\upbeta_\cU$-twisted vector bundle on~$S$ of rank~7.
Therefore, $\upbeta_\cU^7 = 1$ by Corollary~\ref{cor:twisted-rank}.
Combining the above equalities for~$\upbeta_\cU$, we conclude that~$\upbeta_\cU = 1$, hence the bundles~$\cU$ and~$V$ are untwisted.

By~\cite[Proposition~5.5]{KP20a} the bundle~$\cU$ induces a closed embedding $X \to \Gr_S(3,V)$.
Moreover, if~$p_V \colon \Gr_S(3, V) \to S$ is the natural morphism and~$\cI_X$ is the ideal of~$X$ in~$\Gr_S(3,V)$, 
it follows that 
\begin{equation*}
A \coloneqq p_{V*}(\cI_X \otimes \wedge^2\cU^\vee)
\end{equation*}
is a vector bundle on~$S$ of rank~3, and if~$\varphi \colon \wedge^2V \to A^\vee$ 
denotes the dual morphism of the natural embedding~$A \hookrightarrow \wedge^2V^\vee$,
then~\eqref{eq:x12} holds.
\end{proof}

\begin{theorem}
\label{thm:x12}
If $X/S$ is a form of a prime Fano threefold of genus~$12$ then there is a semiorthogonal decomposition
\begin{equation*}
\bD(X) = \langle \cO_X \otimes \bD(S), \cU^\vee \otimes \bD(S), \cE \otimes \bD(S,\upbeta), \wedge^2\cU^\vee \otimes \bD(S) \rangle.
\end{equation*}
Moreover, if~$X(S) \ne \varnothing$ then~$\upbeta$ can be represented by a conic bundle.
\end{theorem}

\begin{proof}
By~\eqref{eq:sod-x22} the collection~$(\cO_X,\cU^\vee,\cE,\wedge^2\cU^\vee)$ is a relative exceptional collection,
hence the first part of the theorem follows from Corollary~\ref{cor:relative-exceptional} 
and Proposition~\ref{prop:relative-sod}\ref{item:linear-sod}.

If~$X(S) \ne \varnothing$ and if~$i \colon S \to X$ is a section of~$X \to S$ 
then~$i^*\cE$ is a $\upbeta$-twisted vector bundle of rank~$2$ on~$S$, 
so that~$\upbeta$ is represented by the conic bundle~$\P_S(i^*\cE)$.
\end{proof}

\subsection{Forms of~$\sX_{10}$}

Recall that over an algebraically closed field every prime Fano threefold~$X$ of genus~$10$ (type~$\sX_{10}$) 
can be represented as the zero locus of a global section of the vector bundle~\mbox{$\cU^\perp(1) \oplus \cO(1)^{\oplus 2}$} on~$\Gr(2,7)$,
or equivalently, as the linear section
\begin{equation*}
X = \Gr(2,V) \cap \P^{11},
\end{equation*}
where~$V$ is a vector space of dimension~$7$, the intersection is considered inside the Pl\"ucker space~$\P(\wedge^2 V) = \P^{20}$,
and the subspace~$\P^{11} \subset \P^{20}$ is the projectivization of the kernel of the morphism~$\wedge^2(V) \to V^\vee \oplus \kk^2$
given by the global section of~$\cU^\perp(1) \oplus \cO(1)^{\oplus 2}$ defining~$X$
(note that~$\rH^0(\Gr(2,V),\cU^\perp(1)) = \wedge^3V^\vee$, hence a global section of~$\cU^\perp(1)$ induces a map~$\wedge^2(V) \to V^\vee$).
The restriction to~$X$ of the tautological bundle~$\cU$ is called the {\sf Mukai bundle}.

The following proposition provides an analogue of the above description over any base.

\begin{proposition}
If~$p \colon X \to S$ is a smooth Fano fibration with fibers of type~$\sX_{10}$
there is a vector bundle~$V$ of rank~$7$, a vector bundle~$A$ of rank~$2$, a line bundle~$\cL$,
and epimorphisms~\mbox{$\varphi_1 \colon \wedge^3V \to \cL^\vee$} 
and~\mbox{$\varphi_2 \colon \Ker(\wedge^2V \xrightarrow{\ \varphi_1\ } V^\vee \otimes \cL^\vee) \to A^\vee$} such that
\begin{equation}
\label{eq:x10}
X = \Gr_S(2,V) \times_{\P_S(\wedge^2V)} \P_S(\Ker(\varphi_2)).
\end{equation} 
\end{proposition}

\begin{proof}
Consider the relative moduli space 
\begin{equation*}
\rM \coloneqq \rM_{X/S}(2;H_X,6L_X,0)
\end{equation*}
where we use the convention of Remark~\ref{rem:moduli-chern-hilbert} in the right-hand side.
By~\cite[Theorem~B.1.1, Proposition~B.1.5]{KPS} the natural projective morphism~$f \colon \rM \to S$ is bijective on geometric points
and for every geometric point~$[U] \in \rM$, the bundle~$U$ is exceptional by~\cite[Lemma~B.1.9]{KPS}.
Therefore, $f$ is an isomorphism by Corollary~\ref{cor:moduli-etale}.

Note that every sheaf parameterized by the moduli space~$\rM$ is $H_X$-stable.
Therefore, applying Proposition~\ref{prop:twisted-universal} 
we obtain a Brauer class~$\upbeta \in \Br(S)$ on~$\rM \cong S$
and a $p^*(\upbeta)$-twisted universal family (which we denote by~$\cU$) on~$X \times_S \rM = X$.
Since $\wedge^2\cU \cong \cO(-H_X)$ (up to twist by a line bundle on~$S$) and
since the line bundle~$\cO(H_X)$ is untwisted, Lemma~\ref{lemma:beta-kr} implies~$\upbeta^2 = 1$.

On the other hand, 
\begin{equation*}
V \coloneqq (p_*\cU^\vee)^\vee,
\end{equation*}
is a~$\upbeta$-twisted vector bundle on~$S$ of rank~7.
Therefore, $\upbeta^7 = 1$ by Corollary~\ref{cor:twisted-rank}.
Combining the above equalities for~$\upbeta$, we conclude that~$\upbeta = 1$, hence the bundles~$\cU$ and~$V$ are untwisted.

The bundle~$\cU$ induces a closed embedding $X \to \Gr_S(2,V)$.
Moreover, it follows that 
\begin{equation*}
\cL \coloneqq p_{V*}(\cI_X \otimes \cU^\perp(1))
\qquad\text{and}\qquad 
A \coloneqq p_{V*}(\cI'_X \otimes \cO(1))
\end{equation*}
are vector bundles of respective ranks~1 and~2, where $\cI_X$ is the ideal of~$X$ in~$\Gr_S(2,V)$,
and $\cI'_X$ is the ideal of~$X$ in~$\Gr_S(2,V) \times_{\P_S(\wedge^2V)} \P_S(\Ker(\wedge^2V \to V^\vee \otimes \cL^\vee))$,
and~$p_V \colon \Gr_S(2,V) \to S$ is the natural projection.
Now~\eqref{eq:x10} easily follows.
\end{proof}

A semiorthogonal decomposition of the derived category of prime Fano threefolds~$X$ of genus~$10$
over an algebraically closed field has been described in~\cite[\S6.4]{K06};
we summarize it in a form that is convenient for our applications below.

\begin{proposition}
\label{prop:x10-moduli}
Let~$X$ be a prime Fano threefold of genus~$10$ over an algebraically closed field~$\kk$.
Let~$\cU$ be the Mukai bundle on~$X$.
Consider the moduli space 
\begin{equation*}
\rM \coloneqq 
\rM_{X,H_X}(3; -H_X, 9L_X, -2P_X)
\end{equation*}
and the open subscheme~$\rM^\circ \subset \rM$ parameterizing sheaves~$\cE$ on~$X$ such that
\begin{equation*}
\rH^\bullet(X, \cE) = \Ext^\bullet(\cU^\vee, \cE) = 0.
\end{equation*}
Then~$\Gamma_X \coloneqq \rM^\circ$ is a smooth projective curve of genus~$2$,
there exists a universal family~$\cE$ of sheaves on~$X \times \rM^\circ = X \times \Gamma_X$,
the Fourier--Mukai functor~$\Phi_\cE \colon \bD(\Gamma_X) \to \bD(X)$ is fully faithful,  
and there is a semiorthogonal decomposition
\begin{equation}
\label{eq:sod-x10}
\bD(X) = \langle \Phi_\cE(\bD(\Gamma_X)), \cO_X \otimes \bD(\kk), \cU^\vee \otimes \bD(\kk) \rangle.
\end{equation}
\end{proposition}

\begin{proof}
By~\cite[\S6.4 and~\S8]{K06}, see also~\cite[\S B.5]{KPS},
there is a smooth curve~$\Gamma_X$ of genus~$2$ and a $\Gamma_X$-flat family~$\cE$ 
of stable vector bundles on~$X$ with the same rank and Chern classes 
as in the definition of the moduli space~$\rM$ (see also~\cite[Remark~B.5.3]{KPS}),
such that the Fourier--Mukai functor
\begin{equation*}
\Phi_\cE \colon \bD(\Gamma_X) \to \bD(X)
\end{equation*}
is fully faithful and its image together with the exceptional vector bundles~$\cO_X$ and~$\cU^\vee$ 
gives the semiorthogonal decomposition~\eqref{eq:sod-x10}.
It remains to provide the curve~$\Gamma_X$ and the bundle~$\cE$ with a modular interpretation.

As we already observed, the bundles~$\cE_y$ parameterized by the curve~$\Gamma_X$ have the correct Chern classes and stable.
Moreover, \eqref{eq:sod-x10} implies that they satisfy the vanishing conditions defining~$\rM^\circ$.
Therefore, there is a morphism~$\Gamma_X \to \rM^\circ $ such that~$\cE$ is the pullback of a universal family.
Applying Corollary~\ref{cor:moduli-curve} we conclude it is an open immersion.
On the other hand, let~$E$ be a sheaf on~$X$ corresponding to a geometric point of~$\rM^\circ$.
Applying Proposition~\ref{prop:curve-general} to the exact sequences
\begin{equation}
\label{eq:sequence-ce-x10}
0 \to \cE_y \to \cO_X^{\oplus 6} \to ({\cU^\vee})^{\oplus 3} \to \cE_y(1) \to 0
\end{equation}
(see~\cite[(B.5.2)]{KPS}) and using the cohomology vanishings in the definition of~$\rM^\circ$,
we conclude that if~$E \not\cong \cE_y$ for~$y \in \Gamma_X$ then~$E$ belongs to the orthogonal 
of the right-hand-side of~\eqref{eq:sod-x10}, which is impossible.
Thus $E \cong \cE_y$, hence the open immersion~$\Gamma_X \to \rM^\circ$ is surjective, hence it is an isomorphism.
\end{proof}

Recall that~$\rF_2(X/S)$ denotes the relative Hilbert scheme of conics on~$X/S$.

\begin{theorem}
\label{thm:x10}
If $X/S$ is a form of a prime Fano threefold of genus~$10$ then there is a semiorthogonal decomposition
\begin{equation*}
\bD(X) = \langle \bD(\Gamma,\upbeta_\Gamma), \cO_X \otimes \bD(S), \cU^\vee \otimes \bD(S) \rangle,
\end{equation*}
where $\Gamma/S$ is a smooth family of curves of genus~$2$ and $\upbeta_\Gamma \in \Br(\Gamma)$ is a $3$-torsion Brauer class.

Moreover, if the natural morphism~$\rF_2(X/S) \to S$ has a section then~$\upbeta_\Gamma = 1$.
\end{theorem}

\begin{proof}
The proof of the first part is analogous to the proof of Theorem~\ref{thm:y4},
with Proposition~\ref{prop:y4-moduli} replaced by Proposition~\ref{prop:x10-moduli}. 
To prove the second part, assume the natural morphism~$\rF_2(X/S) \to S$ has a section. 
Then there is a conic~$C \hookrightarrow X$,
i.e., an $S$-flat subscheme with the appropriate Hilbert polynomial.
Then, denoting by~$\pr_X \colon X \times_S \Gamma \to X$ and~$\pr_\Gamma \colon X \times_S \Gamma \to \Gamma$ the projections, 
one can deduce from the proof of~\cite[Lemma~B.5.4]{KPS} that 
\begin{equation*}
\cL \coloneqq \pr_{\Gamma *}(\cE \otimes \pr_X^*\cO_C) \in \bD(\Gamma,\upbeta_\Gamma)
\end{equation*}
is a line bundle.
Therefore, the Brauer class~$\upbeta_\Gamma$ vanishes by Corollary~\ref{cor:twisted-rank}.
\end{proof}

\subsection{Forms of~$\sX_{9}$}

Recall that over an algebraically closed field every prime Fano threefold~$X$ of genus~$9$ (type~$\sX_9$)
can be represented as a linear section 
\begin{equation*}
X = \LGr(3,V) \cap \P^{10},
\end{equation*}
where~$V$ is a vector space of dimension~$6$ endowed with a symplectic form
and the intersection is considered inside the space~$\P(\Ker(\wedge^3 V \to V)) = \P^{13}$,
where the morphism is induced by the symplectic form.
The restriction~$\cU$ of the tautological bundle from~$\LGr(3,V) \subset \Gr(3,V)$ to~$X$ is called {\sf the Mukai bundle} of~$X$.
The following observation is crucial for the results of this section.

\begin{lemma}
\label{lemma:x16-uniqueness}
Let~$X$ be a prime Fano threefold of genus~$9$ over an algebraically closed field of characteristic zero.
The Mukai bundle~$\cU$ on~$X$ is stable with~$\rc_1 = -H_X$, $\rc_2 = 8L_X$, $\rc_3 = -2P_X$ and the pair~$(\cU,\cO_X)$ is exceptional.
Moreover, any semistable vector bundle~$U$ of rank~$3$ on~$X$ with~$\rH^2(X,U) = 0$ and~$\rc_i(U) = \rc_i(\cU)$, $1 \le i \le 3$,
is isomorphic to the Mukai bundle.
\end{lemma}

\begin{proof}
Exceptionality of the Mukai bundle~$\cU$ and semiorthogonality of the pair~$(\cU, \cO_X)$ is proved in~\cite[Lemma~7.1]{K06}.
Stability of~$\cU$ is equivalent to the vanishings
\begin{equation*}
\rH^0(X,\cU) = 0
\qquad\text{and}\qquad
\rH^0(X,\wedge^2\cU) = \rH^0(X,\cU^\vee(K_X)) = \rH^3(X,\cU)^\vee = 0,
\end{equation*}
which follow from semiorthogonality of the pair~$(\cU, \cO_X)$.
The computation of Chern classes of~$\cU$ can be performed on the Lagrangian Grassmannian and is straightforward. 

Now let~$U$ be a semistable vector bundle of rank~$3$ with the same Chern classes.
Applying Proposition~\ref{prop:uniqueness-general} to~$\cE = U$ and~$\cU$
(conditions~\eqref{eq:cu-ce-assumptions} are satisfied because~$\cE$ is numerically equivalent to~$\cU$),
and the exact sequence
\begin{equation}
\label{eq:sequence-cu-x9}
0 \to\cU \to \cO_X^{\oplus 6} \to \cU^\vee \to 0
\end{equation}
obtained by restriction from~$\LGr(3,6)$, we conclude that~$U \cong \cU$.
\end{proof}

The following proposition provides a description of families of prime Fano threefolds of genus~$9$ over any base.

\begin{proposition}
\label{prop:x16-forms}
If~$p \colon X \to S$ is a smooth fibration with fibers of type~$\sX_9$
there is a vector bundle~$V$ of rank~$6$, a vector bundle~$A$ of rank~$3$, a line bundle~$\cL$,
and epimorphisms~\mbox{$\varphi_1 \colon \wedge^2V \to \cL^\vee$} 
and~\mbox{$\varphi_2 \colon \Ker(\wedge^3V \xrightarrow{\ \varphi_1\ } V \otimes \cL^\vee) \to A^\vee$} such that
\begin{equation}
\label{eq:x9}
X = \Gr_S(3,V) \times_{\P_S(\wedge^3V)} \P_S(\Ker(\varphi_2)).
\end{equation} 
\end{proposition}

\begin{proof}
Consider the relative moduli space 
\begin{equation*}
\rM \coloneqq \rM_{X/S}(3; -H_X, 8L_X, -2P_X)
\end{equation*}
where we use the convention of Remark~\ref{rem:moduli-chern-hilbert} in the right-hand side.
Let also $\rM^\circ \subset \rM$ be the open subscheme parameterizing bundles~$U$ on~$X_s$ with the vanishing
\begin{equation*}
\rH^\bullet(X_s,U) = 0.
\end{equation*}
By Lemma~\ref{lemma:x16-uniqueness} the natural morphism $f \colon \rM^\circ \to S$ is bijective on geometric points
and for every geometric point~$[U] \in \rM^\circ$, the bundle~$U$ is exceptional.
Therefore, $f$ is an isomorphism by Corollary~\ref{cor:moduli-etale}.

Note that every sheaf parameterized by the moduli space~$\rM^\circ$ is $H_X$-stable.
Therefore, applying Proposition~\ref{prop:twisted-universal} 
we obtain a Brauer class~$\upbeta \in \Br(S)$ on~$\rM^\circ \cong S$
and a $p^*(\upbeta)$-twisted universal family (which we denote by~$\cU$) on~$X \times_S \rM^\circ = X$.
Since $\wedge^3\cU \cong \cO(-H_X)$ (up to twist by a line bundle on~$S$) and
since the line bundle~$\cO(H_X)$ is untwisted, Lemma~\ref{lemma:beta-kr} implies~$\upbeta^3 = 1$.

On the other hand, let
\begin{equation*}
V \coloneqq (p_*\cU^\vee)^\vee;
\end{equation*}
this is a~$\upbeta$-twisted vector bundle on~$S$ of rank~6.
Applying Lemma~\ref{lemma:beta-2} and~\eqref{eq:sequence-cu-x9}, we obtain~$\upbeta^2 = 1$.
Combining the above equalities for~$\upbeta$, we deduce~$\upbeta = 1$, hence the bundles~$\cU$ and~$V$ are untwisted.

The bundle~$\cU$ induces a closed embedding~$X \hookrightarrow \Gr_S(3,V)$.
Moreover, it follows that 
\begin{equation*}
\cL \coloneqq p_{V*}(\cI_X \otimes \wedge^2\cU^\vee)
\qquad\text{and}\qquad 
A \coloneqq p_{V*}(\cI'_X \otimes \cO(H_X))
\end{equation*}
are vector bundles of respective ranks~1 and~3, where $\cI_X$ is the ideal of~$X$ in~$\Gr_S(3,V)$,
and $\cI'_X$ is the ideal of~$X$ in~$\Gr_S(3,V) \times_{\P_S(\wedge^2V)} \P_S(\Ker(\wedge^3V \to V \otimes \cL^\vee))$,
and~$p_V \colon \Gr_S(3,V) \to S$ is the natural projection.
Now~\eqref{eq:x9} easily follows.
\end{proof}

A semiorthogonal decomposition of the derived category of prime Fano threefolds~$X$ of genus~$9$
over an algebraically closed field has been described in~\cite[\S6.3]{K06};
we summarize it in a form that is convenient for our applications below.

\begin{proposition}
\label{prop:x9-moduli}
Let~$X$ be a prime Fano threefold of genus~$9$ over an algebraically closed field~$\kk$.
Let~$\cU$ be the Mukai bundle on~$X$.
Consider the moduli space 
\begin{equation*}
\rM \coloneqq 
\rM_{X,H_X}(2; -H_X, 6L_X, 0)
\end{equation*}
and the open subscheme~$\rM^\circ \subset \rM$ parameterizing sheaves~$\cE$ on~$X$ such that
\begin{equation*}
\rH^\bullet(X, \cE) = \Ext^\bullet(\cU^\vee, \cE) = 0.
\end{equation*}
Then~$\Gamma_X \coloneqq \rM^\circ$ is a smooth projective curve of genus~$3$,
there exists a universal family~$\cE$ of sheaves on~$X \times \rM^\circ = X \times \Gamma_X$,
the Fourier--Mukai functor~$\Phi_\cE \colon \bD(\Gamma_X) \to \bD(X)$ 
is fully faithful,  
and there is a semiorthogonal decomposition
\begin{equation}
\label{eq:sod-x9}
\bD(X) = \langle \Phi_\cE(\bD(\Gamma_X)), \cO_X \otimes \bD(\kk), \cU^\vee \otimes \bD(\kk) \rangle.
\end{equation}
\end{proposition}

\begin{proof}
By~\cite[\S6.3 and~\S7]{K06}, see also~\cite[\S9.1]{KP19},
there is a smooth curve~$\Gamma_X$ of genus~$3$ and a $\Gamma_X$-flat family~$\cE$ 
of stable vector bundles on~$X$ with the same rank and Chern classes 
as in the definition of the moduli space~$\rM$ 
such that the Fourier--Mukai functor
\begin{equation*}
\Phi_\cE \colon \bD(\Gamma_X) \to \bD(X)
\end{equation*}
is fully faithful and its image together with the exceptional vector bundles~$\cO_X$ and~$\cU^\vee$ 
gives the semiorthogonal decomposition~\eqref{eq:sod-x9}.
The rest of the proof is analogous to that of Proposition~\ref{prop:x10-moduli} 
with~\eqref{eq:sequence-ce-x10} replaced by the exact sequence
\begin{equation}
\label{eq:sequence-ce-x9}
0 \to \cE_y \to \cO_X^{\oplus 6} \to ({\cU^\vee})^{\oplus 2} \to \cE_y(H_X) \to 0
\end{equation}
(see~\cite[(9.1.4)]{KP19}).
\end{proof}

Recall that~$\rF_3(X/S)$ denotes the relative Hilbert scheme of rational cubic curves on~$X/S$.

\begin{theorem}
\label{thm:x9}
If $X/S$ is a form of a prime Fano threefold of genus~$9$ then there is a semiorthogonal decomposition
\begin{equation*}
\bD(X) = \langle \bD(\Gamma,\upbeta_\Gamma), \cO_X \otimes \bD(S), \cU^\vee \otimes \bD(S) \rangle,
\end{equation*}
where $\Gamma/S$ is a smooth curve of genus~$3$ and $\upbeta_\Gamma \in \Br(\Gamma)$ is a $2$-torsion Brauer class.

Moreover, if the natural morphism~$\rF_3(X/S) \to S$ has a section then~$\upbeta_\Gamma = 1$.
\end{theorem}

\begin{proof}
The proof of the first part is analogous to the proof of Theorem~\ref{thm:y4},
with Proposition~\ref{prop:y4-moduli} replaced by Proposition~\ref{prop:x9-moduli}. 
To prove the second part assume the natural morphism~$\rF_3(X/S) \to S$ has a section. 
Then there is a rational cubic curve~$C \hookrightarrow X$,
i.e., an $S$-flat subscheme with the appropriate Hilbert polynomial.
Then, denoting by~$\pr_X \colon X \times_S \Gamma \to X$ and~$\pr_\Gamma \colon X \times_S \Gamma \to \Gamma$ the projections, 
one can deduce from~\eqref{eq:sequence-ce-x9} that 
\begin{equation*}
\cL \coloneqq \pr_{\Gamma *}(\cE \otimes \pr_X^*\cO_C) \in \bD(\Gamma,\upbeta_\Gamma)
\end{equation*}
is a line bundle.
Therefore, the Brauer class~$\upbeta_\Gamma$ vanishes by Corollary~\ref{cor:twisted-rank}.
\end{proof}

\subsection{Forms of~$\sX_{7}$}

Recall that over an algebraically closed field
every prime Fano threefold~$X$ of genus~$7$ (type~$\sX_7$)
can be represented as a linear section
\begin{equation*}
X = \OGr_+(5,V) \cap \P^{8},
\end{equation*}
where~$V$ is a vector space of dimension~$10$ endowed with a non-degenerate quadratic form, 
$\OGr_+(5,V)$ is one (of the two) connected component of the Grassmannian of 5-dimensional isotropic subspaces in~$V$,
and the intersection is considered inside the half-spinor space~\mbox{$\P(\cS_+) = \P^{15}$}.
The restriction~$\cU$ of the tautological bundle from~$ \OGr_+(5,V) \subset \Gr(5,V)$ to~$X$ is called {\sf the Mukai bundle} of~$X$.
The following observation is crucial for the results of this section.

\begin{lemma}
\label{lemma:x12-uniqueness}
Let~$X$ be a prime Fano threefold of genus~$7$ over an algebraically closed field of characteristic zero.
The Mukai bundle~$\cU$ on~$X$ is stable with~$\rc_1 = -2H_X$, $\rc_2 = 24L_X$, $\rc_3 = -14P_X$ 
and the pair~$(\cU,\cO_X)$ is exceptional.
Moreover, any semistable vector bundle~$U$ of rank~$5$ on~$X$ with~$\rH^2(X,U) = 0$ and~$\rc_i(U) = \rc_i(\cU)$, $1 \le i \le 3$,
is isomorphic to the Mukai bundle.
\end{lemma}

\begin{proof}
Exceptionality of the Mukai bundle~$\cU$ and semiorthogonality of the pair~$(\cU, \cO_X)$ 
are easy to prove by Borel--Bott--Weil theorem (see~\cite[Lemma~3.1]{k2005v12}).
Stability of~$\cU$ is equivalent to the vanishings
\begin{equation*}
\rH^0(X,\cU) = \rH^0(X,\wedge^2\cU) = \rH^0(X,\wedge^3\cU(H_X)) = \rH^0(X,\wedge^4\cU(H_X)) = 0,
\end{equation*}
which can be proved by a similar computation.
The computation of Chern classes of~$\cU$ can be performed on the orthogonal Grassmannian and is straightforward. 

Now let~$U$ be a semistable vector bundle of rank~$5$ with the same Chern classes.
Applying Proposition~\ref{prop:uniqueness-general} to~$\cE = U$ and~$\cU$
(conditions~\eqref{eq:cu-ce-assumptions} are satisfied because~$\cE$ is numerically equivalent to~$\cU$),
and the exact sequence
\begin{equation}
\label{eq:sequence-cu-x7}
0 \to\cU \to \cO_X^{\oplus 10} \to \cU^\vee \to 0
\end{equation}
obtained by restriction from~$\OGr_+(5,10)$, we conclude that~$U \cong \cU$.
\end{proof}

The following proposition provides a description of families of prime Fano threefolds of genus~$7$ over any base.

\begin{proposition}
If~$p \colon X \to S$ is a smooth fibration with fibers of type~$\sX_7$
there is a vector bundle~$V$ of rank~$10$, a line bundle~$\cL$,
and epimorphisms~\mbox{$\varphi_1 \colon \Sym^2V \to \cL^\vee$} such that~\mbox{$X \subset \OGr_S(5,V)$},
where the orthogonal Grassmannian is considered with respect to the family of quadratic forms~$\varphi_1$.
Furthermore, the canonical double covering over~$S$ induced by the Stein factorization of the morphism~$\OGr_S(5,V) \to S$ splits,
and the relative Pl\"ucker class on the component~$\OGr_{S,+}(5,V)$ of~$\OGr_S(5,V)$ containing~$X$, 
is divisible by~$2$ in~$\Pic_{\OGr_{S,+}(5,V)/S}(S)$.
Finally, there is a vector bundle~$A$ of rank~$7$
and an epimorphism~\mbox{$\varphi_2 \colon \cS_+ \to A^\vee$} such that
\begin{equation}
\label{eq:x7}
X = \OGr_{S,+}(5,V) \times_{\P_S(\cS_+)} \P_S(\Ker(\varphi_2)),
\end{equation} 
where $\cS_+$ is the half-spinor bundle of rank~$16$ over~$S$ 
obtained as the pshforward to~$S$ of the line bundle~$\cO(H)$, 
where~$H \in \Pic_{\OGr_{S,+}(5,V)/S}(S)$ is the half of the relative Pl\"ucker class on~$\OGr_{S,+}(5,V)$.
\end{proposition}

\begin{proof}
Consider the relative moduli space 
\begin{equation*}
\rM \coloneqq \rM_{X/S}(5; -2H_X, 24L_X, -14P_X)
\end{equation*}
where we use the convention of Remark~\ref{rem:moduli-chern-hilbert} in the right-hand side.
Let also $\rM^\circ \subset \rM$ be the open subscheme parameterizing bundles~$U$ on~$X_s$ with the vanishing
\begin{equation*}
\rH^\bullet(X_s,U) = 0.
\end{equation*}
By Lemma~\ref{lemma:x12-uniqueness} the natural morphism $f \colon \rM^\circ \to S$ is bijective on geometric points
and for every geometric point~$[U] \in \rM^\circ$, the bundle~$U$ is exceptional.
Therefore, $f$ is an isomorphism by Corollary~\ref{cor:moduli-etale}.

Note that every sheaf parameterized by the moduli space~$\rM^\circ$ is $H_X$-stable.
Therefore, applying Proposition~\ref{prop:twisted-universal} 
we obtain a Brauer class~$\upbeta \in \Br(S)$ on~$\rM^\circ \cong S$
and a $p^*(\upbeta)$-twisted universal family (which we denote by~$\cU$) on~$X \times_S \rM^\circ = X$.
Since $\wedge^5\cU \cong \cO(-2H_X)$ (up to twist by a line bundle on~$S$) and
since the line bundle~$\cO(2H_X)$ is untwisted, Lemma~\ref{lemma:beta-kr} implies~$\upbeta^5 = 1$. 

On the other hand, let
\begin{equation*}
V \coloneqq (p_*\cU^\vee)^\vee;
\end{equation*}
this is a~$\upbeta$-twisted vector bundle on~$S$ of rank~10.
Applying Lemma~\ref{lemma:beta-2} and~\eqref{eq:sequence-cu-x7}, we obtain~\mbox{$\upbeta^2 = 1$}.
Combining the above equalities for~$\upbeta$, we deduce~\mbox{$\upbeta = 1$}, hence the bundles~$\cU$ and~$V$ are untwisted.

The bundle~$\cU$ defines a closed embedding~$X \hookrightarrow \Gr_S(5,V)$.
Moreover, it follows that 
\begin{equation*}
\cL \coloneqq p_{V*}(\cI_X \otimes \Sym^2\cU^\vee)
\end{equation*}
is a line bundle, where $\cI_X$ is the ideal of~$X$ in~$\Gr_S(5,V)$ and~$p_V \colon \Gr_S(5,V) \to S$ is the natural projection.
Furthermore, we have a canonical morphism $\cL \to p_{V*}\Sym^2\cU^\vee = \Sym^2V^\vee$, which can be considered as a family of quadratic forms.
Note that this family is everywhere non-degenerate.

Let $\OGr_S(5,V) \subset \Gr_S(5,V)$ be the zero locus of the natural section of $p_V^*\cL^\vee \otimes \Sym^2\cU^\vee$;
this is a family of orthogonal Grassmannians for the above family of quadratic forms.
Let 
\begin{equation*}
p'_V \colon \OGr_S(5,V) \to \tilde{S} \to S
\end{equation*}
be the Stein factorization; so that $\tilde{S} \to S$ is an \'etale double covering.
Since $X \subset \OGr_S(5,V)$, the natural morphism $p'_{V*}\cO_{\Gr_S(5,V)} \to p_*\cO_X \cong \cO_S$
gives a regular section of this double covering, hence the covering splits.

Let~$H$ be the fundamental class of the Fano fibration~$p^+_V \colon \OGr_{S,+}(5,V) \to S$ and set~$\upbeta \coloneqq \bB(H)$.
Since the fibers of~$p^+_V$ have index~$8$, it follows from Corollary~\ref{cor:fundamental-class} that~$\upbeta^8 = 1$.
Let 
\begin{equation*}
\cS_+ \coloneqq (p^+_{V*}\cO(H))^\vee;
\end{equation*}
this a $\upbeta$-twisted vector bundle of rank~16 on~$S$.
Let, furthermore,
\begin{equation*}
A \coloneqq p^+_{V*}(\cI'_X \otimes \cO(H)),
\end{equation*}
where $\cI'_X$ is the ideal of~$X$ on~$\OGr_{S,+}(5,V)$;
this a $\upbeta$-twisted vector bundle of rank~7 on~$S$.
The existence of~$A$ implies that~$\upbeta^7 = 1$;
combining this with the previous observation we deduce~$\upbeta = 1$,
hence the bundles~$\cS_+$ and~$A$ are untwisted.
Finally, the required formula~\eqref{eq:x7} also follows.
\end{proof}

\begin{remark}
One could also construct the bundle~$\cU$ on~$X$ as the twisted normal bundle
for the relative anticanonical embedding of~$X$, similarly to the proof of Proposition~\ref{prop:forms-v5}.
\end{remark}

A semiorthogonal decomposition of the derived category of prime Fano threefolds~$X$ of genus~$7$
over an algebraically closed field has been described in~\cite{k2005v12} and~\cite[\S6.2]{K06};
we summarize it in a form that is convenient for our applications below.

\begin{proposition}
\label{prop:x7-moduli}
Let~$X$ be a prime Fano threefold of genus~$7$ over an algebraically closed field~$\kk$.
Let~$\cU$ be the Mukai bundle on~$X$.
Consider the moduli space 
\begin{equation*}
\rM \coloneqq 
\rM_{X,H_X}(2; -H_X, 5L_X, 0)
\end{equation*}
and the open subscheme~$\rM^\circ \subset \rM$ parameterizing sheaves~$\cE$ on~$X$ such that
\begin{equation*}
\rH^\bullet(X, \cE) = \Ext^\bullet(\cU^\vee, \cE) = 0.
\end{equation*}
Then~$\Gamma_X \coloneqq \rM^\circ$ is a smooth projective curve of genus~$7$,
there exists a universal family~$\cE$ of sheaves on~$X \times \rM^\circ = X \times \Gamma_X$,
the Fourier--Mukai functor~$\Phi_\cE \colon \bD(\Gamma_X) \to \bD(X)$ 
is fully faithful,  
and there is a semiorthogonal decomposition
\begin{equation}
\label{eq:sod-x7}
\bD(X) = \langle \Phi_\cE(\bD(\Gamma_X)), \cO_X \otimes \bD(\kk), \cU^\vee \otimes \bD(\kk) \rangle.
\end{equation}
\end{proposition}

\begin{proof}
By~\cite[\S6.2]{K06}, see also~\cite[Theorem~4.4]{k2005v12},
there is a smooth curve~$\Gamma_X$ of genus~$7$ and a $\Gamma_X$-flat family~$\cE$ 
of stable vector bundles on~$X$ with the same rank and Chern classes 
as in the definition of the moduli space~$\rM$ 
and such that the Fourier--Mukai functor
\begin{equation*}
\Phi_\cE \colon \bD(\Gamma_X) \to \bD(X)
\end{equation*}
is fully faithful and together with the exceptional vector bundles~$\cO_X$ and~$\cU^\vee$ 
gives the semiorthogonal decomposition~\eqref{eq:sod-x7}.
The rest of the proof is analogous to that of Proposition~\ref{prop:x10-moduli} 
with~\eqref{eq:sequence-ce-x10} replaced by the exact sequence
\begin{equation}
\label{eq:sequence-ce-x7}
0 \to \cE_y \to \cO_X^{\oplus 5} \to \cU^\vee \to \cE_y(H_X) \to 0
\end{equation}
(see~\cite[(3)]{k2005v12}).
\end{proof}

\begin{remark}
One could also construct the curve~$\Gamma$ as the relative linear section
\begin{equation*}
\Gamma \coloneqq \OGr_{S,-}(5,V) \times_{\P_S(\cS_-)} \P_S(A)
\end{equation*}
of the second component of the relative orthogonal Grassmannian with respect to its embedding 
into the projectivization of the other half-spinor bundle~$\cS_-$ on~$S$.
\end{remark}

\begin{theorem}
\label{thm:x7}
If $X/S$ is a form of a prime Fano threefold of genus~$7$ then there is a semiorthogonal decomposition
\begin{equation*}
\bD(X) = \langle \bD(\Gamma), \cO_X \otimes \bD(S), \cU^\vee \otimes \bD(S) \rangle,
\end{equation*}
where $\Gamma/S$ is a smooth curve of genus~$7$.
\end{theorem}

\begin{proof}
The proof is analogous to the proof of Theorem~\ref{thm:y4},
with Proposition~\ref{prop:y4-moduli} replaced by Proposition~\ref{prop:x7-moduli}. 
The only difference is the absence of the Brauer class~$\upbeta_\Gamma$, which is due to the two facts:
first, $\wedge^2\cE \cong \cO(H_X)$ up to twist by a line bundle on~$\Gamma$, hence~$\upbeta_\Gamma^2 = 1$;
and second, the pushforward of~$\cE^\vee$ to~$\Gamma$ has rank~$5$, hence~$\upbeta_\Gamma^5 = 1$;
a combination of these facts implies~$\upbeta_\Gamma = 1$.
\end{proof}

\section{Weil restriction of scalars}
\label{sec:weil}

Let~$S$ be a connected scheme and let $f \colon S' \to S$ be a finite \'etale morphism.
The {\sf Weil restriction of scalars} functor
\begin{equation*}
\Res_{S'/S} \colon \Sch/S' \to \Sch/S
\end{equation*}
is defined (see, e.g., \cite[\S7.6]{BLR}) as the right adjoint functor of the extension of scalars functor 
\begin{equation*}
\Sch/S \to \Sch/S',
\qquad 
X \mapsto X \times_S S'.
\end{equation*}
By definition, we have a natural isomorphism
\begin{equation}
\label{eq:res-adjunction}
\Map_S(X, \Res_{S'/S}(Y)) \cong \Map_{S'}(X \times_S S', Y)
\end{equation}
between the sets of morphisms in the categories of $S$-schemes and $S'$-schemes, respectively.
It is well known that Weil restriction commutes with base changes.

\subsection{Forms of powers of Fano varieties}
\label{ss:forms-powers}

Recall from~\cite[\S3.8]{CTS} the corestriction map 
\begin{equation*}
\cores_{S'/S} \colon \rH^2_\et(S',\Gm) \to \rH^2_\et(S,\Gm)
\end{equation*}
as well as its restriction~$\cores_{S'/S} \colon \Br(S') \to \Br(S)$.
Recall also the sequence~\eqref{eq:picard-brauer-sequence}.

\begin{lemma}
\label{cor:segre}
If~$f \colon S' \to S$ is a finite \'etale morphism of degree~$d$,
$\upbeta' \in \Br(S')$ and~$V$ is a $\upbeta'$-twisted vector bundle of rank~$r$ on~$S'$ 
then there is a $\cores_{S'/S}(\upbeta')$-twisted vector bundle~$W$ of rank~$r^d$ on~$S$ 
and a closed embedding
\begin{equation*}
\upsigma \colon \Res_{S'/S}(\P_{S'}(V)) \hookrightarrow \P_S(W)
\end{equation*}
which restricts to the Segree embedding~$(\P^{r-1})^d \hookrightarrow \P^{r^d-1}$ on each geometric fiber.
\end{lemma}

We will write~$W \coloneqq \cores_{S'/S}(V)$ and call~$W$ the {\sf Segre bundle} of~$V$
and~$\upsigma$ the {\sf Segre embedding}.

\begin{proof}
Let~$s_0 \in S$ be a base point.
Compatibility with base change implies
\begin{equation*}
\Res_{S'/S}(\P_{S'}(V))_{s_0} \cong 
\Res_{f^{-1}(s_0)/s_0}(\P_{f^{-1}(s_0)}(V_{f^{-1}(s_0)})) \cong 
\prod_{s' \in f^{-1}(s_0)} \P(V_{s'}).
\end{equation*}
The fundamental group~$\uppi_1(S,s_0)$ acts on the Picard group~$\ZZ^d$ of the right hand side by permutations,
the fundamental divisor class of~$\Res_{S'/S}(\P_{S'}(V))$ corresponds to the sum of hyperplane classes of the factors,
and its space of global sections on the fiber over~$s_0$ is isomorphic to the dual of the Segre space 
\begin{equation*}
W_{s_0} \coloneqq \bigotimes_{s' \in f^{-1}(s_0)} V_{s'}
\end{equation*}
Comparing this with the definition of the corestriction map in~\cite[\S3.8]{CTS},
it is easy to see that the Brauer class of~$W$ is equal to~$\cores_{S'/S}(\upbeta')$.
The rest follows from Lemma~\ref{lemma:brauer-obstruction}.
\end{proof}

The following result will be used quite often in~\S\ref{sec:big-picard}.

\begin{proposition}
\label{prop:forms-general}
Let $p \colon X \to S$ be a smooth Fano fibration with~$\Pic_{X/S}(S) \cong \ZZ$.
Assume for each geometric point~\mbox{$s \in S$} 
there is a smooth Fano variety~$Y_s$ of Picard rank~$1$ and a closed embedding~$\upxi_s \colon X_s \hookrightarrow Y_s^d$, such that 
\begin{aenumerate}
\item 
\label{item:projections}
the projections~$\pr_{s,i} \colon X_s \to Y_s$ are surjective and have connected fibers for~$1 \le i \le d$, and
\item 
\label{item:picard}
if~$h_s \in \Pic(Y_s)$ is an ample generator and~$h_{s,i} \coloneqq \pr_{s,i}^*(h_s) \in \Pic(X_s)$ then
\begin{equation*}
\Pic(X_s) = \bigoplus_{i=1}^d \ZZ h_{s,i}
\end{equation*}
and~$h_{s,i}$ generate the nef cone of~$X_s$.
\end{aenumerate}
Then there is a finite \'etale covering~$f \colon S' \to S$ of degree~$d$ with connected~$S'$,
a smooth projective Fano fibration~\mbox{$Z \to S'$},
and a closed embedding
\begin{equation*}
\uppsi \colon X \hookrightarrow \Res_{S'/S}(Z)
\end{equation*}
such that for each geometric point~$s' \in S'$ there is an isomorphism~$Z_{s'} \cong Y_{f(s')}$ and the diagram
\begin{equation}
\label{eq:xyz-diagram}
\vcenter{
\xymatrix{
X_s \ar[r]^-{\uppsi_s} \ar[d]_{\xi_s} &
(\Res_{S'/S}(Z))_s \ar@{=}[d]
\\
Y_s^d &
\prod_{s' \in f^{-1}(s)} Z_{s'} \ar[l]_-\sim
}}
\end{equation}
commutes, where the bottom isomorphism is the product of the above isomorphisms.
\end{proposition}
\begin{proof}
Let~$s_0 \in S$ be a geometric point.
Since~$h_{s_0,i}$ are the generators of the nef cone of~$X_{s_0}$,
it follows from Lemma~\ref{lemma:nef-cone} that the monodromy action of~$\uppi_1(S,s_0)$ on~$\Pic(X_{s_0})$
permutes the classes~$h_{s_0,i}$.
Moreover, since~$\Pic(X_{s_0})^{\uppi_1(S,s_0)} \cong \Pic_{X/S}(S) \cong \ZZ$, 
the action of~$\uppi_1(S,s_0)$ on the set~$\{h_{s_0,1},\dots,h_{s_0,d}\}$ is transitive;
in particular, the~$\uppi_1(S,s_0)$-orbit of~$h_{s_0,1}$ has length~$d$.

Let~$f \colon S' \to S$ be the finite \'etale morphism of degree~$d$ with connected~$S'$ 
constructed by applying Corollary~\ref{cor:orbit-covering} to the class~$h_{s_0,1}$,
set~$X' \coloneqq X \times_S S'$,
and let~$s'_0 \in S'$ be the point such that~$h_{s_0,1}\vert_{X_{s'_0}}$ is monodromy invariant.
By Corollary~\ref{cor:picard-invariant} there is a unique class~$h' \in \Pic_{X'/S'}(S')$ that restricts to~$h_{s_0,1} \in \Pic(X_{s'_0})$.
By Lemma~\ref{lemma:nef-cone} the restriction~$h'\vert_{X'_{s'_0}}$ is a generator of the nef cone, therefore
\begin{equation*}
h'\vert_{X'_{s'_0}} = h_{f(s'_0),i} \in \Pic(X'_{s'_0}) = \Pic(X_{f(s'_0)})
\end{equation*}
for some~$1 \le i \le d$.
Since~$h_{f(s'_0)} \in \Pic(Y_{f(s'_0)})$ is an ample class, there exists an integer~$m(s'_0) \in \ZZ$ 
such that the class~$m(s'_0)h_{f(s'_0)}$ is very ample (hence globally generated and has vanishing higher cohomology).
Since the projection~$\pr_{f(s'_0),i}$ has connected fibers, 
the class~$m(s'_0)h'\vert_{X'_{s'_0}}$ is globally generated and has vanishing higher cohomology as well,
and since both properties are open and the map~$f$ is proper, 
the same holds in the preimage~$U' = f^{-1}(U)$ of a Zariski neighborhood~$U$ of the point~$f(s'_0)$ in~$S$.

Replacing~$S'$ and~$S$ by~$U'$ and~$U$ and applying Lemma~\ref{lemma:brauer-obstruction} to the class~$m(s'_0)h'$
we obtain a morphism to a Severi--Brauer variety over~$S'$,
\begin{equation*}
\phi \colon X' \to \P_{U'}(V)
\end{equation*}
such that for any geometric point~$s' \in U'$ its fiber~$\phi_{s'}$ coincides with the map 
\begin{equation*}
X_{s'} = X_{f(s')} \xrightarrow{ \pr_{f(s'),i}\ } Y_{f(s')} \hookrightarrow \P^N,
\end{equation*}
where the last arrow is the map given by the very ample class~$mh_{f(s')}$ on~$Y_{f(s')}$.
Define~$Z_{U'} \subset \P_{U'}(V)$ as the image of the morphism~$\phi$.
Then it follows from the above that for each point~$s' \in S'$ we have
\begin{equation*}
Z_{U',s'} \cong Y_{f(s')};
\end{equation*}
in particular~$Z_{U'} \to U'$ is a smooth Fano fibration. 

Applying this construction to other points~$s_0 \in S$, 
we obtain an open cover~$\{U\}$ of~$S$, the induced open cover~$\{U'\}$ of~$S'$,
and a family of Fano fibrations~$Z_{U'} \to U'$.
The construction shows that these Fano fibrations agree over the intersections of the opens of the cover,
hence they can be glued into a single Fano fibration~$Z \to S'$.

The morphism $X' \xrightarrow{\ \phi\ } Z$ induces by adjunction a morphism~$\uppsi \colon X \to \Res_{S'/S}(Z)$ of~$S$-schemes. 
The commutativity of~\eqref{eq:xyz-diagram} follows from construction.
In particular, $\uppsi_s$ is a closed embedding for each~$s \in S$, 
and since it is a projective morphism, it is a closed embedding globally.
\end{proof}

\subsection{Semiorthogonal decomposition for Weil restriction of scalars}
\label{ss:weil-decomposition}

Let~$f \colon S' \to S$ be a finite \'etale morphism and let~$q \colon Y \to S'$ be a smooth projective morphism.
Set~$X \coloneqq \Res_{S'/S}(Y)$ and let~$p \colon X \to S$ be the natural morphism.
In this subsection we construct a semiorthogonal decomposition for~$\bD(X)$.

Consider the diagram of maps of $S$-schemes
\begin{equation}
\label{eq:weil-diagram}
\vcenter{\xymatrix{
&
X \times_S S' \ar[dr]^{\pr_Y} \ar[dl]_{\pr_X} 
\\
X &&
Y,
}}
\end{equation}
where the map~$\pr_Y \colon X \times_S S' \to Y$ corresponds by~\eqref{eq:res-adjunction} 
to the identity morphism~$X \to \Res_{S'/S}(Y)$, and the map~$\pr_X$ is the projection to the first factor.
Recall the notion of relative exceptionality from~\S\ref{ss:db-linear}.

\begin{theorem}
\label{thm:pmd-qs-ps}
Assume the structure sheaf~$\cO_Y$ is relative exceptional over~$S'$.
Then~$\cO_{X}$ is relative exceptional over~$S$, the Fourier--Mukai functor 
\begin{equation*}
\Phi \coloneqq \pr_{X*} \circ \pr_Y^* \colon \bD(Y) \to \bD(X)
\end{equation*}
is fully faithful on the subcategories~$(\cO_Y \otimes \bD(S'))^\perp$ and~${}^\perp(\cO_Y \otimes \bD(S'))$ in~$\bD(Y)$
and the pairs of subcategories
\begin{equation}
\label{eq:pairs}
\Phi \big( (\cO_Y \otimes \bD(S'))^\perp \big), \cO_X \otimes \bD(S)
\qquad\text{and}\qquad 
\cO_X \otimes \bD(S), \Phi \big( {}^\perp(\cO_Y \otimes \bD(S')) \big)
\end{equation}
in~$\bD(X)$ are semiorthogonal, $S$-linear, and admissible.
\end{theorem}

\begin{proof}
By Proposition~\ref{prop:relative-sod}\ref{item:linear-ff}--\ref{item:linear-so} it is enough to verify the result 
in the case where~$S = \{s\}$ is a point (i.e., the spectrum of an algebraically closed field) 
and~$S' = \{s'_1,\dots,s'_d\}$ is a finite union of reduced points.
In this case~$Y = Y_1 \sqcup \dots \sqcup Y_d$ is a disjoint union of~$d$ smooth projective components, 
the structure sheaf~$\cO_{Y_k}$ is exceptional for each~$1 \le k \le d$,
and~$X \cong Y_1 \times \dots \times Y_d$.
Moreover, the diagram~\eqref{eq:weil-diagram} takes the form 
\begin{equation*}
\vcenter{\xymatrix{
&
(Y_1 \times \dots \times Y_d) \sqcup \dots \sqcup (Y_1 \times \dots \times Y_d) 
\ar[dr]^{\qquad (\pr_1, \dots, \pr_d)} \ar[dl]_{(\id, \dots, \id)\quad } 
\\
Y_1 \times \dots \times Y_d &&
Y_1 \sqcup \dots \sqcup Y_d,
}}
\end{equation*}
where~$\pr_k$ is the projection to the~$i$-th factor.
Thus, the functor~$\Phi$ is isomorphic to the direct sum 
\begin{equation*}
\Phi \cong \moplus_{k=1}^d\, \pr_k^*\, \colon \moplus_{k=1}^d \bD(Y_k) \to \bD(Y_1 \times \dots \times Y_d).
\end{equation*}
It remains to show that this functor is fully faithful on the left and right orthogonals of the object
\begin{equation*}
\cO_{Y} = 
\moplus_{k=1}^d \cO_{Y_k} \in 
\moplus_{k=1}^d \bD(Y_k).
\end{equation*}

So, let $\cF_1 \in \bD(Y_{k_1})$, $\cF_2 \in \bD(Y_{k_2})$.
If~$k_1 \ne k_2$ then~$\Ext^\bullet(\Phi(\cF_1), \Phi(\cF_2))$ can be rewritten as 
\begin{multline}
\label{eq:ext-cf1-cf2}
\Ext^\bullet(\pr_{k_1}^*(\cF_1), \pr_{k_2}^*(\cF_2)) \cong
\rH^\bullet(Y_1 \times \dots \times Y_d, \pr_{k_1}^*(\cF_1^\vee) \otimes \pr_{k_2}^*(\cF_2)) 
\\ \cong
\rH^\bullet(Y_{k_1}, \cF_1^\vee) \otimes \rH^\bullet(Y_{k_2},\cF_2) \cong
\Ext^\bullet(\cF_1, \cO_{Y_{k_1}}) \otimes \Ext^\bullet(\cO_{Y_{k_2}}, \cF_2),
\end{multline}
where the second isomorphism follows from the K\"unneth formula combined with exceptionality of~$\cO_{Y_k}$ for each~$k \ne k_1,k_2$.
Now if~$\cF_1 \in \cO_{Y_{k_1}}^\perp$, $\cF_2 \in \cO_{Y_{k_2}}^\perp$,
the second factor in the right hand side vanishes,
and similarly if~$\cF_1 \in {}^\perp\cO_{Y_{k_1}}$, $\cF_2 \in {}^\perp\cO_{Y_{k_2}}$ 
the first factor in the right hand side vanishes.
Thus, in both cases the right hand side is zero, 
which agrees with the fact that~$\cF_1$ and~$\cF_2$ are orthogonal on~$Y_1 \sqcup \dots \sqcup Y_d$ 
as they are supported on different connected components.

Similarly, if~$k_1 = k_2$, a similar computation gives 
\begin{equation*}
\Ext^\bullet(\pr_{k_1}^*(\cF_1), \pr_{k_2}^*(\cF_2)) \cong \Ext^\bullet(\cF_1, \cF_2),
\end{equation*}
so full faithfulness of~$\Phi$ on the orthogonals of~$\cO_Y \otimes \bD(S')$ follows.

Finally, orthogonality of the components in~\eqref{eq:pairs} 
follows from~\eqref{eq:ext-cf1-cf2} with either~$\cF_1$ or~$\cF_2$ replaced by~$\cO$,
and $S$-exceptionality of~$\cO_X$ follows from~\eqref{eq:ext-cf1-cf2} with both~$\cF_1$ and~$\cF_2$ replaced by~$\cO$.
\end{proof}

\begin{remark}
On the entire category $\bD(Y)$ the functor~$\Phi$ is \emph{not} fully faithful (of course, if~$d > 1$);
indeed, the $S$-linear subcategory of~$\bD(Y)$ generated by~$\cO_Y$ is equivalent to~$\bD(S')$, 
while the $S$-linear subcategory of~$\bD(X)$ generated by~$\cO_X$ is equivalent to~$\bD(S)$. 
\end{remark}

\begin{remark}
The components of~\eqref{eq:pairs} do not generate~$\bD(X)$ even in the simplest example 
where~\mbox{$S' \to S$} has degree~$2$ and~$Y = \P^1_{S'}$;
in this case the orthogonal is equal to~\mbox{$\cO_X(H_X) \otimes \bD(S)$}.
The case where~$S' \to S$ has degree~$3$ and~$Y \to S'$ is a $\P^1$-fibration will be discussed in~\S\ref{ss:x111}.
\end{remark}

\section{Fano threefolds of higher geometric Picard number}
\label{sec:big-picard}

In this section we describe Fano fibrations $p \colon X \to S$ 
with geometric Picard rank of fibers greater than~2.
We keep the notation~$H_X = -K_X \in \Pic(X)$ for the relative fundamental class.

\subsection{Forms of~$\sX_{1,1,1}$}
\label{ss:x111}

Recall that according to notation from the Introduction 
over an algebraically closed field every Fano threefold of type~$\sX_{1,1,1}$ is isomorphic to~$\P^1 \times \P^1 \times \P^1$.
Using Proposition~\ref{prop:forms-general} we easily obtain a description of all Fano fibrations of this type.

\begin{proposition}
\label{prop:x111}
If~$p \colon X \to S$ is a smooth Fano fibration with fibers of type~$\sX_{1,1,1}$
there is an \'etale covering~$S' \to S$ of degree~$3$ with connected~$S'$, 
a $2$-torsion Brauer class~$\upbeta' \in \Br(S')$,
and a $\upbeta'$-twisted vector bundle~$V$ on~$S'$ of rank~$2$ 
such that
\begin{equation}
\label{eq:form-x111}
X = \Res_{S'/S}(\P_{S'}(V)).
\end{equation} 
If~$X(S) \ne \varnothing$ then~$\upbeta' = 1$.
\end{proposition}
\begin{proof}
We apply Proposition~\ref{prop:forms-general} with~$Y_s = \P^1$, $d = 3$, and~$\upxi_s$ being an isomorphism.
Then~$Z \to S'$ is a smooth~$\P^1$-fibration, hence~$Z \cong \P_{S'}(V)$ by Lemma~\ref{lemma:sb-twisted-sheaves},
where~$V$ is a $\upbeta'$-twisted vector bundle of rank~$2$ and~${\upbeta'}^2 = 1$ by Corollary~\ref{cor:twisted-rank};
the isomorphism~\eqref{eq:form-x111} follows.
\end{proof}

To describe the derived category of~$X$ we will use the functor~$\Phi$ from Theorem~\ref{thm:pmd-qs-ps}.
We denote by~$H_V$ the fundamental class of~$\P_{S'}(V)$.

\begin{theorem}
\label{thm:x111}
If $X/S$ is a form of a Fano threefold of type~$\sX_{1,1,1}$ then there is a semiorthogonal decomposition
\begin{multline*}
\bD(X) = \langle
\cO_X \otimes \bD(S), 
\Phi(\cO_{\P_{S'}(V)}(H_V) \otimes \bD(S',\upbeta')), 
\\
\cO_X(H_X) \otimes \bD(S,\upbeta), 
\cO_X(H_X) \otimes \Phi(\cO_{\P_{S'}(V)}(H_V) \otimes \bD(S',\upbeta' \cdot f^*(\upbeta))) 
\rangle,
\end{multline*}
where~$f \colon S' \to S$ and~$\upbeta' \in \Br(S')$ are described in Proposition~\textup{\ref{prop:x111}}
and~$\upbeta \coloneqq \cores_{S'/S}(\upbeta')$.
\end{theorem}

Note that if~$X(S) \ne \varnothing$ then~$\upbeta' = 1$ by Proposition~\ref{prop:x111}, hence also~$\upbeta = 1$,
so that in this case all the components of~$\bD(X)$ are untwisted.

\begin{proof}
By Theorem~\ref{thm:bernardara} 
we have~$\bD(\P_{S'}(V)) = \langle \cO_{\P_{S'}(V)} \otimes \bD(S'), \cO_{\P_{S'}(V)}(H_V) \otimes \bD(S', \upbeta') \rangle$.
Applying Theorem~\ref{thm:pmd-qs-ps} we obtain the first two components in the required semiorthogonal decomposition of~$\bD(X)$.
Tensoring them by the~$\upbeta$-twisted line bundle~$\cO(H_X)$ 
(note that the Brauer twist of~$H_X$ is equal to~$\cores_{S'/S}(\upbeta')$ by Lemma~\ref{cor:segre})
and modifying the Brauer twists of the components appropriately,
we obtain the last two components.
To check the semiorthogonality and generation, we apply Proposition~\ref{prop:relative-sod}\ref{item:linear-so}--\ref{item:linear-sod}.
Accordingly, we need to consider the case where~$S$ is the spectrum of an algebraically closed field.
In this case~$X \cong \P^1 \times \P^1 \times \P^1$, and the required semiorthogonal decomposition 
follows from the exceptional collection
\begin{equation*}
\bD(\P^1 \times \P^1 \times \P^1) = \langle \cO, \cO(1,0,0), \cO(0,1,0), \cO(0,0,1), \cO(1,1,1), \cO(2,1,1), \cO(1,2,1), \cO(1,1,2) \rangle 
\end{equation*}
which has been proved in~\cite{Mir}, see also~\cite[Appendix~D]{K19}.
\end{proof}

\subsection{Forms of~$\sX_{2,2}$}
\label{ss:x22}

Recall that according to notation from the Introduction 
over an algebraically closed field every Fano threefold of type~$\sX_{2,2}$ 
is isomorphic to a divisor of bidegree~$(1,1)$ in~$\P^2 \times \P^2$.
Using Proposition~\ref{prop:forms-general} we obtain a description of all Fano fibrations of this type.

For an \'etale double covering~$S' \to S$ we denote the action 
of the Galois involution of~$S'$ over~$S$ on~$\Br(S')$ by~$\upbeta' \mapsto \bar\upbeta'$.
Recall the Segre embedding defined in Lemma~\ref{cor:segre}.

\begin{proposition}
\label{prop:x22}
If~$p \colon X \to S$ is a smooth Fano fibration with fibers of type~$\sX_{2,2}$
there is an \'etale covering~$S' \to S$ of degree~$2$ with connected~$S'$, 
a $3$-torsion Brauer class~$\upbeta' \in \Br(S')$ such that 
\begin{equation}
\label{eq:x22-beta}
\bar\upbeta' = {\upbeta'}^{-1},
\end{equation} 
and a $\upbeta'$-twisted vector bundle~$V$ on~$S'$ of rank~$3$ 
such that~$X \subset \Res_{S'/S}(\P_{S'}(V))$.
Moreover, if~$W \coloneqq \cores_{S'/S}(V)$ is the Segre bundle then~$W$ is untwisted and
there is a line bundle~$\cL$ and an epimorphism~\mbox{$\varphi \colon W \twoheadrightarrow \cL^\vee$} 
such that 
\begin{equation}
\label{eq:form-x22}
X = \Res_{S'/S}(\P_{S'}(V)) \times_{\P_S(W)} \P_S(\Ker(\varphi)),
\end{equation}
where the morphism $\Res_{S'/S}(\P_{S'}(V)) \to \P_S(W)$ in the fiber product is the Segre embedding.

If~$X(S) \ne \varnothing$ then~$\upbeta' = 1$.
\end{proposition}

\begin{proof}
We apply Proposition~\ref{prop:forms-general} with~$Y_s = \P^2$, $d = 2$, and~$\upxi_s$ being the natural embedding.
Then~$Z \to S'$ is a smooth~$\P^2$-fibration, hence~$Z \cong \P_{S'}(V)$ by Lemma~\ref{lemma:sb-twisted-sheaves},
where~$V$ is a $\upbeta'$-twisted vector bundle of rank~$3$ and~${\upbeta'}^3 = 1$ by Corollary~\ref{cor:twisted-rank};
in this way we obtain a divisorial embedding~$\uppsi \colon X \hookrightarrow \Res_{S'/S}(\P_{S'}(V))$.

Let~$H_W$ be the fundamental class of~$\Res_{S'/S}(\P_{S'}(V))$.
By Lemma~\ref{cor:segre} it is $p^*(\upbeta^{-1})$-twisted, where~$\upbeta \coloneqq \cores_{S'/S}(\upbeta')$.
So, if~$p_W \colon \Res_{S'/S}(\P_{S'}(V)) \to S$ is the projection then
\begin{equation*}
\cL \coloneqq p_{W*}(\cI_X \otimes \cO(H_W))
\end{equation*}
is an $\upbeta^{-1}$-twisted line bundle on~$S$,
hence $\upbeta = 1$ by Corollary~\ref{cor:twisted-rank}
and~\mbox{$\upbeta' \cdot \bar\upbeta' = f^*(\upbeta) = 1$},
which implies~\eqref{eq:x22-beta}.
Denoting by~$\varphi$ the dual of the embedding~$\cL \hookrightarrow p_{W*}\cO(H_W) = W^\vee$, we deduce~\eqref{eq:form-x22}.

Finally, if~$X(S) \ne \varnothing$, the morphism~$\P_{S'}(V) \to S'$ has a section by~\eqref{eq:res-adjunction}, 
hence~$\upbeta' = 1$.
\end{proof}

\begin{theorem}
\label{thm:x22}
If $X/S$ is a form of a Fano threefold of type~$\sX_{2,2}$ then there is a semiorthogonal decomposition
\begin{multline*}
\bD(X) = \langle
\cO_X \otimes \bD(S), 
\uppsi^* \big( \Phi(\cO_{\P_{S'}(V)}(H_V) \otimes \bD(S',\upbeta')) \big), 
\\
\cO_X(H_X) \otimes \bD(S), 
\cO_X(H_X) \otimes \uppsi^* \big( \Phi(\cO_{\P_{S'}(E)}(H_V) \otimes \bD(S',\upbeta')) \big) 
\rangle,
\end{multline*}
\end{theorem}

Note that if~$X(S) \ne \varnothing$ then~$\upbeta' = 1$ by Proposition~\ref{prop:x22}
and all the components are untwisted.

\begin{proof}
The proof is analogous to the proof of Theorem~\ref{thm:x111};
the only difference is that this time the class~$\upbeta = \cores_{S'/S}(\upbeta')$ is trivial,
and for the geometric fibers~$X_s \cong \Fl(1,2;3)$ of~$X \to S$ we use the exceptional collection
\begin{equation}
\label{eq:db-fl123}
\bD(\Fl(1,2;3)) = \langle \cO, \cO(1,0), \cO(0,1), \cO(1,1), \cO(2,1), \cO(1,2) \rangle 
\end{equation}
established in~\cite[Appendix~C]{K19}.
\end{proof}

\subsection{Forms of~$\sX_{2,2,2}$}
\label{ss:x222}

Recall that according to notation from the Introduction 
over an algebraically closed field every Fano threefold of type~$\sX_{2,2,2}$ 
is isomorphic to a complete intersection of divisors of multidegree~$(1,1,0)$, $(1,0,1)$, and~$(0,1,1)$
in~$\P^2 \times \P^2 \times \P^2$.
It has Picard number~3, anticanonical degree~30, and number~3.13 in Mori--Mukai classification.
Using Proposition~\ref{prop:forms-general} we obtain a description of all Fano fibrations of this type.

\begin{proposition}
\label{prop:x222}
If~$p \colon X \to S$ is a smooth Fano fibration with fibers of type~$\sX_{2,2,2}$
there is an \'etale covering~$S' \to S$ of degree~$3$ with connected~$S'$,
a vector bundle~$V$ on~$S'$ of rank~$3$, and a closed embedding 
\begin{equation*}
\uppsi \colon X \hookrightarrow \Res_{S'/S}(\P_{S'}(V))
\end{equation*}
which over each geometric point~$s \in S$ coincides with the natural embedding~$X_s \hookrightarrow \P^2 \times \P^2 \times \P^2$.
\end{proposition}

\begin{proof}
We apply Proposition~\ref{prop:forms-general} with~$Y_s = \P^2$, $d = 3$, and~$\upxi_s$ being the natural embedding.
Then~$Z \to S'$ is a smooth~$\P^2$-fibration, hence~$Z \cong \P_{S'}(V)$ by Lemma~\ref{lemma:sb-twisted-sheaves},
where~$V$ is a $\upbeta'$-twisted vector bundle of rank~$3$ and~${\upbeta'}^3 = 1$ by Corollary~\ref{cor:twisted-rank};
in this way we obtain the required closed embedding~$\uppsi \colon X \hookrightarrow \Res_{S'/S}(\P_{S'}(V))$,
so it remains to show that~$\upbeta' = 1$.

For this let~$H_V$ be the fundamental class of~$\P_{S'}(V)$ and let~$H_X$ be the fundamental class of~$X$.
Consider the fiber product~$X \times_S S'$ and the two maps 
\begin{equation*}
\xymatrix@1@C=3em{
\Res_{S'/S}(\P_{S'}(V)) &
\Res_{S'/S}(\P_{S'}(V)) \times_S S' \ar[l]_-{\pr_1} \ar[r]^-{\pr_2} &
\P_{S'}(V),
}
\end{equation*}
where~$\pr_1$ is the projection to the first factor 
and~$\pr_2$ corresponds by adjunction~\eqref{eq:res-adjunction} to the embedding~$\uppsi$.
Finally, let~$p_V \colon \Res_{S'/S}(\P_{S'}(V)) \times_S S' \to S'$ be the natural projection.
Now let~$\cI_X$ be the ideal of~$X$ in~$\Res_{S'/S}(\P_{S'}(V))$ and consider the sheaf 
\begin{equation*}
\cL \coloneqq {p_V}_*(\pr_1^*(\cI_X(H_X)) \otimes \pr_2^*\cO(-H_V)).
\end{equation*}
If~$s' \in S'$ is a geometric point and~$s = f(s') \in S$, 
then~$p_V^{-1}(s') \cong \P^2 \times \P^2 \times \P^2$, 
and if~$h_i$ denote the hyperplane classes of the factors, then 
\begin{equation*}
(\pr_1^*(\cI_X(H_X)) \otimes \pr_2^*\cO(-H_V))\vert_{p_V^{-1}(s')} \cong 
\cI_{X_s}(h_i + h_j),
\end{equation*}
the ideal of~$X_s \subset \P^2 \times \P^2 \times \P^2$ twisted by~$\cO(h_i + h_j)$, 
where~$1 \le i < j \le 3$ and the point~\mbox{$s' \in f^{-1}(s)$} corresponds to the class~$h_k$ with~$k \not\in \{i,j\}$
(see the construction of~$S'$ in the proof of Proposition~\ref{prop:forms-general}).
Using the Koszul resolution
\begin{multline*}
0 \to \cO(-2h_1-2h_2-2h_3) 
\\
\to \cO(-2h_1-h_2-h_3) \oplus \cO(-h_1-2h_2-h_3) \oplus \cO(-h_1-h_2-2h_3) 
\\ 
\to \cO(-h_1-h_2) \oplus \cO(-h_1-h_3) \oplus \cO(-h_2-h_3) \to \cI_{X_s} \to 0
\end{multline*}
it is easy to check that~$\rH^0(\P^2 \times \P^2 \times \P^2, \cI_{X_s}(h_i + h_j))$ is $1$-dimensional.
As this holds for any~$s'$, we conclude that~$\cL$ is a line bundle,
and since~$H_X$ is $p^*(\upbeta^{-1})$-twisted, where~$\cores_{S'/S}(\upbeta')$, 
while~$H_V$ is~${\upbeta'}^{-1}$-twisted, we see that~$\cL$ is~$\upbeta' \cdot f^*(\upbeta^{-1})$-twisted.
Therefore,~$\upbeta' = f^*(\upbeta)$ by Corollary~\ref{cor:twisted-rank},
and we deduce from~\cite[\S3.8]{CTS} the equality
\begin{equation*}
\upbeta = \cores_{S'/S}(\upbeta') = \cores_{S'/S}(f^*(\upbeta)) = \upbeta^3.
\end{equation*}
But~${\upbeta'}^3 = 1$ by construction, hence~$\upbeta^3 = 1$, hence~$\upbeta = 1$, and hence~$\upbeta' = f^*(\upbeta) = 1$.
\end{proof}

Now to describe the derived category, we start with the case of a variety~$X \subset \P^2 \times \P^2 \times \P^2$ of type~$\sX_{2,2,2}$ over an algebraically closed field.
Let~$h_i$, $1 \le i \le 3$, be the pullbacks to~$X$ of the hyperplane classes of the~$\P^2$-factors.

\begin{lemma}
\label{lemma:x222-picrd}
Let~$X \subset \P^2 \times \P^2 \times \P^2$ be a variety of type~$\sX_{2,2,2}$ over an algebraically closed field.
Each projection
\begin{equation*}
\pr_{i,j} \colon X \to \P^2 \times \P^2,
\qquad 
1 \le i < j \le 3, 
\end{equation*}
is the blowup of a smooth Fano threefold~$Y_{i,j} \subset \P^2 \times \P^2$ of type~$\sX_{2,2}$ 
along a smooth rational curve~$\Gamma_{i,j} \subset Y_{i,j} \subset \P^2 \times \P^2$ of bidegree~$(2,2)$.
The Picard group of~$X$ is freely generated by~$h_1,h_2,h_3$, 
and if~$E_{i,j}$ are the exceptional divisors of the projections~$\pr_{i,j}$, we have
\begin{equation}
\label{eq:x222-picard}
E_{1,2} = h_1 + h_2 - h_3,
\quad 
E_{1,3} = h_1 + h_3 - h_2,
\quad 
E_{2,3} = h_2 + h_3 - h_1,
\quad 
K_X = -h_1 -h_2 -h_3.
\end{equation} 
The cone of effective divisors of~$X$ is generated by the classes~$h_1$, $h_2$, $h_3$, $E_{1,2}$, $E_{1,3}$, $E_{2,3}$, 
and each of these has intersection product~$1$ with the symmetric curve class 
\begin{equation}
\label{eq:delta-x222}
\updelta_X \coloneqq \frac1{10}(h_1 + h_2 + h_3)^2.
\end{equation}
\end{lemma}

\begin{proof}
The first part is easy, see~\cite[Lemma~2.4(i), Lemma~2.3]{KP21}. 

To describe the cone of effective divisors, let~$\rf_{i,j}$ denote the class of a fiber of the map~$E_{i,j} \to \Gamma_{i,j}$.
Note that if~$(i,j,k)$ is a permutation of~$(1,2,3)$, we have
\begin{equation*}
h_i \cdot \rf_{i,j} = h_j \cdot \rf_{i,j} = 0,
\qquad
h_k \cdot \rf_{i,j} = - E_{i,j} \cdot \rf_{i,j} = 1
\end{equation*}
Note also that curves of type~$\rf_{i,j}$ sweep the divisor~$E_{i,j}$, 
so if an effective divisor~$D \subset X$ has negative intersection with~$\rf_{i,j}$, it contains~$E_{i,j}$,
so we can write~$D = D' + E_{i,j}$, where~$D'$ is also effective.
Iterating this observation, we eventually obtain an effective divisor~$\bar{D}$ 
such that~$\bar{D} \cdot \rf_{i,j} \ge 0$ for all~$1 \le i < j \le 3$.
If~$\bar{D} = a_1h_1 + a_2h_2 + a_3h_3$ then
\begin{equation*}
0 \le \bar{D} \cdot \rf_{i,j} = (a_1h_1 + a_2h_2 + a_3h_3) \cdot \rf_{i,j} = a_k,
\end{equation*}
i.e., now~$\bar{D}$ is a nonnegative linear combination of the~$h_i$, 
while~$D$ by construction is a sum of~$\bar{D}$ and a nonnegative linear combination of~$E_{i,j}$.

Finally, the equalities~$h_i \cdot \updelta_X = 1$ and~$E_{i,j} \cdot \updelta_X = 1$ are straightforward.
\end{proof}

In the next key proposition we use freely the notation introduced in Lemma~\ref{lemma:x222-picrd}.

\begin{proposition}
\label{prop:x222-ec}
Let~$X \subset \P^2 \times \P^2 \times \P^2$ be a Fano threefold of type~$X_{2,2,2}$ 
over an algebraically closed field~$\kk$.
Set~$H_X \coloneqq h_1 + h_2 + h_3$.
Then there is a semiorthogonal decomposition
\begin{equation}
\label{eq:ec-x222}
\bD(X) = \langle \cO_X, \cE_X^\vee, \cO_X(h_1), \cO_X(h_2), \cO_X(h_3), \cO_X(h_1 + h_2), \cO_X(h_1 + h_3), \cO_X(h_2 + h_3) \rangle,
\end{equation}
where~$\cE_X$ is an $H_X$-stable 
exceptional vector bundle of rank~$4$ with~$\det(\cE_X) \cong \cO_X(-H_X)$.

Moreover, there are unique self-dual exact sequences
\begin{equation}
\label{eq:x222-cex-cex-1}
0 \to \cE_X \to \cO_X^{\oplus 8} \to \cE_X^\vee \to 0
\end{equation}
and 
\begin{multline}
\label{eq:x222-cex-cex-2}
0 \to \cE_X^\vee \to 
\cO_X(h_1)^{\oplus 2} \oplus \cO_X(h_2)^{\oplus 2} \oplus \cO_X(h_3)^{\oplus 2} \\ \to
\cO_X(H_X - h_3)^{\oplus 2} \oplus \cO_X(H_X - h_2)^{\oplus 2} \oplus \cO_X(H_X - h_1)^{\oplus 2} \to
\cE_X(H_X) \to 0.
\end{multline}
\end{proposition}

\begin{proof}
Using the blowup formula for~$\pr_{1,2} \colon X \to Y_{1,2}$, we obtain a semiorthogonal decomposition
\begin{equation*}
\bD(X) = \langle \pr_{1,2}^*(\bD(Y_{1,2})), i_*q^*\bD(\Gamma_{1,2}) \rangle,
\end{equation*}
where $q \colon E_{1,2} \to \Gamma_{1,2}$ is the exceptional divisor of the blowup~$\pr_{1,2}$,
and $i \colon E_{1,2} \hookrightarrow X$ is its embedding.
Now we choose exceptional collections in the two components of the above decomposition.
Mutating the last two bundles in the exceptional collection~\eqref{eq:db-fl123} in~$\bD(Y_{1,2})$ to the far left
and taking into account that~$K_{Y_{1,2}} = -2h_1 - 2h_2$, we obtain 
\begin{equation*}
\bD(Y_{1,2}) = \langle \cO(-h_2), \cO(-h_1), \cO, \cO(h_1), \cO(h_2), \cO(h_1 + h_2) \rangle.
\end{equation*}
Combining this collection with~$\bD(\Gamma_{1,2}) = \langle \cO_{\Gamma_{1,2}}(3), \cO_{\Gamma_{1,2}}(4) \rangle$ 
and denoting by~$\rf = \rf_{1,2}$ the class of a fiber of~$q$, we obtain a semiorthogonal decomposition
\begin{equation*}
\bD(X) = \langle \cO(-h_2), \cO(-h_1), \cO, \cO(h_1), \cO(h_2), \cO(h_1 + h_2), \cO_{E_{1,2}}(3\rf), \cO_{E_{1,2}}(4\rf) \rangle.
\end{equation*}
Now we find a sequence of mutations that transforms it into the form~\eqref{eq:ec-x222}.

First, we mutate~$\cO(-h_1)$ and~$\cO(-h_2)$ to the far right.
Since~$-K_X =  H_X = h_1 + h_2 + h_3$, we obtain the collection
\begin{equation*}
\bD(X) = \langle \cO, \cO(h_1), \cO(h_2), \cO(h_1 + h_2), \cO_{E_{1,2}}(3\rf), \cO_{E_{1,2}}(4\rf), \cO(h_1 + h_3), \cO(h_2 + h_3) \rangle.
\end{equation*}
Now we mutate~$\cO_{E_{1,2}}(3\rf)$ three steps to the left.
Since~$\Gamma_{1,2}$ has bidegree~$(2,2)$ in~$Y_{1,2}$, we have 
\begin{equation*}
\Ext^\bullet(\cO(h_1+h_2), \cO_{E_{1,2}}(3\rf)) = 
\Ext^\bullet(\cO(h_1+h_2), \cO_{\Gamma_{1,2}}(3)) = 
H^\bullet(\Gamma_{1,2},\cO_{\Gamma_{1,2}}(-1)) = 0,
\end{equation*}
hence the first step is just the transposition.
Similarly, $\Ext^\bullet(\cO(h_i), \cO_{E_{1,2}}(3\rf)) = \kk^2$, 
hence the result of the mutation, which we denote by~$\cE_X^\vee$, fits into the distinguished triangle
\begin{equation}
\label{eq:x222-cex}
\cE_X^\vee \to \cO(h_1)^{\oplus 2} \oplus \cO(h_2)^{\oplus 2} \to \cO_{E_{1,2}}(3\rf).
\end{equation}
The morphism on the right is given by evaluation, hence surjective,
hence~$\cE_X^\vee$ is an exceptional vector bundle of rank~4, the triangle is an exact sequence, 
and we obtain an exceptional collection
\begin{equation*}
\bD(X) = \langle \cO, \cE_X^\vee, \cO(h_1), \cO(h_2), \cO(h_1 + h_2), \cO_{E_{1,2}}(4\rf), \cO(h_1 + h_3), \cO(h_2 + h_3) \rangle.
\end{equation*}
Finally, we mutate $\cO_{E_{1,2}}(4\rf)$ one step to the left.
We have~$\Ext^\bullet(\cO(h_1+h_2), \cO_{E_{1,2}}(4\rf)) = \kk$. 
On the other hand, using~\eqref{eq:x222-picard} we deduce an exact sequence
\begin{equation*}
0 \to \cO(h_3 - h_1 - h_2) \to \cO \to \cO_{E_{1,2}} \to 0.
\end{equation*}
Clearly, its twist by~$\cO(h_1 + h_2)$ gives the required mutation, 
hence we obtain the required exceptional collection~\eqref{eq:ec-x222}.

The determinant of~$\cE_X$ is easy to compute from~\eqref{eq:x222-cex} and~\eqref{eq:x222-picard},
so it remains to show the stability of~$\cE_X$ and to construct the exact sequences~\eqref{eq:x222-cex-cex-1} and~\eqref{eq:x222-cex-cex-2}.

We start with constructing the exact sequences.
First, note that the construction of the bundle~$\cE_X$ described above depends on a choice of the projection~$\pr_{1,2} \colon X \to Y_{1,2}$.
The same construction applied to other projections produces two other vector bundles, say~$\cE_{1,3}$ and~$\cE_{2,3}$, on~$X$,
which fit into the same exceptional collection~\eqref{eq:ec-x222}.
Therefore, these bundles are isomorphic to~$\cE_X$.
In other words, besides the triangle (in fact exact sequence)~\eqref{eq:x222-cex}, we have two other triangles
\begin{equation*}
\cE_X^\vee \to \cO(h_1)^{\oplus 2} \oplus \cO(h_3)^{\oplus 2} \to \cO_{E_{13}}(3\rf_{13})
\qquad\text{and}\qquad 
\cE_X^\vee \to \cO(h_2)^{\oplus 2} \oplus \cO(h_3)^{\oplus 2} \to \cO_{E_{23}}(3\rf_{23}).
\end{equation*}
Note that all these triangles are mutation triangles, in particular we have 
\begin{equation*}
\Ext^\bullet(\cE_X^\vee, \cO(h_i)) \cong 
\kk^2
\end{equation*}
for all~$1 \le i \le 3$, and the first morphisms in all these triangles are the coevaluation morphisms.

Now let~$\cF_X$ be the right mutation of~$\cE_X^\vee$ 
through the triple of mutually completely orthogonal line bundles~$\cO_X(h_i)$, $1 \le i \le 3$,
so that we have a distinguished triangle
\begin{equation}
\label{eq:x222-cex-cfx}
\cE_X^\vee \to \cO(h_1)^{\oplus 2} \oplus \cO(h_2)^{\oplus 2} \oplus \cO(h_3)^{\oplus 2} \to \cF_X,
\end{equation}
where~$\cF_X$ is an exceptional object and~$\det(\cF_X) \cong \cO_X(H_X)$.
Note that~\eqref{eq:x222-cex-cfx} implies that the second arrow is the evaluation morphism.
Comparing~\eqref{eq:x222-cex-cfx} to~\eqref{eq:x222-cex} and the other two defining triangles of~$\cE_X^\vee$, we obtain triangles
\begin{equation}
\label{eq:x22-cfx-1}
\cO_X(h_i)^{\oplus 2} \to \cF_X \to \cO_{E_{j,k}}(3\rf_{j,k})
\end{equation}
for each permutation~$(i,j,k)$ of~$(1,2,3)$, and dualizing them we obtain triangles
\begin{equation}
\label{eq:x22-cfx-2}
\cF_X^\vee \to \cO_X(-h_i)^{\oplus 2} \to \cO_{E_{j,k}}(\rf_{j,k} - h_i\vert_{E_{j,k}}).
\end{equation}
Each arrow in these triangles is the evaluation or coevaluation morphism.
It follows, that the second arrow in~\eqref{eq:x22-cfx-2} is surjective, hence~$\cF_X^\vee$ is a vector bundle of rank~$2$.
Therefore, its dual~$\cF_X$ is also a vector bundle of rank~2, 
and the triangles~\eqref{eq:x222-cex-cfx} and~\eqref{eq:x22-cfx-1} are exact sequences.

Since~$\cF_X$ is a vector bundle of rank~2 and~$\det(\cF_X) \cong \cO_X(H_X)$, we have
\begin{equation}
\label{eq:x222-cfx-self-dual}
\cF_X \cong \cF_X^\vee(H_X).
\end{equation}
Therefore, merging exact sequence~\eqref{eq:x222-cex-cfx} with its dual twisted by~$\cO_X(H_X)$, we obtain~\eqref{eq:x222-cex-cex-2}.
Since isomorphism~\eqref{eq:x222-cfx-self-dual} is skew-symmetric, so is the sequence~\eqref{eq:x222-cex-cex-2}.

Now note that~\eqref{eq:x222-cex-cfx} is a mutation sequence by construction, 
hence the same is true for its twisted dual, as well as for the sequence~\eqref{eq:x222-cex-cex-2} obtained by merging these two.
Therefore, the result of the right mutation of~$\cE_X^\vee$ to the far right in~\eqref{eq:ec-x222} is~$\cE_X(H_X)[-2]$.
On the other hand, the left mutation of~$\cE_X^\vee$ through~$\cO_X$ is isomorphic to this right mutation 
composed with the Serre functor~$- \otimes \cO_X(-H_X)[3]$ of~$\bD(X)$, thus the result of this left mutation is~$\cE_X[1]$.
Therefore, the corresponding left mutation triangle looks like 
\begin{equation*}
\cE_X \to \Ext^\bullet(\cO_X,\cE_X^\vee) \otimes \cO_X \to \cE_X^\vee.
\end{equation*}
Since the first and last terms are vector bundles of rank~4, the middle term must be isomorphic to~$\cO_X^{\oplus 8}$,
hence the triangle gives the sequence~\eqref{eq:x222-cex-cex-1}.

Now it remains to prove the stability of~$\cE_X$.
By Lemma~\ref{lemma:x222-picrd} the {\sf normalized} slope of~$\cE_X$ is
\begin{equation*}
\upmu(\cE_X) = \frac{ \rc_1(\cE_X) \cdot \updelta_X} {\rank(\cE_X)}
= - \frac{ (h_1 + h_2 + h_3) \cdot \updelta_X}4
= - \frac34.
\end{equation*}
Therefore, to verify stability of~$\cE_X$ or~$\cE_X^\vee$, it is enough to exclude the following possibilities:
\begin{aenumerate}
\item 
\label{item:x222-rank-1}
$\cG \subset \cE_X$ is a reflexive subsheaf of rank~$1$ with~$\upmu(\cG) \ge 0$;
\item 
\label{item:x222-rank-2}
$\cG \subset \cE_X^\vee$ is a stable reflexive subsheaf of rank~$2$ with~$\upmu(\cG) \ge 1$;
\item 
\label{item:x222-rank-3}
$\cG \subset \cE_X^\vee$ is a reflexive subsheaf of rank~$1$ with~$\upmu(\cG) \ge 1$.
\end{aenumerate}

Assume~\ref{item:x222-rank-1}.
Then~$\cG$ is a line bundle and by~\eqref{eq:x222-cex-cex-1} it has a nontrivial morphism to~$\cO_X$, 
therefore~\mbox{$\cG \cong \cO_X(-D)$} with effective~$D$.
Furthermore,~$D \cdot \updelta_X = -\upmu(\cG) \le 0$, 
hence~$D = 0$ by Lemma~\ref{lemma:x222-picrd} and~$\cG \cong \cO_X$.
But~$\Hom(\cO_X,\cE_X) = 0$ by semiorthogonality in~\eqref{eq:ec-x222}, hence~\ref{item:x222-rank-1} is impossible. 

A similar argument (with~\eqref{eq:x222-cex-cex-1} replaced by~\eqref{eq:x222-cex-cex-2}) 
shows that if~\ref{item:x222-rank-3} holds then~$\cG \cong \cO(h_i)$,
which contradicts semiorthogonality of~\eqref{eq:ec-x222}.
Thus, \ref{item:x222-rank-3} is also impossible. 

Finally, assume~\ref{item:x222-rank-2}. 
Combining~\eqref{eq:x222-cex} and~\eqref{eq:x22-cfx-2} 
we obtain an exact sequence
\begin{equation*}
0 \to \cF_X(-h_1) \to \cE_X^\vee \to \cO_X(h_1)^{\oplus 2} \to 0.
\end{equation*}
Since~$\upmu(\cG) \ge 1 = \upmu(\cO_X(h_1))$ and~$\cG$ is stable and reflexive, while~$\upmu(\cF_X(-h_1)) = \frac12$,
the composition~\mbox{$\cG \hookrightarrow \cE_X^\vee \twoheadrightarrow \cO_X(h_1)^{\oplus 2}$} must be an isomorphism,
which contradicts stability of~$\cG$.
Thus, \ref{item:x222-rank-2} is impossible. 
\end{proof}

Now that we know a good symmetric exceptional collection for~$X$ over algebraically closed fields, 
we can pass to fibrations over any base scheme~$S$.

\begin{theorem}
\label{thm:x222}
If $X/S$ is a form of a Fano threefold of type~$\sX_{2,2,2}$ there is a semiorthogonal decomposition
\begin{multline*}
\bD(X) = \langle
\cO_X \otimes \bD(S), 
\cE_X^\vee \otimes \bD(S,\upbeta_\cE), \\
\uppsi^* \big( \Phi(\cO_{\P_{S'}(V)}(1) \otimes \bD(S')) \big), 
\cO_X(H_X) \otimes \uppsi^* \big( \Phi(\cO_{\P_{S'}(V)}(-1) \otimes \bD(S')) \big) 
\rangle,
\end{multline*}
where~$\upbeta_\cE \in \Br(S)$, $\upbeta_\cE^2 = 1$, 
and~$\cE_X$ is an $S$-exceptional $\upbeta_\cE$-twisted vector bundle of rank~$4$ on~$X$.

Moreover, if~$X(S) \ne \varnothing$ then~$\upbeta_\cE$ can be represented by a conic bundle.
\end{theorem}

\begin{proof}
First, we construct a global version of the bundle~$\cE_X$ by using the argument of Proposition~\ref{prop:x16-forms}.
Consider the relative moduli space 
\begin{equation*}
\rM \coloneqq \rM_{X/S,H_X}(4;-H_X, 5\updelta_X, -4P_X),
\end{equation*}
where~$H_X$ is the fundamental class, 
the class~$\updelta_X$ is defined in~\eqref{eq:delta-x222}
and~$P_X$ is the class of a point.
Let also $\rM^\circ \subset \rM$ be the open subscheme parameterizing bundles~$E$ on~$X_s$ with the vanishing
\begin{equation*}
\rH^\bullet(X_s,E) = 0.
\end{equation*}
By Propositions~\ref{prop:x222-ec} and Proposition~\ref{prop:uniqueness-general} 
(conditions~\eqref{eq:cu-ce-assumptions} are satisfied by Lemma~\ref{lemma:hilbert-euler})
applied to exact sequence~\eqref{eq:x222-cex-cex-1}
the natural morphism~\mbox{$f \colon \rM^\circ \to S$} is bijective on geometric points
and for every geometric point~$[E] \in \rM^\circ$, the bundle~$E$ is exceptional.
Therefore, $f$ is an isomorphism by Corollary~\ref{cor:moduli-etale}.

By Proposition~\ref{prop:x222-ec} every sheaf parameterized by the moduli space~$\rM^\circ$ is $H_X$-stable.
Therefore, applying Proposition~\ref{prop:twisted-universal} 
we obtain a Brauer class~$\upbeta_\cE \in \Br(S)$ on~$\rM^\circ \cong S$
and a $p^*(\upbeta_\cE)$-twisted universal family~$\cE_X$ on~$X \times_S \rM^\circ = X$.
Applying Lemma~\ref{lemma:beta-2} and~\eqref{eq:x222-cex-cex-1}, we obtain~\mbox{$\upbeta_\cE^2 = 1$}.

Next, the proof of the semiorthogonal decomposition follows easily from~\eqref{eq:ec-x222}
by the argument used in the proof of Theorems~\ref{thm:x111} and~\ref{thm:x22}.

Finally, if~$X(S) \ne \varnothing$, if~$i \colon S \to X$ is a section of~$X \to S$,
and if~$\cF_X$ is the right mutation of~$\cE_X$ through~$\uppsi^* \big( \Phi(\cO_{\P_{S'}(V)}(1) \otimes \bD(S')) \big)$
(this is a global version of the vector bundle from~\eqref{eq:x222-cex-cfx}) 
then~$i^*\cF_X$ is a $\upbeta_\cE$-twisted vector bundle of rank~$2$ on~$S$, 
so that~$\upbeta_\cE$ is represented by the conic bundle~$\P_S(i^*\cF_X)$.
\end{proof}

\subsection{Forms of~$\sX_{4,4}$}
\label{ss:x44}

Recall that according to notation from the Introduction 
over an algebraically closed field every Fano threefold of type~$\sX_{4,4}$ 
is isomorphic to an intersection of the graph of the Cremona transformation~$\P^5 \dashrightarrow \P^5$ 
with~$\P^4 \times \P^4 \subset \P^5 \times \P^5$.
It has Picard number~2, anticanonical degree~28, and number~2.21 in Mori--Mukai classification.
It is easy to see (e.g., \cite[Lemma~2.4(ii)]{KP21}) that 
each of the projections~$X \to \P^4$ is birational onto a smooth quadric~$Q^3 \subset \P^4$.
Using Proposition~\ref{prop:forms-general} we obtain a description of all Fano fibrations of this type.

\begin{proposition}
\label{prop:forms-x44}
If~$p \colon X \to S$ is a smooth Fano fibration with fibers of type~$\sX_{4,4}$
there is an \'etale covering~$S' \to S$ of degree~$2$ with connected~$S'$,
a smooth projective morphism~$Z \to S'$ with fiber~$Q^3$, 
and a closed embedding 
\begin{equation*}
\uppsi \colon X \hookrightarrow \Res_{S'/S}(Z)
\end{equation*}
which over each geometric point~$s \in S$ coincides with the natural embedding~$X_s \hookrightarrow Q^3 \times Q^3$.
\end{proposition}

\begin{proof}
We apply Proposition~\ref{prop:forms-general} with~$Y_s = Q^3$, $d = 2$, and~$\upxi_s$ being the natural embedding.
Then~$Z \to S'$ is a smooth fibration with fiber~$Q^3$, 
and we obtain the required closed embedding~\mbox{$\uppsi \colon X \hookrightarrow \Res_{S'/S}(Z)$}.
\end{proof}

Now to describe the derived category, we start with the case of a variety~$X \subset Q^3 \times Q^3$ 
of type~$\sX_{4,4}$ over an algebraically closed field.
Let~$h_i$, $i = 1,2$, be the pullbacks to~$X$ of the hyperplane classes of the factors and~$H_X \coloneqq h_1 + h_2$.

\begin{lemma}
\label{lemma:x44-picrd}
Let~$X \subset Q_1 \times Q_2$ be a variety of type~$\sX_{4,4}$ over an algebraically closed field,
where~$Q_1$ and~$Q_2$ are smooth $3$-dimensional quadrics.
Each projection
\begin{equation*}
\pr_{i} \colon X \to Q_i,
\qquad 
i = 1,2, 
\end{equation*}
is the blowup of~$Q_i$ along a smooth rational quartic curve~$\Gamma_i \subset Q_i$.
The Picard group of~$X$ is freely generated by~$h_1$ and~$h_2$, 
and if~$E_i$ is the exceptional divisor of the projection~$\pr_i$, we have
\begin{equation}
\label{eq:x44-picard}
E_1 = 2h_1 - h_2,
\qquad 
E_2 = 2h_2 - h_1,
\qquad 
K_X = -h_1 -h_2 = -H_X.
\end{equation} 
The cone of effective divisors of~$X$ is generated by the classes~$h_1$, $h_2$, $E_1$, $E_2$, 
and each of these has intersection product~$1$ with the symmetric curve class 
\begin{equation}
\label{eq:delta-x44}
\updelta_X \coloneqq \frac1{14}(h_1 + h_2)^2.
\end{equation}
\end{lemma}

\begin{proof}
The first part is easy, see~\cite[Lemma~2.4(ii), Lemma~2.3]{KP21}
and the second is analogous to Lemma~\ref{lemma:x222-picrd}.
\end{proof}

We denote by~$\rf_i$ the class of a fiber of~$E_i \to \Gamma_i$
and by~$\cS_i$ the spinor bundle on the quadric~$Q_i$ and its pullback to~$X$.
We will also need the following result.

\begin{lemma}
\label{lemma:x44-spinor}
Let~$X \subset Q_1 \times Q_2$ be a Fano threefold of type~$X_{4,4}$ 
over an algebraically closed field~$\kk$.
Then there are exact sequences
\begin{equation*}
0 \to \cS_2^\vee \to \cO_X(h_1)^{\oplus 2} \to \cO_{E_1}(5\rf_1) \to 0,
\qquad 
0 \to \cS_1^\vee \to \cO_X(h_2)^{\oplus 2} \to \cO_{E_2}(5\rf_2) \to 0,
\end{equation*}
where the second arrows are induced by the evaluation morphisms~$\cO_X^{\oplus 2} \to \cO_{E_i}(\rf_i)$.
\end{lemma}

\begin{proof}
We prove the first sequence; the proof of the second is analogous.
Denote the kernel of the map~$\cO_X(h_1)^{\oplus 2} \to \cO_{E_1}(5\rf_1)$ by~$\cK$,
so that we have an exact sequence
\begin{equation}
\label{eq:sequence-ck}
0 \to \cK \to \cO_X(h_1)^{\oplus 2} \to \cO_{E_1}(5\rf_1) \to 0,
\end{equation}
Let us check that the bundle~$\cK$ is $H_X$-stable.
Indeed, using~\eqref{eq:sequence-ck} and~\eqref{eq:x44-picard} 
we see that~$\rank(\cK) = 2$ and~$\rc_1(\cK) = 2h_1 - E_1 = h_2$, 
hence the normalized slope of~$\cK$ is
\begin{equation*}
\upmu(\cK) = 
\frac{\rc_1(\cK) \cdot \updelta_X}{\rank(\cK)} =
\frac12.
\end{equation*}
Assume~$\cG \subset \cK$ is a destabilizing reflexive sheaf.
Then~$\rank(\cG) = 1$, hence~$\cG$ is a line bundle, 
and~$\cG$ has a nontrivial morphism to~$\cO_X(h_1)$, hence~$\cG \cong \cO_X(h_1 - D)$, where~$D$ is effective.
But the slope of~$\cG$ must be positive, hence~$(h_1 - D) \cdot \updelta_X = 1 - D \cdot \updelta_X \ge 1$, hence~$D = 0$.
Thus, $\cG \cong \cO(h_1)$, which is absurd because~$\Hom(\cO_X(h_1), \cK) = 0$ by definition of~$\cK$.

Next we note that~$\cK$ is numerically equivalent to~$\cS_2^\vee$;
indeed, the Chern classes of~$\cK$ can be computed from~\eqref{eq:sequence-ck}, 
and it is easy to see that they are the same as those of~$\cS_2^\vee$
(we already have seen this for~$\rc_1$).
Now we apply Proposition~\ref{prop:uniqueness-general} to~$\cU = \cS_2 \cong \cS_2^\vee(-h_2)$, $\cE = \cK(-h_2)$, 
and the pullback to~$X$ of sequence~\eqref{eq:exact-spinor},
and conclude that~$\cK \cong \cS_2^\vee$, as we need.
\end{proof}

In the next proposition we use freely the notation introduced in Lemma~\ref{lemma:x44-picrd} and Lemma~\ref{lemma:x44-spinor}.

\begin{proposition}
\label{prop:x44-ec}
Let~$X \subset Q_1 \times Q_2$ be a Fano threefold of type~$\sX_{4,4}$ 
over an algebraically closed field~$\kk$.
Then there is a semiorthogonal decomposition
\begin{equation}
\label{eq:ec-x44}
\bD(X) = \langle \cE_X, \cO, \cS_1^\vee, \cS_2^\vee, \cO_X(h_1), \cO_X(h_2) \rangle
\end{equation}
where~$\cE_X$ is a stable exceptional vector bundle of rank~$3$ with~$\det(\cE_X) \cong \cO_X(-H_X)$.
Moreover, there are unique exact sequences
\begin{equation}
\label{eq:x44-cex-cexvee}
0 \to \cE_X \to \cO_X^{\oplus 8} \to {\cS_1^\vee}^{\oplus 2} \oplus {\cS_2^\vee}^{\oplus 2} \to \cE_X^\vee \to 0,
\end{equation}
and
\begin{equation}
\label{eq:x44-cexvee-cex}
0 \to \cE_X^\vee \to \cO_X(h_1)^{\oplus 3} \oplus \cO_X(h_2)^{\oplus 3} \to \cE_X(H_X) \to 0
\end{equation} 
\end{proposition}

\begin{proof}
Consider the projection~$\pr_1 \colon X \to Q_1$,
which, recall, is the blowup of a smooth rational quartic curve~$\Gamma_1 \subset Q_1$.
Using the blowup formula we obtain a semiorthogonal decomposition
\begin{equation*}
\bD(X) = \langle \pr_1^*(\bD(Q_1)), i_*q^*\bD(\Gamma_1) \rangle,
\end{equation*}
where $q \colon E_1 \to \Gamma_1$ is the exceptional divisor of the blowup~$\pr_1$,
and $i \colon E_1 \hookrightarrow X$ is its embedding.

Now we choose exceptional collections in the two components of the above decomposition.
In the first component we choose one of the standard exceptional collections
\begin{equation*}
\bD(Q_1) = \langle \cO(-h_1), \cO, \cS_1^\vee, \cO(h_1) \rangle.
\end{equation*}
Combining it with~$\bD(\Gamma_1) = \langle \cO_{\Gamma_1}(5), \cO_{\Gamma_1}(6) \rangle$,
we obtain a semiorthogonal decomposition
\begin{equation*}
\bD(X) = \langle \cO(-h_1), \cO, \cS_1^\vee, \cO(h_1), \cO_{E_1}(5\rf_1), \cO_{E_1}(6\rf_1) \rangle.
\end{equation*}
Now we find a sequence of mutations that transforms it into the form~\eqref{eq:ec-x44}. 

First, we mutate~$\cO(-h_1)$ to the far right.
Since~$-K_X = h_1 + h_2$, we obtain the collection
\begin{equation*}
\bD(X) = \langle \cO, \cS_1^\vee, \cO(h_1), \cO_{E_1}(5\rf_1), \cO_{E_1}(6\rf_1), \cO(h_2) \rangle.
\end{equation*}
Now we mutate~$\cO_{E_1}(6\rf_1)$ to the right.
Using relations~\eqref{eq:x44-picard} and~$\cO_{E_1}^\vee \cong \cO_{E_1}(E_1)[-1]$, we obtain
\begin{multline*}
\Ext^\bullet( \cO_{E_1}(6\rf_1), \cO(h_2) ) 
\cong \Ext^\bullet( \cO(h_2)^\vee, \cO_{E_1}(6\rf_1)^\vee ) 
\\
\cong \Ext^\bullet( \cO(E_1 - 2h_1), \cO_{E_1}(E_1 - 6\rf_1)[-1] ) 
\cong H^\bullet(X, \cO_{E_1}(2\rf_1))[-1] \cong \kk^3[-1],
\end{multline*}
therefore, the mutation triangle takes the form of the canonical extension
\begin{equation}
\label{eq:x44-cex}
0 \to \cO_X(h_2)^{\oplus 3} \to \cE_X(h_1 + h_2) \to \cO_{E_1}(6\rf_1) \to 0,
\end{equation}
where we denote the middle term (the result of the mutation) by~$\cE_X(h_1 + h_2)$ for further convenience.
Note that~$\cE_X$ is a coherent sheaf by construction.
Moreover, taking the dual of~\eqref{eq:x44-cex} and using the above identifications, we obtain a triangle
\begin{equation}
\label{eq:x44-cexvee}
\cE_X^\vee(-h_1-h_2) \to \cO_X(E_1 - 2h_1)^{\oplus 3} \to \cO_{E_1}(E_1 - 6\rf_1),
\end{equation}
where the second arrow is the evaluation morphism, hence it is surjective.
This proves that~$\cE_X^\vee$ is locally free, hence the same is true for~$\cE_X$.

Now we mutate~$\cE_X(h_1 + h_2)$ to the far left, and, using again the equality~$K_X = -h_1 - h_2$,
obtain the exceptional collection
\begin{equation*}
\bD(X) = \langle \cE_X, \cO, \cS_1^\vee, \cO(h_1), \cO_{E_1}(5\rf_1), \cO(h_2) \rangle.
\end{equation*}
Finally, we mutate $\cO_{E_1}(5\rf_1)$ one step to the left.
We have~$\Ext^\bullet(\cO(h_1), \cO_{E_1}(5\rf_1)) = \kk^2$, 
so applying Lemma~\ref{lemma:x44-spinor} we conclude that the corresponding mutation is given by the dual spinor bundle~$\cS_2^\vee$,
hence we obtain the required exceptional collection~\eqref{eq:ec-x44}.

The rank and determinant of~$\cE_X$ are easy to compute from~\eqref{eq:x44-cex}, 
so it remains to show the stability of~$\cE_X$ 
and to construct the exact sequences~\eqref{eq:x44-cex-cexvee} and~\eqref{eq:x44-cexvee-cex}.

We start with constructing the exact sequences.
For this we take~\eqref{eq:ec-x44}, mutate~$\cE_X$ to the far right, 
and then mutate the obtained object~$\cE_X(h_1 + h_2)[-3]$ two steps to the left.
The first mutation (through~$\cO(h_2)$) is given by the exact sequence~\eqref{eq:x44-cex}, 
hence the result of this mutation is the object~$\cO_{E_1}(6\rf_1)[-3]$.
Furthermore,
\begin{equation*}
\Ext^\bullet(\cO(h_1), \cO_{E_1}(6\rf_1)) \cong H^\bullet(X, \cO_{E_1}(2\rf_1)) = \kk^3,
\end{equation*}
and the evaluation morphism is surjective, hence the mutation triangle takes the form
\begin{equation}
\label{eq:x44:cfx}
0 \to \cF_X \to \cO_X(h_1)^{\oplus 3} \to \cO_{E_1}(6\rf_1) \to 0
\end{equation} 
so that the result of the mutation is~$\cF_X[-2]$.
Note that this triangle coincides with a twist of~\eqref{eq:x44-cexvee} by~$h_1 + h_2 = 3h_1 - E_1$,
hence~$\cF_X \cong \cE_X^\vee$.
So, merging~\eqref{eq:x44:cfx} with~\eqref{eq:x44-cex} and taking into account the equality~$\Ext^1(\cO(h_1),\cO(h_2)) = 0$, 
we obtain the exact sequence~\eqref{eq:x44-cexvee-cex}.

To construct~\eqref{eq:x44-cex-cexvee} it remains to note 
that the mutation of~$\cE_X^\vee[-2]$ through~$\langle \cO_X, \cS_1^\vee, \cS_2^\vee \rangle$ is~$\cE_X$,
therefore there is a distinguished triangle
\begin{equation*}
A^\bullet \otimes \cO_X \to B_1^\bullet \otimes \cS_1^\vee \oplus B_2^\bullet \otimes \cS_2^\vee \to \Cone(\cE_X^\vee \to \cE_X[2])
\end{equation*}
for some graded vector spaces~$A^\bullet$, $B_1^\bullet$, and~$B_2^\bullet$.
In other words, the cone of the first arrow has two cohomology sheaves, $\cE_X$ in degree~$-1$ and~$\cE_X^\vee$ in degree~$0$.
Looking at the cohomology exact sequence and taking into account 
semiorthogonality of the pairs~$(\cO_X,\cE_X^\vee)$ and~$(\cS_i^\vee,\cE_X)$, $i = 1,2$, 
we conclude that the triangle implies an exact sequence
\begin{equation*}
0 \to \cE_X \to A^0 \otimes \cO_X \to B_1^0 \otimes \cS_1^\vee \oplus B_2^0 \otimes \cS_2^\vee \to \cE_X^\vee \to 0.
\end{equation*}
Comparing the first Chern classes and ranks, we see that~$\dim(B_1^0) = \dim(B_2^0) = 2$, \mbox{$\dim(A^0) = 8$},
so that this sequence takes the form~\eqref{eq:x44-cex-cexvee}.

Now it remains to prove the stability of~$\cE_X$.
By Lemma~\ref{lemma:x44-picrd} the {\sf normalized} slope of~$\cE_X$ is
\begin{equation*}
\upmu(\cE_X) = 
\frac{ \rc_1(\cE_X) \cdot \updelta_X} {\rank(\cE_X)}
= - \frac{ (h_1 + h_2) \cdot \updelta_X}3
= - \frac23.
\end{equation*}
Therefore, to check the stability of~$\cE_X$ or~$\cE_X^\vee$, it is enough to exclude the following possibilities:
\begin{aenumerate}
\item 
\label{item:x44-rank-1}
$\cG \subset \cE_X$ is a reflexive subsheaf of rank~$1$ with~$\upmu(\cG) \ge 0$;
\item 
\label{item:x44-rank-2}
$\cG \subset \cE_X^\vee$ is a reflexive subsheaf of rank~$1$ with~$\upmu(\cG) \ge 1$.
\end{aenumerate}
Assume~\ref{item:x44-rank-1}.
Then~$\cG$ is a line bundle and by~\eqref{eq:x44-cex-cexvee} it has a nontrivial morphism to~$\cO_X$, 
therefore~\mbox{$\cG \cong \cO_X(-D)$} with effective~$D$.
Furthermore,~$D \cdot \updelta_X = -\upmu(\cG) \le 0$, 
hence~$D = 0$ by Lemma~\ref{lemma:x44-picrd} and~$\cG \cong \cO_X$.
But~$\Hom(\cO_X,\cE_X) = 0$ by semiorthogonality in~\eqref{eq:ec-x44}, hence~\ref{item:x44-rank-1} is impossible. 

A similar argument (with~\eqref{eq:x44-cex-cexvee} replaced by~\eqref{eq:x44-cexvee-cex}) 
shows that if~\ref{item:x222-rank-2} holds then~$\cG \cong \cO(h_i)$,
which contradicts semiorthogonality of~\eqref{eq:ec-x44}.
Thus, \ref{item:x222-rank-2} is also impossible. 
\end{proof}

Now that we know a good symmetric exceptional collection for~$X$ over algebraically closed fields, 
we can pass to fibrations over any base scheme~$S$.
Recall the description of~$X$ given in Proposition~\ref{prop:forms-x44},
in particular the double covering~$S' \to S$ and the quadric fibration~$q \colon Z \to S'$.

\begin{theorem}
\label{thm:x44}
If $X/S$ is a form of a Fano threefold of type~$\sX_{4,4}$ there is a semiorthogonal decomposition
\begin{equation*}
\bD(X) = \langle
\cE_X \otimes \bD(S), 
\cO_X \otimes \bD(S), 
\uppsi^* \big( \Phi(\cS^\vee \otimes \bD(S',\upbeta_\cS)) \big), 
\uppsi^* \big( \Phi(\cO_{Z}(H_Z) \otimes \bD(S')) \big)
\rangle,
\end{equation*}
where~$\upbeta_\cS$ is a $2$-torsion Brauer class on~$S'$,
$\cS$ is a $q^*(\upbeta_\cS)$-twisted spinor bundle of rank~$2$ on~$Z$,
and~$\cE_X$ is an $S$-exceptional vector bundle of rank~$3$ on~$X$.

Moreover, if~$X(S) \ne \varnothing$ then~$\upbeta_\cS \in \Br(S')$ can be represented by a conic bundle. 
\end{theorem}

\begin{proof}
First, we construct a global version of the bundle~$\cE_X$ by using the argument of Proposition~\ref{prop:x16-forms}.
Consider the relative moduli space 
\begin{equation*}
\rM \coloneqq \rM_{X/S}(3; -H_X, 6\updelta_X, -2P_X),
\end{equation*}
where~$H_X$ is the fundamental class, 
the class~$\updelta_X$ is defined in~\eqref{eq:delta-x44}, 
and~$P_X$ is the class of a point.
Let also $\rM^\circ \subset \rM$ be the open subscheme parameterizing bundles~$E$ on~$X_s$ with the vanishings
\begin{equation*}
\rH^\bullet(X_s,E(-h_1)) = \rH^\bullet(X_s,E(-h_2)) = 0.
\end{equation*}
By Propositions~\ref{prop:x44-ec} and Proposition~\ref{prop:uniqueness-general} 
(conditions~\eqref{eq:cu-ce-assumptions} follow from numerical equivalence)
applied to the exact sequence~\eqref{eq:x44-cexvee-cex},
the natural morphism~\mbox{$f \colon \rM^\circ \to S$} is bijective on geometric points
and for every geometric point~$[E] \in \rM^\circ$, the bundle~$E$ is exceptional.
Therefore, $f$ is an isomorphism by Corollary~\ref{cor:moduli-etale}.

By Proposition~\ref{prop:x222-ec} every sheaf parameterized by the moduli space~$\rM^\circ$ is $H_X$-stable.
Therefore, applying Proposition~\ref{prop:twisted-universal} 
we obtain a Brauer class~$\upbeta_\cE \in \Br(S)$ on~$\rM^\circ \cong S$
and a $p^*(\upbeta_\cE)$-twisted universal family~$\cE_X$ on~$X \times_S \rM^\circ = X$.
Let 
\begin{equation*}
W \coloneqq (p_*\cE_X^\vee)^\vee;
\end{equation*}
this is a~$\upbeta_\cE$-twisted vector bundle on~$S$ of rank~8.
Therefore, $\upbeta_\cE^8 = 1$.
On the other hand, we have~$\wedge^3\cE_X \cong \cO_X(-H_X)$ is untwisted, hence~$\upbeta_\cE^3 = 1$.
From these two observations it follows that~$\upbeta_\cE = 1$, hence~$\cE_X$ is untwisted.

Now let~$\upbeta_\cS \in \Br(S')$ be the 2-torsion Brauer class 
and let~$\cS$ be the~$q^*(\upbeta_\cS)$-twisted spinor bundle on the smooth 3-dimensional quadric bundle~$q \colon Z \to S'$, 
constructed in Theorem~\ref{thm:db-quadric}.
Then we have a semiorthogonal collection of $S'$-admissible subcategories
\begin{equation*}
\langle \cO_Z \otimes \bD(S'), \cS^\vee \otimes \bD(S',\upbeta_\cS), \cO_Z(H_Z) \otimes \bD(S') \rangle \subset \bD(Z).
\end{equation*}
Applying Theorem~\ref{thm:pmd-qs-ps} we obtain the last three components in the required semiorthogonal decomposition of~$\bD(X)$.
Using the bundle~$\cE_X$ constructed above, we obtain the first component.
To check full faithfulness, semiorthogonality and generation, 
we apply Proposition~\ref{prop:relative-sod}.
Accordingly, we need to consider the case where~$S$ is the spectrum of an algebraically closed field.
In this case the required semiorthogonal decomposition was constructed in Proposition~\ref{prop:x44-ec}.

Finally, if~$X(S) \ne \varnothing$ then~$Z(S') \ne \varnothing$ by~\eqref{eq:res-adjunction},
and if~$i \colon S' \to Z$ is a section of~$Z \to S'$ 
then~$i^*\cS$ is a $\upbeta_\cS$-twisted vector bundle of rank~$2$ on~$S'$, 
so that~$\upbeta_\cS$ is represented by the conic bundle~$\P_{S'}(i^*\cS)$.
\end{proof}

\subsection{Forms of~$\sX_{3,3}$}

Recall that according to notation from the Introduction 
over an algebraically closed field every Fano threefold of type~$\sX_{3,3}$ 
is isomorphic to a linear section of~$\P^3 \times \P^3$ of codimension~3.
Using Proposition~\ref{prop:forms-general} we obtain a description of all Fano fibrations of this type.

Recall that for an \'etale double covering~$S' \to S$ 
we denote the action of the Galois involution of~$S'$ over~$S$ on~$\Br(S')$ by~$\upbeta' \mapsto \bar\upbeta'$.

\begin{proposition}
\label{prop:x33}
If~$p \colon X \to S$ is a smooth Fano fibration with fibers of type~$\sX_{3,3}$
there is an \'etale covering~$S' \to S$ of degree~$2$ with connected~$S'$, 
a $4$-torsion Brauer class~$\upbeta' \in \Br(S')$ such that 
\begin{equation}
\label{eq:x33-beta}
\bar\upbeta' = {\upbeta'}^{-1},
\end{equation} 
and a $\upbeta'$-twisted vector bundle~$V$ on~$S'$ of rank~$4$ 
such that~$X \subset \Res_{S'/S}(\P_{S'}(V))$.
Moreover, if~$W \coloneqq \cores_{S'/S}(V)$ is the Segre bundle then~$W$ is untwisted and
there is a vector bundle~$A$ of rank~$3$ on~$S$ and an epimorphism~\mbox{$\varphi \colon W \twoheadrightarrow A^\vee$} 
such that 
\begin{equation}
\label{eq:form-x33}
X = \Res_{S'/S}(\P_{S'}(V))) \times_{\P_S(W)} \P_S(\Ker(\varphi)) 
\end{equation}
where the morphism $\Res_{S'/S}(\P_{S'}(V)) \to \P_S(W)$ in the fiber product is the Segre embedding.
\end{proposition}

\begin{proof}
The proof is analogous to the proof of Proposition~\ref{prop:x22}.
\end{proof}

Using this description and Theorem~\ref{thm:pmd-qs-ps} we can construct a semiorthogonal decomposition.
Recall that~$H_V$ denotes the fundamental class of~$\P_{S'}(V)$.

\begin{theorem}
\label{thm:x33}
If $X/S$ is a form of a Fano threefold of type~$\sX_{3,3}$ then there is a semiorthogonal decomposition
\begin{equation*}
\bD(X) = \langle
\cA_X,
\cO_X \otimes \bD(S), 
\uppsi^* \big( \Phi(\cO_{\P_{S'}(V)}(H_V) \otimes \bD(S',\upbeta')) \big) 
\rangle,
\end{equation*}
where~$\cA_X \subset \bD(X)$ is an $S$-linear admissible triangulated subcategory. 
Moreover, the base change of~$\cA_X$ along the double covering~$S' \to S$ has a semiorthogonal decomposition
\begin{equation}
\label{eq:x33-cax-sp}
(\cA_X)_{S'} = \langle \bD(\Gamma'), \cE \otimes \bD(S', {\upbeta'}^2) \rangle,
\end{equation}
where $\Gamma'/S'$ is a smooth curve of genus~$3$ and~$\cE$ 
is a $p^*({\upbeta'}^2)$-twisted $S'$-exceptional vector bundle of rank~$3$ on~$X \times_S S'$.
\end{theorem}

\begin{proof}
The proof of the first part is analogous to the proof of Theorem~\ref{thm:x111}.
More precisely, we use Theorem~\ref{thm:pmd-qs-ps} to construct the last two components
and define the category~$\cA_X$ as their orthogonal.
So, it remains to describe the base change~$(\cA_X)_{S'}$.

Now we prove the second part.
The embedding~$\uppsi \colon X \to \Res_{S'/S}(\P_{S'}(V))$ constructed in Proposition~\ref{prop:x33} 
after base change along~$f \colon S' \to S$ gives an embedding
\begin{equation*}
\uppsi' \colon X' \coloneqq X \times_S S' \hookrightarrow \Res_{S'/S}(\P_{S'}(V)) \times_S S' 
\cong \P_{S'}(V) \times_{S'} \P_{S'}(\uptau^*V),
\end{equation*}
where~$\uptau \colon S' \to S'$ is the involution of the double covering~$f$.
Consider the first projection
\begin{equation*}
\pr_1 \colon X' \to \P_{S'}(V).
\end{equation*}
On each geometric fiber over~$S'$ this is the blowup of~$\P^3$ 
along a smooth sextic curve of genus~$3$ (see, e.g., \cite[Lemma~2.4(iii)]{KP21}), 
hence the same is true globally, i.e., 
there is a smooth projective morphism~$\Gamma' \to S'$ with geometric fibers curves of genus~$3$,
and an embedding~$\Gamma' \hookrightarrow \P_{S'}(V)$ (of relative degree~6)
such that
\begin{equation*}
X' \cong \Bl_{\Gamma'}(\P_{S'}(V)).
\end{equation*}
Using the blowup formula, we obtain a semiorthogonal decomposition
\begin{equation*}
\bD(X') = \langle i_*q^*\bD(\Gamma') \otimes \omega_{X'/S'}, \pr_1^*(\bD(\P_{S'}(V))) \rangle,
\end{equation*}
where $q \colon E_1 \to \Gamma'$ is the exceptional divisor of the blowup~$\pr_1$ 
and~$i \colon E_1 \hookrightarrow X'$ is its embedding.
Let~$h_1$ and~$h_2$ denote the fundamental classes of~$\P_{S'}(V)$ and~$\P_{S'}(\uptau^*V)$, respectively, 
as well as their pullbacks to~$X'$.
Then using one of standard (twisted) exceptional collections for~$\P_{S'}(V)$ 
we can rewrite the above semiorthogonal decomposition as
\begin{equation*}
\bD(X') = 
\langle i_*q^*\bD(\Gamma') \otimes \omega_{X'/S'}, \cO(-h_1), \cT_1(-2h_1), \cO, \cO(h_1) \rangle
\end{equation*}
(recall from Example~\ref{ex:p3} that~$\cO(kh_1)$ is~${\upbeta'}^{-k}$-twisted, 
while~$\cT_1$ is the relative tangent bundle for~$X'/S'$ and so~$\cT_1(-2h_1)$ is~${\upbeta'}^{-2}$-twisted,
and we omit the derived categories~$\bD(S',{\upbeta'}^k)$ which should appear as the corresponding factors)
and apply a couple of mutations.
First, we mutate~$i_*q^*\bD(\Gamma') \otimes \omega_{X'/S'}$ one step to the right,
obtaining the decomposition
\begin{equation*}
\bD(X') = \langle \cO(-h_1), \bR_{\cO(-h_1)}(i_*q^*\bD(\Gamma') \otimes \omega_{X'/S'}), \cT_1(-2h_1), \cO, \cO(h_1) \rangle.
\end{equation*}
After that we mutate~$\cO(-h_1)$ to the far right,
and since~$K_{X'} = -h_1 - h_2$, we obtain the decomposition
\begin{equation*}
\bD(X') = \langle \bR_{\cO(-h_1)}(i_*q^*\bD(\Gamma') \otimes \omega_{X'/S'}), \cT_1(-2h_1), \cO, \cO(h_1), \cO(h_2) \rangle.
\end{equation*}
We note that the base change to~$S'$ of~$\uppsi^* \big( \Phi(\cO_{\P_{S'}(V)}(H_V) \otimes \bD(S',\upbeta')) \big)$ 
is the $S'$-linear subcategory of~$\bD(X')$ 
generated by~$\cO(h_1)$ and~$\cO(h_2)$, therefore we obtain~\eqref{eq:x33-cax-sp} with~$\cE = \cT_1(-2h_1)$.
\end{proof}

\subsection{Forms of~$\sX_{1,1,1,1}$}

Recall that according to notation from the Introduction 
over an algebraically closed field every Fano threefold of type~$\sX_{1,1,1,1}$ 
is isomorphic to a hyperplane section of~$\P^1 \times \P^1 \times \P^1 \times \P^1$.
Using Proposition~\ref{prop:forms-general} we easily obtain a description of all Fano fibrations of this type.

\begin{proposition}
\label{prop:forms-x1111}
If $p \colon X \to S$ is a smooth Fano fibration with fibers of type~$\sX_{1,1,1,1}$
then there is an \'etale covering~$f \colon S' \to S$ of degree~$4$ with connected~$S'$, 
a~$2$-torsion Brauer class~$\upbeta' \in \Br(S')$ such that
\begin{equation*}
\cores_{S'/S}(\upbeta') = 1,
\end{equation*}
and a $\upbeta'$-twisted vector bundle~$V$ of rank~$2$ on~$S'$
such that~$X \subset \Res_{S'/S}(\P_{S'}(V))$.
Moreover, if~$W \coloneqq \cores_{S'/S}(V)$ is the Segre bundle then~$W$ is untwisted and
there is an untwisted line bundle~$\cL$ on~$S$, 
and an epimorphism~$\varphi \colon W \to \cL^\vee$ such that 
\begin{equation}
\label{eq:form-x1111}
X = \Res_{S'/S}(\P_{S'}(V)) \times_{\P_S(W)} \P(\Ker(\varphi)),
\end{equation}
where the morphism $\Res_{S'/S}(\P_{S'}(V)) \to \P_S(W)$ in the fiber product is the Segre embedding.
\end{proposition}

\begin{proof}
The proof is analogous to the proofs of Proposition~\ref{prop:x22} and Proposition~\ref{prop:x33}.
\end{proof}

Using this description and Theorem~\ref{thm:pmd-qs-ps} we can construct a semiorthogonal decomposition.
Recall that~$H_V$ denotes the fundamental class of~$\P_{S'}(V)$.

\begin{theorem}
\label{thm:x1111}
If $X/S$ is a form of a Fano threefold of type~$\sX_{1,1,1,1}$ then there is a semiorthogonal decomposition
\begin{equation*}
\bD(X) = \langle
\cA_X,
\cO_X \otimes \bD(S), 
\uppsi^* \big( \Phi(\cO_{\P_{S'}(V)}(H_V) \otimes \bD(S',\upbeta')) \big)
\rangle,
\end{equation*}
where~$\cA_X \subset \bD(X)$ is an $S$-linear admissible triangulated subcategory. 
Moreover, the base change of~$\cA_X$ along the covering~$S' \to S$ has a semiorthogonal decomposition
\begin{equation*}
(\cA_X)_{S'} = \langle \bD(\Gamma'), \bD(S'', g_1^*(\upbeta') \cdot g_2^*(\upbeta')) \rangle,
\end{equation*}
where $\Gamma'/S'$ is a smooth curve of genus~$1$,
$S'' \subset S' \times_S S'$ is the complement of the diagonal,
and~$g_1,g_2 \colon S'' \to S'$ are the projections.
\end{theorem}

\begin{proof}
The proof of the first part is analogous to the proof of Theorem~\ref{thm:x111}.
More precisely, we use Theorem~\ref{thm:pmd-qs-ps} to construct the last two components
and define the category~$\cA_X$ as their orthogonal.
So, it remains to describe the base change~$(\cA_X)_{S'}$.

The embedding~$\uppsi \colon X \to \Res_{S'/S}(\P_{S'}(V))$ constructed in Proposition~\ref{prop:x33} 
after base change along~$f \colon S' \to S$ gives an embedding
\begin{equation*}
\uppsi' \colon X' \coloneqq X \times_S S' \hookrightarrow \Res_{S'/S}(\P_{S'}(V)) \times_S S'
\cong \P_{S'}(V) \times_{S'} Y,
\end{equation*}
where recall that~$S''$ is defined as the complement of the first (diagonal) component in
\begin{equation}
\label{eq:sp-sp}
S' \times_S S' = S' \sqcup S''
\end{equation}
the maps~$g_1,g_2 \colon S'' \to S'$ are induced by the projections~$S' \times_S S' \to S'$,
so that they are \'etale covering of degree~$3$, and
\begin{equation*}
Y \coloneqq \Res_{g_1}(\P_{S''}(g_2^*V)).
\end{equation*}
Note that~$Y \to S'$ is a Fano fibration of type~$\sX_{1,1,1}$, see~\S\ref{ss:x111}.
Note also that the Brauer class of the bundle~$g_2^*(V)$ on~$S''$ is~$g_2^*(\upbeta')$;
therefore, using~\eqref{eq:sp-sp} and compatibility of corestriction with base change, we obtain
\begin{equation*}
\upbeta' \cdot \cores_{g_1}(g_2^*(\upbeta')) = 
\cores_{S' \times_S S'/S'}(\upbeta') =
f^*(\cores_{S'/S}(\upbeta')) = 
f^*(1) = 1,
\end{equation*}
and since~$\upbeta'$ is a 2-torsion class, we conclude that~$\cores_{g_1}(g_2^*(\upbeta')) = \upbeta'$.

Now consider the composition of~$\uppsi'$ with the second projection~$\P_{S'}(V) \times_{S'} Y \to Y$:
\begin{equation*}
\pr_2 \colon X' \to Y.
\end{equation*}
On each geometric fiber over~$S'$ this is a blowup of~$\P^1 \times \P^1 \times \P^1$ 
along a smooth curve of genus~$1$ (see, e.g., \cite[Lemma~2.4(iv)]{KP21}), 
hence the same is true globally, i.e., 
there is a smooth projective morphism~$\Gamma' \to S'$ with geometric fibers curves of genus~$1$,
and an embedding~$\Gamma' \hookrightarrow Y$ such that
\begin{equation*}
X' \cong \Bl_{\Gamma'}(Y).
\end{equation*}
Using the blowup formula, we obtain a semiorthogonal decomposition
\begin{equation*}
\bD(X') = \langle i_*q^*\bD(\Gamma') \otimes \omega_{X'/S'}, \pr_2^*(\bD(Y)) \rangle,
\end{equation*}
where $q \colon E_2 \to \Gamma'$ is the exceptional divisor of the blowup~$\pr_2$ 
and~$i \colon E_2 \hookrightarrow X$ is its embedding.
Using the semiorthogonal decompositon of Theorem~\ref{thm:x111} 
twisted by the opposite of the fundamental class~$H_Y$ 
(note that~$\bB(H_Y) = \upbeta'$ by the above computation together with Lemma~\ref{cor:segre}) 
gives
\begin{multline*}
\bD(X') = 
\langle i_*q^*\bD(\Gamma') \otimes \cO(-H_{X'}), 
\cO(-H_Y) \otimes \bD(S',\upbeta'), \\
\cO(-H_Y) \otimes \Phi'(\cO_{\P_{S''}(g_2^*V)}(H_V) \otimes \bD(S'',g_1^*(\upbeta') \cdot g_2^*(\upbeta'))), \\
\cO_{X'} \otimes \bD(S'), 
\Phi'(\cO_{\P_{S''}(g_2^*V)}(H_V) \otimes \bD(S'',g_2^*(\upbeta'))) \rangle,
\end{multline*}
where~$\Phi'$ is the functor from Theorem~\ref{thm:x111}.

Next, we apply some mutations.
First, we mutate~$i_*q^*\bD(\Gamma') \otimes \cO(-H_{X'})$ one step to the right.
We obtain the decomposition
\begin{multline*}
\bD(X') = 
\langle 
\cO(-H_Y) \otimes \bD(S',\upbeta'), 
\bR_{\cO(-H_Y)}(i_*q^*\bD(\Gamma') \otimes \cO(-H_{X'})), \\
\cO(-H_Y) \otimes \Phi'(\cO_{\P_{S''}(g_2^*V)}(g_2^*H_V) \otimes \bD(S'',g_1^*(\upbeta') \cdot g_2^*(\upbeta'))), \\
\cO_{X'} \otimes \bD(S'), 
\Phi'(\cO_{\P_{S''}(g_2^*V)}(g_2^*H_V) \otimes \bD(S'',g_2^*(\upbeta'))) \rangle.
\end{multline*}
Next, we mutate the first component to the far right.
Since $K_X = - H_{X'}$, we obtain the decomposition
\begin{multline*}
\bD(X') = 
\langle 
\bR_{\cO(-H_Y)}(i_*q^*\bD(\Gamma') \otimes \cO(-H_{X'})), \\
\cO(-H_Y) \otimes \Phi'(\cO_{\P_{S''}(g_2^*V)}(g_2^*H_V) \otimes \bD(S'',g_1^*(\upbeta') \cdot g_2^*(\upbeta'))), 
\cO_{X'} \otimes \bD(S'), \\
\Phi'(\cO_{\P_{S''}(g_2^*V)}(g_2^*H_V) \otimes \bD(S'',g_2^*(\upbeta'))),
\cO(H_{X'} - H_Y) \otimes \bD(S',\upbeta')
\rangle.
\end{multline*}
We note that the base change to~$S'$ of~$\uppsi^* \big( \Phi(\cO_{\P_{S'}(V)}(H_V) \otimes \bD(S',\upbeta')) \big)$ 
coincides with the $S'$-linear subcategory of~$\bD(X')$ 
generated by the last two components of the above decomposition.
Therefore we obtain~\eqref{eq:x33-cax-sp}.
\end{proof}

\appendix

\section{Relative Griffiths components for threefold fibrations}
\label{sec:griffiths-components}

Let $p \colon X \to S$ be a morphism of schemes.
Recall that an admissible subcategory~$\cA \subset \bD(X)$ is {\sf $S$-linear}
if for any perfect complex~$\cF$ on~$S$ one has 
\begin{equation*}
\cA \otimes p^*\cF \subset \cA.
\end{equation*}
Note that this condition for any perfect complex~$\cF$ 
is equivalent to the same condition for a single object~$\cG$, 
if it is a \emph{classical generator} of the category of perfect complexes, 
i.e., if the minimal triangulated subcategory of~$\bD(S)$ containing~$\cG$ and closed under direct summands 
coincides with the perfect derived category of~$S$.
Similarly, a semiorthogonal decomposition~$\bD(X) = \langle \cA_1, \dots, \cA_n \rangle$ is {\sf $S$-linear}
if each component~$\cA_i$ is $S$-linear.

\begin{definition}
\label{def:griffiths}
Let~$p \colon X \to S$ be a smooth proper morphism to a connected scheme~$S$.
Let 
\begin{equation*}
\bD(X) = \langle \cA_1, \dots, \cA_n \rangle
\end{equation*}
be an $S$-linear semiorthogonal decomposition.
A component~$\cA_i$ is called {\sf relative Griffiths component} 
if it does not have further (nontrivial) $S$-linear semiorthogonal decompositions
and does not have a fully faithful $S$-linear embedding into the derived category 
of a smooth proper $S$-variety~$Y$ with~$\dim(Y/S) \le \dim(X/S) - 2$.
\end{definition}

In general it is quite hard to characterize explicitly Griffiths components.
However, in the case where~$\dim(X/S) = 3$ this is easy 
(the easier case of relative dimension~2 is left to the interested reader).
We start with a useful lemma.

\begin{lemma}
\label{lemma:linearity-covering}
Let~$f \colon S' \to S$ be a finite morphism and let~$Y$ be a scheme over~$S'$.
Then any $S$-linear subcategory in~$\bD(Y)$ is also~$S'$-linear.
\end{lemma}

\begin{proof}
Let~$p \colon Y \to S'$ be a morphism, let~$\cA \subset \bD(Y)$ be an~$S$-linear subcategory, 
and let~$\cG$ be a classical generator of the perfect derived category of~$S$.
Then~$\cA \otimes p^*f^*\cG \subset \cA$ because~$\cA$ is $S$-linear.
But~$f^*\cG$ is a classical generator of the perfect derived category of~$S'$ (see, e.g., \cite[Lemma~2.2]{Pir21}),
hence the above inclusion also implies that~$\cA$ is~$S'$-linear.
\end{proof}

\begin{proposition}
\label{prop:non-griffiths}
Let~$Y \to S$ be a smooth proper morphism with~$\dim(Y/S) \le 1$.
If~\mbox{$\cA \subset \bD(Y)$} is an $S$-linearly indecomposable $S$-linear semiorthogonal component 
then~$\cA$ is $S$-inearly equivalent to one of the following categories:
\begin{renumerate}
\item 
$\bD(S')$, where $S'$ is a connected finite \'etale covering of~$S$, or
\item 
$\bD(S',\upbeta)$, where $S'$ is as above 
and~$\upbeta \in \Br(S')$ is the Brauer class of a conic bundle, or
\item 
$\bD(C)$, where $C \to S$ is a smooth proper fibration
such that each connected component of any geometric fiber of~$C \to S$ is a curve of positive genus.
\end{renumerate}
\end{proposition}

\begin{proof}
First, note that if~$Y$ is not connected, then~$\bD(Y)$ has a completely orthogonal decomposition 
and every indecomposable admissible subcategory of~$Y$ is contained in the derived category of one of the components.
Thus, without loss of generality, we may assume that~$Y$ is connected.

Let~$Y \to S' \to S$ be the Stein factorization, so that~$Y \to S'$ has geometrically connected fibers,
$S'$ is connected, and~$f \colon S' \to S$ is a finite \'etale morphism.
By Lemma~\ref{lemma:linearity-covering} any $S$-linear semiorthogonal component in~$\bD(Y)$ is also~$S'$-linear.
Therefore, we can replace~$S$ by~$S'$, or, in other words, we can assume that the fibers of~$Y \to S$ are geometrically connected.

Now let~$p \colon Y \to S$ be a smooth fibration with all geometric fibers connected curves of the same genus or points.
Let~$\cA \subset \bD(Y)$ be an $S$-linear semiorthogonal component.
For each point~$s \in S$, by base change~\cite{K11} we obtain a semiorthogonal component~$\cA_s \subset \bD(Y_{s})$.
By~\cite{Okawa} we have the following possibilities:
\begin{aenumerate}
\item 
\label{item:as-zero}
$\cA_s = 0$, or
\item 
\label{item:as-all}
$\cA_s = \bD(Y_s)$, or
\item 
\label{item:as-oi}
$Y_s \cong \P^1$ and~$\cA_s = \langle \cO(i) \rangle$ for some~$i \in \ZZ$.
\end{aenumerate}
Let~$S^{\mathrm{a}}$, $S^{\mathrm{b}}$, and~$S^{\mathrm{c}}_i$ be the subsets of points of~$S$ for which the corresponding possibilities hold
(the subsets~$S^{\mathrm{c}}_i$ are only defined if~$Y \to S$ is a $\P^1$-fibration, otherwise we set~$S^{\mathrm{c}}_i \coloneqq \varnothing$).
Clearly,
\begin{equation*}
S = S^{\mathrm{a}} \sqcup S^{\mathrm{b}} \sqcup \left( \bigsqcup_{i \in \ZZ} S^{\mathrm{c}}_i \right).
\end{equation*}
On the other hand, as we will show below, each of subsets~$S^{\mathrm{a}}$, $S^{\mathrm{b}}$, and~$S^{\mathrm{c}}_i$ is open in~$S$.
Since~$S$ is connected, it follows that~$S$ coincides with one of these sets.

Let~$\cG$ be a classical generator of~$\cA$ and~$\cG'$ a classical generator of~$\cA^\perp$. 
Note that~$\cG \oplus \cG'$ is a classical generator of~$\bD(Y)$. 
Now, the situations~\ref{item:as-zero}, \ref{item:as-all}, and~\ref{item:as-oi} 
are equivalent to the vanishings at the point~$s$ of the objects~$p_*\cG$, $p_*\cG'$, 
or~$p_*(\cG(-i-1))$ and~$p_*(\cG'(-i))$, respectively.
But the supports of these objects are closed in~$S$, hence the opennes of the corresponding subsets follows.

If~$S = S^{\mathrm{a}}$ then evidently~$\cA = 0$ and if~$S = S^{\mathrm{b}}$ then $\cA^\perp = 0$, hence~$\cA = \bD(Y)$.
Finally, if~$S = S^{\mathrm{c}}_i$ (hence~$Y$ is a $\P^1$-fibration over~$S$)
then~$\bD(Y) = \langle \cA^\perp, \cA \rangle$ is one of the standard decompositions from Theorem~\ref{thm:bernardara}.
Thus, either~$\cA \simeq \bD(S)$, 
or (if~$Y$ is a non-trivial Severi--Brauer variety associated to a Brauer class~$\upbeta$ and~$i$ is odd)~$\cA \simeq \bD(S,\upbeta)$,
and~$\upbeta$ is represented by the conic bundle~$Y$.
\end{proof}

\bibliographystyle{alpha}
\bibliography{fano}

\end{document}